\documentclass[11pt]{amsart}

\usepackage[T1]{fontenc}
\usepackage[utf8]{inputenc} 
\usepackage{lmodern}
\usepackage{microtype}
\usepackage{mathpazo}

\usepackage[margin=3.5cm]{geometry}
\usepackage{setspace}
\usepackage{enumitem}
\usepackage{parskip}

\usepackage{amsmath}
\usepackage{amssymb}
\usepackage{amsthm}
\usepackage{mathtools}
\usepackage{bm}
\usepackage{mathrsfs}

\usepackage{graphicx}
\usepackage{caption}
\usepackage{subcaption}
\usepackage{booktabs}
\usepackage{array}
\usepackage{multirow}
\usepackage{tabularx}
\allowdisplaybreaks[4]

\usepackage[dvipsnames,table]{xcolor}
\usepackage{soul}

\usepackage{tikz}
\usetikzlibrary{
    arrows.meta,
    calc,
    decorations.markings,
    fit,
    matrix,
    patterns,
    positioning,
    shapes.geometric
}

\usepackage[
    colorlinks=true,
    linkcolor=MidnightBlue,
    citecolor=ForestGreen,
    urlcolor=BrickRed
]{hyperref}
\usepackage[nameinlink,capitalise,noabbrev]{cleveref}
\crefname{equation}{Eq.}{Eqs.}
\Crefname{equation}{Eq.}{Eqs.}

\newtheorem{theorem}{Theorem}[section]
\newtheorem{lemma}[theorem]{Lemma}

\crefname{fact}{Fact}{Facts}
\Crefname{fact}{Fact}{Facts}
\newtheorem{proposition}[theorem]{Proposition}
\newtheorem{corollary}[theorem]{Corollary}
\newtheorem{conjecture}[theorem]{Conjecture}

\theoremstyle{definition}

\theoremstyle{remark}

\theoremstyle{plain}

\theoremstyle{remark}

\newcommand{\R}{\mathbb{R}}

\newcommand{\E}{\mathbb{E}}
\newcommand{\Prob}{\mathbb{P}}
\newcommand{\Var}{\operatorname{Var}}

\newcommand{\down}{\mathbf{D}}
\newcommand{\up}{\mathbf{U}}
\newcommand{\downup}{\mathbf{P}}
\newcommand{\lift}{\mathbf{T}}
\newcommand{\gen}{\mathbf{L}}

\renewcommand{\deg}{\mathrm{deg}}
\newcommand{\res}{\mathrm{res}}

\title[Mixing Times of Switch Chains via HDX]{Mixing Times of Switch Chains\\via High-Dimensional Expansion}
\author[Sawyer Jack Robertson]{Sawyer Jack Robertson\textsuperscript{$\dagger$}}
\thanks{\textsuperscript{$\dagger$}Simons Institute for the Theory of Computing, University of California Berkeley}
\date{\today}

\subjclass[2020]{Primary \texttt{60J10}; Secondary \texttt{05C80, 05C81, 05E45}}

\begin{document}

\begin{abstract}
    The switch chain is a Markov chain defined on the set of labelled realizations of a given graphical degree sequence. At each step, a pair of vertex-disjoint edges is chosen at random and the process attempts to replace them with a uniformly chosen perfect matching of the same four vertices, rejecting any proposal that would create a multiple edge. The resulting process is reversible with respect to the uniform distribution on all realizations. We investigate the mixing time of this chain by viewing realizations as the facets of a simplicial complex and studying a variant of the original process called the \emph{simplicial switch chain}, which we analyze using tools from the theory of high-dimensional expansion. Our technical contributions include a proof that links of faces of sufficiently high codimension are strong spectral expanders and a comparison between the Dirichlet energies of large block updates and two-edge updates.

    Our main result is an $O(\Delta^{2}m\log m)$ bound on the mixing time of both simplicial and classical switch chains whenever $m\ge C\Delta^{8}$, where $m$ is the number of edges, $\Delta$ is the maximum prescribed degree, and $C>0$ is an absolute constant. For sequences on $n$ vertices with fixed maximum degree, this proves that the chain mixes in $O(n\log n)$ steps, resolving a longstanding conjecture of Cooper, Dyer, and Greenhill and extending its conclusion to irregular degree sequences.
\end{abstract}

\maketitle

\section{Introduction}\label{sec:introduction}

Let $n\ge 1$ be a positive integer and let $V$ denote a ground set of $n$ vertices. Let $\mathbf{d} = (d_{v})_{v\in V}$ be a sequence of positive integers which is \emph{graphical}, i.e., such that there exists at least one simple graph with vertex set $V$ and with vertex degree sequence $\mathbf{d}$. We call such a graph a \emph{realization} of $\mathbf{d}$. We write $m = \frac{1}{2}\sum_{v\in V}d_{v}$ to denote the number of edges in any realization of $\mathbf{d}$, and $\Delta=\max_{v\in V}d_{v}$ for the maximum prescribed degree. We assume throughout that $\mathbf{d}$ supplies at least two such realizations.

The subject of this paper is a natural process on the set of realizations called the \emph{switch chain}, defined as follows. From the current realization, the chain stays put with probability $1/2$. Otherwise, it chooses an unordered pair of vertex-disjoint edges uniformly at random and then chooses uniformly one of the three perfect matchings of their four endpoints. If the proposed matching contains an edge of the current realization, the chain stays put; otherwise, it replaces the selected edges by the proposed matching. This chain is irreducible and reversible with respect to the uniform distribution on realizations (these facts, respectively, follow from the Hakimi--Havel algorithm, see, e.g.~\cite{bereg2017transforming}, and the symmetry of the transition probabilities). One motivation for our results is the following conjecture of Cooper, Dyer, and Greenhill~\cite{cooper2007sampling}, whose work established polynomial mixing time bounds for regular degree sequences, following earlier work in the setting of bipartite switch chains by Kannan, Tetali, and Vempala~\cite{kannan1999simple}.

\begin{conjecture}[Cooper, Dyer, and Greenhill~{\cite{cooper2007sampling}}]\label{conj:cdg}
    For each integer $d\ge3$, there exists a constant $C>0$, depending only on $d$, such that the switch chain $\mathbf{S}$ on labelled $d$-regular graphs on $n$ vertices satisfies
        \begin{align*}
            t_{\mathrm{mix}}(\mathbf{S}) &\le Cn\log n.
        \end{align*}
\end{conjecture}

In~\cref{conj:cdg}, $t_{\mathrm{mix}}(\mathbf{S})$ denotes the total variation mixing time of the chain at distance $1/4$ (see~\cref{eq:mixing-time}). A simple coupon collector heuristic motivates the conjectured scale. Each proposed switch shuffles two edges, suggesting that roughly $m\log m$ steps are needed before every edge of the initial realization has had an opportunity to be updated. For fixed $d$, this is of order $n\log n$.

Polynomial mixing of the switch chain is known for regular graphs~\cite{cooper2007sampling,cooper2012corrigendum}, positive irregular degree sequences satisfying $3\le\Delta\le\sqrt{2m}/3$~\cite{greenhill2018switch}, and families satisfying general stability criteria~\cite{amanatidis2020rapid,gao2021mixing,erdos2022mixing}. These results cover a variety of sparse, dense, and heavy-tailed sequences, although the quantitative bounds can be large. For fixed maximum degree, Cooper, Dyer, and Greenhill give an $O_{\Delta}(n^{9}\log n)$ mixing bound for regular graphs, while Greenhill and Sfragara give an $O_{\Delta}(n^{10}\log n)$ bound for irregular degree sequences. The regular bipartite setting is perhaps more well understood: Tikhomirov and Youssef obtained sharp spectral estimates~\cite{tikhomirov2023sharp} and optimal $O_{\Delta}(n\log n)$ mixing for fixed degree~\cite{tikhomirov2024regularized}. Additionally, recent work of Fu, Qin, and Wang~\cite{fu2026spectral} establishes polynomial mixing for arbitrary bipartite degree sequences. We remark that OpenAI has very recently announced the analogous conclusion for all simple undirected degree sequences, a claim which is pending verification by the mathematical community~\cite{openai2026polynomial}.

In this paper we present improved bounds on the mixing time of the switch chain in the regime where $m\gtrsim\Delta^{8}$. This includes $\Delta\lesssim n^{1/8}$ for arbitrary positive graphical degree sequences and $\Delta\lesssim n^{1/7}$ for regular sequences. Our main result is the following theorem.

\begin{theorem}\label{thm:intro-mixing}
    There exists an absolute constant $C>0$ with the following property. If $m\ge C\Delta^{8}$, then
        \begin{align*}
            t_{\mathrm{mix}}(\mathbf{S}) &\le C\Delta^{2}m\log m.
        \end{align*}
\end{theorem}

\Cref{thm:intro-mixing} is proved in~\cref{sec:classical-switch-chain}. Since $n/2\le m\le\Delta n/2$, for each fixed $\Delta$ and all sufficiently large $n$ this gives a mixing time bound of a constant multiple of $n\log n$, which in turn proves~\cref{conj:cdg} and extends its conclusion to irregular degree sequences.

We prove~\cref{thm:intro-mixing} by first introducing a simplicial analogue of the switch chain. Concretely, a \emph{simplicial complex} $\mathcal{X}$ is a finite, nonempty family of finite sets closed under taking subsets. Its elements are called \emph{faces}, and its inclusion-maximal faces are called \emph{facets}. We write $\mathcal{X}(j)=\{\sigma\in\mathcal{X}:|\sigma|=j\}$ and call its elements \emph{$j$-faces}. The $1$-faces are called vertices, whose elements we identify with their singleton sets, and the $2$-faces are called edges. We say that $\mathcal{X}$ is \emph{pure} if each of its facets has the same cardinality. With some slight abuse of terminology, we adopt the convention that the complex is of \emph{dimension $m$} if each of its facets has $m$ elements.

For $1\le k\le m$, the \emph{down-$k$-up-$k$ walk} $\downup_{k}$ on the facets of a pure $m$-dimensional complex chooses $k$ elements of the current facet uniformly at random, deletes them, and then samples a new facet uniformly among those containing the retained face. The completions of a face $\sigma\in\mathcal{X}(j)$ with $0\le j\le m$ are described by its \emph{link}, denoted $\mathcal{X}^{(\sigma)}$. This is the pure $(m-j)$-dimensional complex whose facets are given by
    \begin{align*}
        \mathcal{X}^{(\sigma)}(m-j) &= \{\tau\setminus\sigma:\tau\in\mathcal{X}(m),\ \sigma\subseteq\tau\}.
    \end{align*}
If $\sigma\in\mathcal{X}(m-k)$, its link is pure of dimension $k$, and its facets are the sets of $k$ elements that complete $\sigma$ to a facet. For $k\ge2$, the \emph{link walk} $\downup^{(\sigma)}$ is the Markov chain corresponding to the simple random walk on the vertices of the link (see~\cref{eq:link-walk-kernel}).

For a subset of edges $\sigma\subseteq\binom{V}{2}$, write $\deg_{\sigma}(v)$ for the number of edges in $\sigma$ incident to $v$. The \emph{switch complex} $\mathcal{R}_{\mathbf{d}}$ is the pure $m$-dimensional simplicial complex with facets
    \begin{align*}
        \mathcal{R}_{\mathbf{d}}(m) &= \Bigl\{\sigma\subseteq\binom{V}{2} \,:\, \deg_{\sigma}(v)=d_{v}\text{ for each }v\in V \Bigr\},
    \end{align*}
i.e., with facets consisting of the realizations of $\mathbf{d}$. A vertex of $\mathcal{R}_{\mathbf{d}}$ is a graph edge that occurs in a realization, and a $j$-face is a set of $j$ graph edges that can be completed to a realization.

On $\mathcal{R}_{\mathbf{d}}$, we call $\downup_{2}$ the \emph{simplicial switch chain}. Unlike the classical switch chain, the two edges deleted at the beginning of each step need not be vertex-disjoint, and the next realization is chosen uniformly from among valid completions rather than by proposing a matching and rejecting invalid proposals. We illustrate an example of a switch complex and a step of this process on the set of cubic graphs on $8$ vertices in~\cref{fig:down-two-up-two}.

Down-up walks on simplicial complexes have become a central object of study in sampling and counting algorithms. This explosion of work owes much to the \emph{local-to-global method}, which, roughly speaking, works by combining lower bounds on the spectral gaps of link walks to obtain spectral gap bounds for walks on the whole complex, and hence estimates on their relaxation and mixing times. These ideas form the bedrock of the broader theory of \emph{high dimensional expansion}. The method can be implemented in several ways (see, e.g.,~\cite{alev2020improved}), but perhaps the best known is through \emph{trickle-down theorems}, notably in the work of Oppenheim~\cite{oppenheim2018local} and Kaufman and Oppenheim~\cite{kaufman2020high}. These tools assert that if the links of $(m-2)$-faces have sufficiently strong spectral expansion, and the graphs of the larger links are connected, then expansion estimates propagate to those larger links.

However, switch complexes do not generally satisfy the strong hypotheses needed for a trickle-down argument. For example, the links of faces of dimension $m-2$ are disjoint unions of one or more copies of $K_{2}$: once one completion edge is specified, the residual degrees determine the other. These links are in general disconnected, and disconnectedness can also persist in links of larger dimension. Moreover, even among links whose graphs are connected, it is not clear how to organize and describe their graph structure in a way that makes the spectral estimates required by trickle-down theorems accessible. We illustrate some examples of links and their associated graphs in~\cref{fig:realization-links-depth-preview}.

Our main technical result provides a method to circumvent these obstructions and hence to rescue the local-to-global method. We show that the links of the switch complexes exhibit what we call \emph{deep link expansion}: for faces $\sigma$ with sufficiently large codimension $k=m-|\sigma|$ (or ``depth''), the corresponding links exhibit strong spectral expansion, even though links of faces with smaller codimension may be disconnected. For a reversible Markov kernel $\mathbf{Q}$, we denote its second largest eigenvalue, counted with multiplicity, by $\lambda_{2}(\mathbf{Q})$.

\begin{theorem}[Deep link expansion]\label{thm:intro-link-eigenvalue-realization-complex}
    There exists an absolute constant $C>0$ with the following property. Let $k$ be an integer with $2\le k\le m$, and suppose that $k\ge C\Delta^{4}$. Let $\sigma\in\mathcal{R}_{\mathbf{d}}(m-k)$ be arbitrary. Then
        \begin{align*}
            \lambda_{2}(\downup^{(\sigma)}) &\le C\Delta^{2}k^{-3/2}.
        \end{align*}
\end{theorem}

\Cref{thm:intro-link-eigenvalue-realization-complex} is proved in~\cref{sec:deep-link-expansion}.

By using a local-to-global theorem of Alev and Rao~\cite{alev2025expanderizing} (see~\cref{thm:alev-rao}) in conjunction with~\cref{thm:intro-link-eigenvalue-realization-complex}, we are able to bound the spectral gap of the down-$k$-up-$k$ walk on $\mathcal{R}_{\mathbf{d}}$.

\begin{corollary}\label{cor:intro-higher-order-eigenvalue}
    There exists an absolute constant $C>0$ with the following property. Let $k$ be an integer with $2\le k\le m$, and suppose that $k\ge C\Delta^{4}$. Then the down-$k$-up-$k$ walk on $\mathcal{R}_{\mathbf{d}}(m)$ satisfies
        \begin{align*}
            \lambda_{2}(\downup_{k}) &\le 1-\frac{k}{Cm}.
        \end{align*}
\end{corollary}

\Cref{cor:intro-higher-order-eigenvalue} is proved in~\cref{sec:spectral-gaps}. Notably, our proof does not use trickle-down theorems.

We then prove the following comparison theorem for the Dirichlet energies of the down-$k$-up-$k$ walk and the down-$2$-up-$2$ walk, with a comparison constant linear in $k$.

\begin{theorem}\label{thm:intro-linear-comparison}
    There exist absolute constants $C,C'>0$ with the following property. Let $k$ be an integer satisfying $C\Delta^{8}\le k\le m$. Then, for every $f\in\ell^{2}(\mathcal{R}_{\mathbf{d}}(m),\pi_{m})$,
        \begin{align*}
            \mathcal{E}_{\downup_{k}}(f,f) &\le C'k\,\mathcal{E}_{\downup_{2}}(f,f).
        \end{align*}
\end{theorem}

\Cref{thm:intro-linear-comparison} is proved in~\cref{sec:spectral-gaps}. We then obtain the following theorem as an immediate consequence.

\begin{theorem}\label{thm:intro-two-edge-eigenvalue}
    There exists an absolute constant $C>0$ with the following property. If $m\ge C\Delta^{8}$, then
        \begin{align*}
            1-\frac{2}{m} &\le\lambda_{2}(\downup_{2})\le1-\frac{1}{Cm}.
        \end{align*}
\end{theorem}

\Cref{thm:intro-two-edge-eigenvalue} is proved in~\cref{sec:spectral-gaps}.

By then using a relative entropy estimate and a lemma of Salez~\cite{salez2024cutoff}, we obtain the following bound on the total variation mixing time of the simplicial switch chain.

\begin{corollary}\label{cor:intro-salez-sharp-switch-mixing}
    There exists an absolute constant $C>0$ with the following property. If $m\ge C\Delta^{8}$, then
        \begin{align*}
            t_{\mathrm{mix}}(\downup_{2},1/4) &\le C\Delta^{2}m\log(m).
        \end{align*}
\end{corollary}

\Cref{cor:intro-salez-sharp-switch-mixing} is proved in~\cref{sec:mixing-times-simplicial}.

Finally, in~\cref{sec:classical-switch-chain}, we explain how to convert the results concerning the relaxation and mixing times of the simplicial switch chain back to the setting of the classical switch chain. This proves~\cref{thm:intro-mixing}.

Our main results apply whenever $m\ge C\Delta^{8}$, allowing the maximum degree to grow with the number of edges. Improving this degree range remains an interesting question. A systematic approach to such comparisons could also be useful for analyzing other simplicial walks. Other directions for future work include investigating these techniques for simplicial complexes whose facets consist of $\mathbf{d}$-realizations satisfying additional constraints such as bipartiteness or connectivity, as well as complexes formed from other graph families like planar triangulations. More generally, it would be of interest to understand to what extent deep link expansion extends beyond switch complexes, and whether there are structural conditions that guarantee it.

The results and questions presented herein connect to an extensive literature concerning switch chains, with a history perhaps dating as far back as the work of Ryser~\cite{ryser1957combinatorial} in the 1950s. Rapid mixing has been established for graph classes and degree sequences of various families and types~\cite{miklos2013towards,erdos2018new,erdos2021half}, and further work has developed general criteria for rapid mixing and studied their relation to graphicality~\cite{amanatidis2020rapid,gao2021mixing,erdos2022mixing,erdos2025fully}. Related work has also considered alternative chains for bipartite graphs~\cite{carstens2018speeding}, switches that preserve connectivity~\cite{feder2006local}, graphs with additional degree constraints~\cite{erdos2015decomposition}, and regular directed graphs~\cite{greenhill2011polynomial}. More broadly, related questions have been studied for random walks on trees and matchings~\cite{diaconis2002random}, sampling hypergraphs with prescribed degrees~\cite{dyer2021sampling}, and graph processes generated by random edits~\cite{chung2025stochastically}.

This paper is organized as follows. In~\cref{sec:background}, we review some further background on Markov chains and simplicial complexes required for the remainder of the paper. In~\cref{sec:deep-link-expansion}, we prove the deep link expansion theorem. In~\cref{sec:spectral-gaps}, we prove the linear energy comparison and spectral gap bounds, and in~\cref{sec:mixing-times-simplicial}, we prove the entropy and mixing time bounds. In~\cref{sec:classical-switch-chain}, we convert the results concerning the simplicial switch chain back to the setting of the classical switch chain.

\clearpage

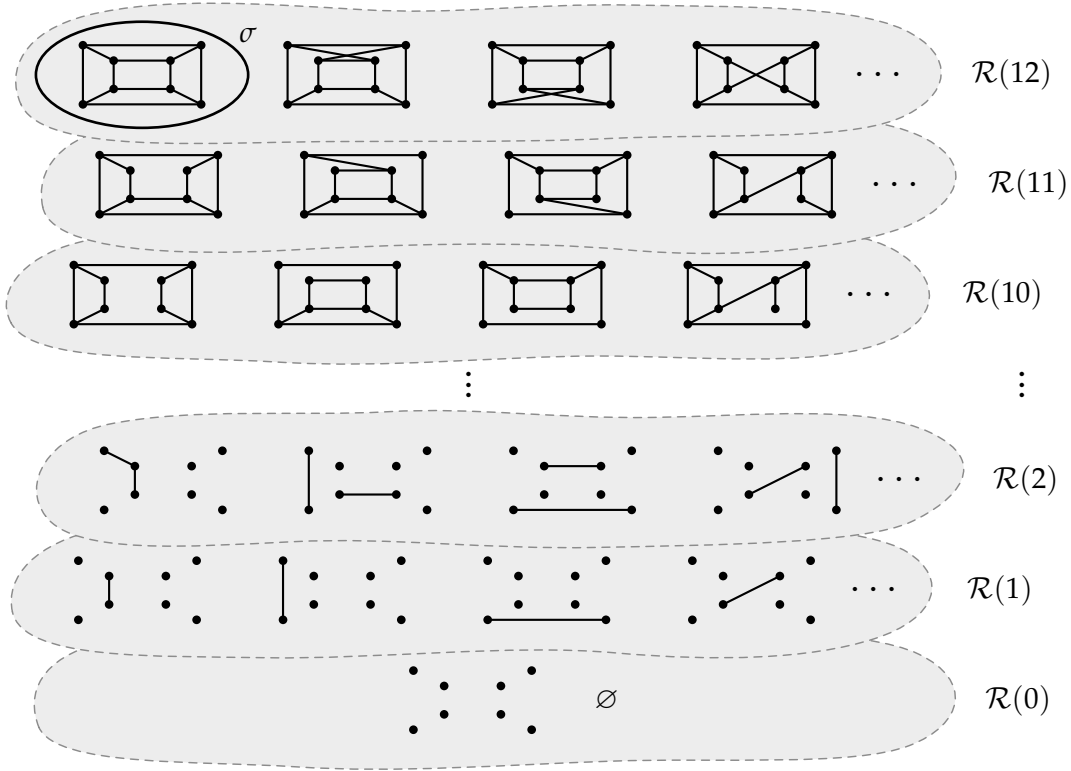
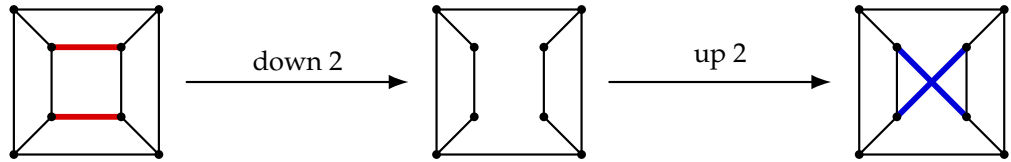
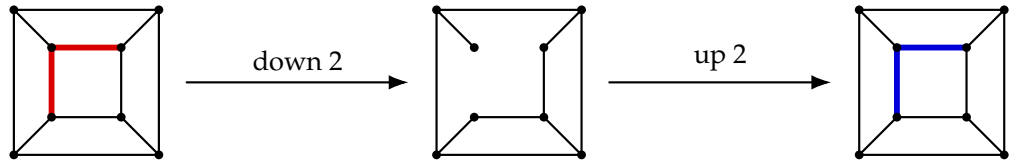
\begin{figure}[p]
\centering
\begingroup
\setlength{\parskip}{0pt}
\captionsetup{font=small,labelfont=bf,skip=5pt}
\captionsetup[subfigure]{font=small,labelfont=bf,skip=4pt,
  justification=raggedright,singlelinecheck=false}

\def\CubeEdges{1/2,1/4,1/5,2/3,2/6,3/4,3/7,4/8,5/6,5/8,6/7,7/8}
\def\SwitchedEdges{1/3,1/4,1/5,2/3,2/4,2/6,3/7,4/8,5/6,5/8,6/7,7/8}
\def\OverviewSix{1/2,1/4,1/5,2/3,2/6,3/4,3/8,4/7,5/6,5/8,6/7,7/8}
\def\OverviewFive{1/2,1/4,1/6,2/3,2/5,3/4,3/7,4/8,5/6,5/8,6/7,7/8}

\def\MoveSelected{1/2,3/4}
\def\MoveAdded{1/3,2/4}
\def\MoveRetained{1/4,1/5,2/3,2/6,3/7,4/8,5/6,5/8,6/7,7/8}
\def\HoldSelected{1/2,1/4}
\def\HoldAdded{1/2,1/4}
\def\HoldRetained{1/5,2/3,2/6,3/4,3/7,4/8,5/6,5/8,6/7,7/8}

\tikzset{
  rf ordinary/.style={draw=black,line width=.8pt},
  rf deleted/.style={draw=red!85!black,line width=2.15pt},
  rf added/.style={draw=blue!85!black,line width=2.15pt},
  rf arrow/.style={draw=black,line width=.8pt,
    -{Latex[length=2.6mm,width=1.9mm]}},
  rf arrow label/.style={font=\small,above=3pt,inner sep=1pt},
  rf plane/.style={cm={1,0,0,.50,(0,0)}},
  rf sheet/.style={fill=black!7,draw=black!45,line width=.55pt,dash pattern=on 3pt off 2.5pt},
  rf level/.style={font=\normalsize,anchor=west,inner sep=1pt},
}

\newcommand{\RFVertices}{%
  \coordinate (v5) at (-1,1);
  \coordinate (v6) at (1,1);
  \coordinate (v7) at (1,-1);
  \coordinate (v8) at (-1,-1);
  \coordinate (v1) at (-.48,.48);
  \coordinate (v2) at (.48,.48);
  \coordinate (v3) at (.48,-.48);
  \coordinate (v4) at (-.48,-.48);
}

\newcommand{\RFGraph}[4][]{%
  \begin{scope}[#1]
    \RFVertices
    \edef\RFList{#2}%
    \ifx\RFList\empty\else
      \foreach \u/\v in \RFList {\draw[rf ordinary] (v\u)--(v\v);}
    \fi
    \edef\RFList{#3}%
    \ifx\RFList\empty\else
      \foreach \u/\v in \RFList {\draw[rf deleted] (v\u)--(v\v);}
    \fi
    \edef\RFList{#4}%
    \ifx\RFList\empty\else
      \foreach \u/\v in \RFList {\draw[rf added] (v\u)--(v\v);}
    \fi
    \foreach \v in {1,...,8} {
      \node[circle,fill=black,inner sep=0pt,minimum size=3.25pt] at (v\v) {};
    }
  \end{scope}%
}

\newcommand{\RFOutline}[2]{%
  \expandafter\def\csname RFBlobPath#1\endcsname{#2}%
}
\RFOutline{0}{%
  (-.24,-.72)
  .. controls (-.42,.18) and (-.22,1.12) .. (.58,1.51)
  .. controls (1.62,2.02) and (2.85,1.84) .. (4.05,1.66)
  .. controls (5.25,1.48) and (6.11,1.86) .. (7.42,1.84)
  .. controls (8.80,1.82) and (10.66,1.90) .. (11.37,1.21)
  .. controls (12.08,.52) and (11.97,-.64) .. (11.24,-1.20)
  .. controls (10.51,-1.76) and (8.73,-1.51) .. (7.30,-1.61)
  .. controls (5.87,-1.71) and (4.47,-1.93) .. (3.16,-1.76)
  .. controls (1.85,-1.59) and (.00,-1.77) .. (-.24,-.72)
  -- cycle
}
\RFOutline{1}{%
  (-.35,-.57)
  .. controls (-.45,.35) and (-.05,1.27) .. (.83,1.62)
  .. controls (1.89,2.04) and (3.02,1.61) .. (4.35,1.60)
  .. controls (5.68,1.59) and (6.64,2.01) .. (8.05,1.85)
  .. controls (9.46,1.69) and (10.92,1.68) .. (11.46,1.00)
  .. controls (12.00,.32) and (11.86,-.74) .. (11.02,-1.28)
  .. controls (10.18,-1.82) and (8.96,-1.89) .. (7.59,-1.65)
  .. controls (6.22,-1.41) and (4.45,-1.83) .. (2.99,-1.78)
  .. controls (1.53,-1.73) and (-.21,-1.86) .. (-.35,-.57)
  -- cycle
}
\RFOutline{2}{%
  (-.25,-.95)
  .. controls (-.58,-.08) and (-.22,1.15) .. (.58,1.48)
  .. controls (1.68,1.93) and (3.10,1.68) .. (4.26,1.81)
  .. controls (5.42,1.94) and (6.45,1.62) .. (7.77,1.66)
  .. controls (9.09,1.70) and (10.63,2.02) .. (11.43,1.25)
  .. controls (12.23,.48) and (11.78,-.53) .. (11.16,-1.10)
  .. controls (10.54,-1.67) and (8.54,-1.71) .. (7.02,-1.77)
  .. controls (5.50,-1.83) and (4.35,-1.49) .. (3.00,-1.62)
  .. controls (1.65,-1.75) and (.14,-1.98) .. (-.25,-.95)
  -- cycle
}
\RFOutline{10}{%
  (-.31,-.74)
  .. controls (-.55,.17) and (-.05,1.14) .. (.70,1.43)
  .. controls (1.79,1.85) and (2.91,1.99) .. (4.22,1.85)
  .. controls (5.53,1.71) and (6.47,1.47) .. (7.91,1.65)
  .. controls (9.35,1.83) and (10.73,1.88) .. (11.35,1.17)
  .. controls (11.97,.46) and (12.02,-.71) .. (11.12,-1.32)
  .. controls (10.22,-1.93) and (8.77,-1.62) .. (7.44,-1.62)
  .. controls (6.11,-1.62) and (4.70,-1.96) .. (3.21,-1.78)
  .. controls (1.72,-1.60) and (-.02,-1.84) .. (-.31,-.74)
  -- cycle
}
\RFOutline{11}{%
  (-.23,-.59)
  .. controls (-.33,.37) and (-.02,1.28) .. (.88,1.60)
  .. controls (1.90,1.96) and (3.19,1.52) .. (4.44,1.62)
  .. controls (5.69,1.72) and (6.85,1.95) .. (8.24,1.85)
  .. controls (9.63,1.75) and (10.75,1.94) .. (11.43,1.20)
  .. controls (12.11,.46) and (11.82,-.59) .. (11.12,-1.13)
  .. controls (10.42,-1.67) and (8.92,-1.98) .. (7.48,-1.72)
  .. controls (6.04,-1.46) and (4.37,-1.58) .. (2.98,-1.73)
  .. controls (1.59,-1.88) and (-.11,-1.77) .. (-.23,-.59)
  -- cycle
}
\RFOutline{12}{%
  (-.30,-.77)
  .. controls (-.50,.14) and (-.26,1.15) .. (.58,1.57)
  .. controls (1.59,2.07) and (2.91,1.87) .. (4.19,1.72)
  .. controls (5.47,1.57) and (6.35,1.89) .. (7.75,1.84)
  .. controls (9.15,1.79) and (10.65,1.84) .. (11.35,1.19)
  .. controls (12.05,.54) and (11.99,-.59) .. (11.26,-1.17)
  .. controls (10.53,-1.75) and (8.86,-1.55) .. (7.47,-1.68)
  .. controls (6.08,-1.81) and (4.62,-1.72) .. (3.14,-1.78)
  .. controls (1.66,-1.84) and (-.04,-1.95) .. (-.30,-.77)
  -- cycle
}
\newcommand{\RFBlob}[1]{%
  \path[rf sheet] \csname RFBlobPath#1\endcsname;
}

\newcommand{\RFSheet}[4]{%
  \begin{scope}[shift={(#1,#2)}]
    \begin{scope}[rf plane]
      \RFBlob{#3}
      #4
    \end{scope}
    \node[rf level] at (12.20,0) {$\mathcal{R}(#3)$};
  \end{scope}%
}

\begin{subfigure}{\linewidth}
\centering
\resizebox{\linewidth}{!}{%
\begin{tikzpicture}[line cap=round,line join=round]
  \path[use as bounding box] (-.60,-.78) rectangle (13.95,10.30);

  \RFSheet{.12}{.20}{0}{%
    \RFGraph[shift={(5.48,0)},scale=.78]{}{}{}
    \node[font=\normalsize] at (7.25,0) {$\varnothing$};
  }

  \RFSheet{-.12}{1.65}{1}{%
    \RFGraph[shift={(1.30,0)},scale=.78]{1/4}{}{}
    \RFGraph[shift={(4.00,0)},scale=.78]{5/8}{}{}
    \RFGraph[shift={(6.70,0)},scale=.78]{7/8}{}{}
    \RFGraph[shift={(9.40,0)},scale=.78]{2/4}{}{}
    \node[font=\Large] at (11.05,0) {$\cdots$};
  }

  \RFSheet{.22}{3.10}{2}{%
    \RFGraph[shift={(1.30,0)},scale=.78]{1/4,1/5}{}{}
    \RFGraph[shift={(4.00,0)},scale=.78]{3/4,5/8}{}{}
    \RFGraph[shift={(6.70,0)},scale=.78]{1/2,7/8}{}{}
    \RFGraph[shift={(9.40,0)},scale=.78]{2/4,6/7}{}{}
    \node[font=\Large] at (11.05,0) {$\cdots$};
  }

  \node[font=\Large,inner sep=0pt] at (5.55,4.47) {$\vdots$};
  \node[font=\Large,inner sep=0pt] at (12.85,4.47) {$\vdots$};

  \RFSheet{-.18}{5.55}{10}{%
    \RFGraph[shift={(1.30,0)},scale=.78]{1/4,1/5,2/3,2/6,3/7,4/8,5/6,5/8,6/7,7/8}{}{}
    \RFGraph[shift={(4.00,0)},scale=.78]{1/2,1/4,2/3,3/4,3/7,4/8,5/6,5/8,6/7,7/8}{}{}
    \RFGraph[shift={(6.70,0)},scale=.78]{1/2,1/4,1/5,2/3,2/6,3/4,5/6,5/8,6/7,7/8}{}{}
    \RFGraph[shift={(9.40,0)},scale=.78]{1/4,1/5,2/3,2/4,2/6,4/8,5/6,5/8,6/7,7/8}{}{}
    \node[font=\Large] at (11.05,0) {$\cdots$};
  }

  \RFSheet{.16}{7.00}{11}{%
    \RFGraph[shift={(1.30,0)},scale=.78]{1/4,1/5,2/3,2/6,3/4,3/7,4/8,5/6,5/8,6/7,7/8}{}{}
    \RFGraph[shift={(4.00,0)},scale=.78]{1/2,1/4,2/3,2/5,3/4,3/7,4/8,5/6,5/8,6/7,7/8}{}{}
    \RFGraph[shift={(6.70,0)},scale=.78]{1/2,1/4,1/5,2/3,2/6,3/4,4/7,5/6,5/8,6/7,7/8}{}{}
    \RFGraph[shift={(9.40,0)},scale=.78]{1/4,1/5,2/3,2/4,2/6,3/7,4/8,5/6,5/8,6/7,7/8}{}{}
    \node[font=\Large] at (11.05,0) {$\cdots$};
  }

  \RFSheet{-.06}{8.45}{12}{%
    \RFGraph[shift={(1.30,0)},scale=.78]{1/2,1/4,1/5,2/3,2/6,3/4,3/7,4/8,5/6,5/8,6/7,7/8}{}{}
    \RFGraph[shift={(4.00,0)},scale=.78]{1/2,1/4,1/6,2/3,2/5,3/4,3/7,4/8,5/6,5/8,6/7,7/8}{}{}
    \RFGraph[shift={(6.70,0)},scale=.78]{1/2,1/4,1/5,2/3,2/6,3/4,3/8,4/7,5/6,5/8,6/7,7/8}{}{}
    \RFGraph[shift={(9.40,0)},scale=.78]{1/3,1/4,1/5,2/3,2/4,2/6,3/7,4/8,5/6,5/8,6/7,7/8}{}{}
    \draw[black,line width=1.0pt] (1.30,0) circle[radius=1.40cm];
    \node[font=\normalsize,inner sep=0pt] at (2.70,1.05) {$\sigma$};
    \node[font=\Large] at (11.05,0) {$\cdots$};
  }
\end{tikzpicture}%
}
\caption{An illustration of the simplicial complex $\mathcal{R}$, viewed as stacked ``levels.'' Each level contains shows a selection of $j$-faces for various values $j=12,11,\dotsc,0$. A facet $\sigma$ is circled for emphasis.}
\label{fig:down-two-up-two-facets}
\end{subfigure}

\par\vspace{8pt}

\begin{subfigure}{\linewidth}
\centering
\resizebox{\linewidth}{!}{%
\begin{tikzpicture}[line cap=round,line join=round]
  \path[use as bounding box] (-.25,-1.19) rectangle (14.05,1.30);
  \RFGraph[shift={(1.42,0)},scale=.94]{\MoveRetained}{\MoveSelected}{}
  \RFGraph[shift={(6.90,0)},scale=.94]{\MoveRetained}{}{}
  \RFGraph[shift={(12.38,0)},scale=.94]{\MoveRetained}{}{\MoveAdded}
  \draw[rf arrow] (2.72,0)--node[rf arrow label]{down $2$}(5.60,0);
  \draw[rf arrow] (8.20,0)--node[rf arrow label]{up $2$}(11.08,0);
\end{tikzpicture}%
}
\caption{An illustration of a non-holding step of the walk $\downup_2$ starting from $\sigma$: delete two edges (red), obtaining a $10$-face, then add two different edges (blue).}
\label{fig:down-two-up-two-nonholding}
\end{subfigure}

\par\vspace{8pt}

\begin{subfigure}{\linewidth}
\centering
\resizebox{\linewidth}{!}{%
\begin{tikzpicture}[line cap=round,line join=round]
  \path[use as bounding box] (-.25,-1.19) rectangle (14.05,1.30);
  \RFGraph[shift={(1.42,0)},scale=.94]{\HoldRetained}{\HoldSelected}{}
  \RFGraph[shift={(6.90,0)},scale=.94]{\HoldRetained}{}{}
  \RFGraph[shift={(12.38,0)},scale=.94]{\HoldRetained}{}{\HoldAdded}
  \draw[rf arrow] (2.72,0)--node[rf arrow label]{down $2$}(5.60,0);
  \draw[rf arrow] (8.20,0)--node[rf arrow label]{up $2$}(11.08,0);
\end{tikzpicture}%
}
\caption{An illustration of a holding step of the walk $\downup_2$ starting from $\sigma$: delete two edges (red), obtaining a $10$-face, then restore the same two edges (blue).}
\label{fig:down-two-up-two-holding}
\end{subfigure}

\caption{This figure contains an illustration of the simplicial complex $\mathcal{R}$ consisting of labelled cubic graphs on eight vertices in addition to a diagrammatic illustration of steps of the down-$2$-up-$2$ walk $\downup_2$ on the $19,355$ facets of $\mathcal{R}$.}
\label{fig:down-two-up-two}
\endgroup
\end{figure}

\clearpage
\input{figures/fig2-links.tex}
\clearpage

\section{Background and notation}\label{sec:background}

We begin by reviewing some general background on Markov chains.

Let $\mathcal{Y}$ be a finite set. For any positive weight function $h\in\R^{\mathcal{Y}}$, we write $\ell^{2}(\mathcal{Y},h)$ for the Hilbert space of real-valued functions on $\mathcal{Y}$ with inner product $\langle f,g\rangle_{h}=\sum_{x\in\mathcal{Y}}h(x)f(x)g(x)$ and induced norm $\|f\|_{h}^{2}=\langle f,f\rangle_{h}$. When $h$ is also a probability measure, these are $\langle f,g\rangle_{h}=\E_{x\sim h}(f(x)g(x))$ and $\|f\|_{h}^{2}=\E_{x\sim h}(f(x)^{2})$. We also use the notation $\|f\|_{2}^{2}$ for the ordinary squared $\ell^{2}$-norm of $f\in\R^{\mathcal{Y}}$.

Let $\mathbf{Q}$ be a Markov kernel on a finite set $\mathcal{Y}$ with at least two elements, and let $\pi$ be a stationary probability measure satisfying $\pi(y)>0$ for every $y\in\mathcal{Y}$. We regard $\mathbf{Q}$ as both a linear operator on functions, acting by left multiplication, and a transition kernel on probability measures, acting by right multiplication. For $f\in\ell^{2}(\mathcal{Y},\pi)$, the \emph{Dirichlet energy} of $f$ is given by
    \begin{align*}
        \mathcal{E}_{\mathbf{Q}}(f,f) &= \frac{1}{2}\sum_{x,y\in\mathcal{Y}}\pi(x)\mathbf{Q}(x,y)(f(x)-f(y))^{2} = \langle f,(\mathbf{I}-\mathbf{Q})f\rangle_{\pi},
    \end{align*}
where $\mathbf{I}$ is the identity operator. For any probability measure $\mu$ on $\mathcal{Y}$ and $f\in\R^{\mathcal{Y}}$, we define the \emph{variance} of $f$ according to
    \begin{align*}
        \Var_{x\sim\mu}(f(x)) &= \sum_{x\in\mathcal{Y}}\mu(x)(f(x)-\E_{y\sim\mu}(f(y)))^{2}.
    \end{align*}
When $\mathbf{Q}$ is reversible with respect to $\pi$, it is self-adjoint with respect to $\langle\cdot,\cdot\rangle_{\pi}$, and hence admits an orthonormal basis of real eigenfunctions in $\ell^{2}(\mathcal{Y},\pi)$. We denote its eigenvalues, counted with multiplicity, by
    \begin{align*}
        1 &=\lambda_{1}(\mathbf{Q})\ge\lambda_{2}(\mathbf{Q})\ge\cdots\ge\lambda_{|\mathcal{Y}|}(\mathbf{Q})\ge-1.
    \end{align*}
For a reversible kernel $\mathbf{Q}$, we call $1-\lambda_{2}(\mathbf{Q})$ its \emph{spectral gap}. We will use the following two variational characterizations of $1-\lambda_{2}(\mathbf{Q})$. First,
    \begin{align*}
        1-\lambda_{2}(\mathbf{Q}) &= \inf_{\substack{f\in\ell^{2}(\mathcal{Y},\pi)\\ \Var_{x\sim\pi}(f(x))>0}} \frac{\mathcal{E}_{\mathbf{Q}}(f,f)}{\Var_{x\sim\pi}(f(x))}.
    \end{align*}
Second,
    \begin{align}\label{eq:spectral-gap-variational-characterization-2}
        1-\lambda_{2}(\mathbf{Q}) &= \inf_{\substack{f\in\ell^{2}(\mathcal{Y},\pi)\\ \E_{x\sim\pi}(f(x))=0,\ f\ne0}} \frac{\mathcal{E}_{\mathbf{Q}}(f,f)}{\|f\|_{\pi}^{2}}.
    \end{align}
For additional background on the various connections between spectral theory and Markov chain mixing times, we refer the reader to, e.g.,~\cite{levin2017markov}. 

For probability measures $\mu,\nu$ on $\mathcal{Y}$, we define their \emph{relative entropy} $D(\mu\,\Vert\,\nu)$ according to
    \begin{align*}
        D(\mu\,\Vert\,\nu) &=\sum_{\substack{y\in\mathcal{Y}\\\mu(y)>0}}\mu(y)\log\frac{\mu(y)}{\nu(y)}.
    \end{align*}
In the preceding equation, we set $D(\mu\,\Vert\,\nu)=+\infty$ if there exists $y\in\mathcal{Y}$ such that $\nu(y) = 0$ and $\mu(y) > 0$. Note that $D(\mu\,\Vert\,\nu) = 0$ if and only if $\mu = \nu$.

We will make use of the following two facts regarding relative entropy. The first is known as the \emph{weak triangle inequality} for relative entropy; see also~{\cite[Cor.~A.3]{altschuler2026shifted}}.

\begin{lemma}\label{lem:weak-triangle-inequality}
    Let $\mu,\mu',\nu$ be probability measures on $\mathcal{Y}$. Suppose that $\nu$ is strictly positive on all of $\mathcal{Y}$, and moreover that $\mu(y)>0$ implies $\mu'(y)>0$ for each $y\in\mathcal{Y}$. Then
        \begin{align*}
            D(\mu\,\Vert\,\nu) &\le 2 D(\mu\,\Vert\,\mu') + \log\Bigl(1+\Bigl\|\frac{\mu'}{\nu}-\mathbf{1}\Bigr\|_{\nu}^{2} \Bigr).
        \end{align*}
\end{lemma}

The proof of~\cref{lem:weak-triangle-inequality} is straightforward and uses Jensen's inequality; we do not include it here. The second fact is related to the so-called chain rule for relative entropy (see, e.g.,~\cite[Thm.~2.5.3]{cover2006elements}). 

\begin{lemma}\label{lem:entropy-chain-rule}
    Let $\mathbf{Q},\mathbf{Q}'$ be Markov kernels on $\mathcal{Y}$. Let $\mu,\nu$ be probability measures on $\mathcal{Y}$. Then
        \begin{align*}
            D(\mu \mathbf{Q}\,\Vert\,\nu\mathbf{Q}') &\le D(\mu\,\Vert\,\nu) + \sum_{\substack{x\in\mathcal{Y} \\ \mu(x) >0 }}\mu(x) D(\mathbf{Q}(x,\cdot)\,\Vert\,\mathbf{Q}'(x,\cdot)),
        \end{align*} 
    where the right-hand side is interpreted as $+\infty$ if $D(\mathbf{Q}(x,\cdot)\,\Vert\,\mathbf{Q}'(x,\cdot))=+\infty$ for some $x\in\mathcal{Y}$ such that $\mu(x)>0$.
\end{lemma}

The proof of~\cref{lem:entropy-chain-rule} is straightforward and uses the log-sum inequality; we do not include it here.

For probability measures $\mu,\nu$ on $\mathcal{Y}$, define their total variation distance $\|\mu-\nu\|_{\mathrm{TV}}$ by
    \begin{align*}
        \|\mu-\nu\|_{\mathrm{TV}} &= \frac{1}{2}\sum_{x\in\mathcal{Y}}|\mu(x)-\nu(x)| = \max_{A\subseteq\mathcal{Y}}|\mu(A)-\nu(A)|.
    \end{align*}
For $0<\varepsilon<1$, the \emph{total variation mixing time} $t_{\mathrm{mix}}(\mathbf{Q},\varepsilon)$ of $\mathbf{Q}$ is given by
    \begin{align*}
        t_{\mathrm{mix}}(\mathbf{Q},\varepsilon) &= \min\Bigl\{s\in\mathbb{Z}_{\ge0}: \max_{x\in\mathcal{Y}} \|\mathbf{Q}^{s}(x,\cdot)-\pi\|_{\mathrm{TV}} \le\varepsilon\Bigr\}.
    \end{align*}
If this set is empty, we set $t_{\mathrm{mix}}(\mathbf{Q},\varepsilon)=+\infty$. We write
    \begin{align}\label{eq:mixing-time}
        t_{\mathrm{mix}}(\mathbf{Q}) &= t_{\mathrm{mix}}(\mathbf{Q},\varepsilon=1/4).
    \end{align}

Next let $\mathcal{X}$ be a pure $m$-dimensional simplicial complex, with $m\ge1$. For $0\le j\le m$, we equip $\mathcal{X}(j)$ with a probability measure $\pi_{j}$. We start by letting $\pi_{m}$ denote the uniform probability measure on the facets. Then, for $j=m-1,\dotsc,0$, we write
    \begin{align}\label{eq:induced-probability-measure}
        \pi_{j}(\sigma) &= \frac{1}{j+1} \sum_{\substack{\tau\in\mathcal{X}(j+1)\\ \sigma\subseteq\tau}} \pi_{j+1}(\tau)
    \end{align}
for each $\sigma\in\mathcal{X}(j)$. In other words, after iterating~\cref{eq:induced-probability-measure}, $\pi_{j}$ is the distribution obtained by choosing a facet uniformly at random and then choosing one of its $j$-faces uniformly at random.

Next fix $1\le j\le m$. We define the $j$-th \emph{down operator} $\down_{j}:\ell^{2}(\mathcal{X}(j),\pi_{j})\to\ell^{2}(\mathcal{X}(j-1),\pi_{j-1})$ as follows. Let $f\in\ell^{2}(\mathcal{X}(j),\pi_{j})$ and $\sigma\in\mathcal{X}(j-1)$. We write
    \begin{align*}
        (\down_{j} f)(\sigma) &= \frac{1}{j}\sum_{\substack{\tau\in\mathcal{X}(j)\\ \sigma\subseteq\tau}} \frac{\pi_{j}(\tau)}{\pi_{j-1}(\sigma)}f(\tau).
    \end{align*}

Fix $0\le j\le m-1$. We define the $j$-th \emph{up operator} $\up_{j}:\ell^{2}(\mathcal{X}(j),\pi_{j})\to\ell^{2}(\mathcal{X}(j+1),\pi_{j+1})$ as follows. Let $f\in\ell^{2}(\mathcal{X}(j),\pi_{j})$ and $\sigma\in\mathcal{X}(j+1)$. We write
    \begin{align*}
        (\up_{j} f)(\sigma) &= \frac{1}{j+1} \sum_{\substack{\tau\in\mathcal{X}(j)\\ \tau\subseteq\sigma}} f(\tau).
    \end{align*}

The following proposition establishes the adjoint relationship between the up and down operators. Its proof is routine and is not presented here (see, e.g.,~\cite[Prop.~2.11]{oppenheim2018local} or~\cite[Lem.~3.6]{kaufman2020high}).

\begin{proposition}\label{prop:up-down-adjoint}
    Let $j$ be an integer with $0\le j\le m-1$, and let $f\in\ell^{2}(\mathcal{X}(j),\pi_{j})$ and $g\in\ell^{2}(\mathcal{X}(j+1),\pi_{j+1})$. Then the operators $\up_{j}$ and $\down_{j+1}$ are adjoint with respect to the inner products induced by $\pi_{j}$ and $\pi_{j+1}$; specifically,
        \begin{align*}
            \langle \up_{j} f,g\rangle_{\pi_{j+1}} &= \langle f,\down_{j+1}g\rangle_{\pi_{j}}.
        \end{align*}
\end{proposition}

For $1\le k\le m$, the down-$k$-up-$k$ walk introduced in~\cref{sec:introduction} has the operator representation
    \begin{align*}
        \downup_{k} &= \up_{m-1}\up_{m-2}\cdots\up_{m-k}\, \down_{m-k+1}\down_{m-k+2}\cdots\down_{m}.
    \end{align*}
This defines a Markov kernel on $\mathcal{X}(m)$ which, by~\cref{prop:up-down-adjoint}, is reversible with respect to the uniform measure $\pi_{m}$ and has nonnegative eigenvalues (see, e.g.,~\cite{alev2020improved} for further background). 

For $\sigma\in\mathcal{X}(j)$ with $0\le j\le m$, let $\pi^{\sigma}_{m-j}$ denote the uniform probability measure on the facets of the link $\mathcal{X}^{(\sigma)}$. We extend this measure to the other face levels via~\cref{eq:induced-probability-measure}, writing $\pi^{\sigma}_{\ell}$ for the measure on $\mathcal{X}^{(\sigma)}(\ell)$.

For $\sigma\in\mathcal{X}(j)$ with $0\le j\le m-2$, the link walk $\downup^{(\sigma)}$ introduced in~\cref{sec:introduction} has kernel
    \begin{align}\label{eq:link-walk-kernel}
        \downup^{(\sigma)}(\tau,\rho) &=\begin{cases}
            \frac{\pi_{2}^{\sigma}(\{\tau,\rho\})}{2\pi_{1}^{\sigma}(\tau)} &\text{if } \{\tau,\rho\}\in\mathcal{X}^{(\sigma)}(2),\\
            0 &\text{otherwise},
        \end{cases}
    \end{align} 
for each $\tau,\rho\in\mathcal{X}^{(\sigma)}(1)$. This walk is reversible with respect to the measure $\pi^{\sigma}_{1}$, but it need not be irreducible in general.

We will use the following theorem, due to Alev and Rao, which provides a local-to-global-type bound on the spectral gap of $\downup_{k}$. We have simplified its statement somewhat to apply it directly to the present setting.

\begin{theorem}[Alev--Rao~{\cite[Cor.~4.2]{alev2025expanderizing}}]\label{thm:alev-rao}
    Let $\mathcal{X}$ be a pure $m$-dimensional simplicial complex with at least two facets. For $0\le j\le m-2$, define
        \begin{align*}
            \gamma_{j} &= \max_{\sigma\in\mathcal{X}(j)} \lambda_{2}(\downup^{(\sigma)}).
        \end{align*}
    Let $k$ be an integer with $1\le k\le m$. Then the down-$k$-up-$k$ walk on $\mathcal{X}(m)$ satisfies
        \begin{align}\label{eq:local-to-global-eigenvalue}
            1-\lambda_{2}(\downup_{k}) &\ge \frac{k}{m} \prod_{j=0}^{m-k-1}(1-\gamma_{j}).
        \end{align}
\end{theorem}

We remark that in~\cref{eq:local-to-global-eigenvalue}, if $k=m$, then the product is empty and equals $1$; in this case, the walk $\downup_{k}$ samples directly from $\pi_{m}$ and has second eigenvalue $0$.

\section{Deep link expansion}\label{sec:deep-link-expansion}

In this section we study the links of the switch complex $\mathcal{R}_{\mathbf{d}}$, defined in~\cref{sec:introduction}. Our main goal is to prove the following theorem, which relates their spectral expansion to the codimension of the corresponding faces. We are able to guarantee that, beyond a cutoff codimension which is polynomial in $\Delta$, the link walks $\downup^{(\sigma)}$ become strong spectral expanders.

\begin{theorem}[Deep link expansion]\label{thm:link-eigenvalue-switch-complex}
    Let $2\le k\le m$ be an integer satisfying
        \begin{align*}
            k &\ge \max\{10\Delta^{2},\frac{1}{2}\Delta^{4}\}.
        \end{align*}
    Let $\sigma\in\mathcal{R}_{\mathbf{d}}(m-k)$. Then
        \begin{align}\label{eq:link-eigenvalue-switch-complex}
            \lambda_{2}(\downup^{(\sigma)}) &\le \frac{200\,\Delta^{2}}{(k-1)\sqrt{k}}.
        \end{align}
\end{theorem}

The proof of~\cref{thm:link-eigenvalue-switch-complex} exploits the connection between the expansion of the link walk and the correlations among edges of a random completion. From a current vertex $\{u,v\}$ of the link, the walk chooses a uniform completion containing $\{u,v\}$ and then selects uniformly from its other $k-1$ edges. Strong expansion can therefore be established by controlling the tendency of random completions to concentrate their edges in particular regions of the link. At the variational level, we do this by bounding $\Var_{\eta\sim\pi^{\sigma}_{k}}\Bigl(\sum_{\{u,v\}\in\eta}f(\{u,v\})\Bigr)$ relative to $k\Var_{\{u,v\}\sim\pi_{1}^{\sigma}}(f(\{u,v\}))$, uniformly over centered functions $f$ on the vertices of the link. The following lemma provides the required bound.

\begin{lemma}\label{lem:covariance}
    Let $2\le k\le m$ be an integer satisfying
        \begin{align*}
            k &\ge \max \{10 \Delta^{2}, \frac{1}{2}\Delta^{4}\}.
        \end{align*}
    Let $\sigma\in\mathcal{R}_{\mathbf{d}}(m-k)$ and let $f\in\ell^{2}(\mathcal{R}_{\mathbf{d}}^{(\sigma)}(1),\pi_{1}^{\sigma})$. Then
        \begin{align}\label{eq:covariance-bound}
            \Var_{\eta\sim\pi^{\sigma}_{k}} \Bigl(\sum_{\{u,v\}\in \eta}f(\{u,v\})\Bigr) &\le (1+\frac{200\,\Delta^{2}}{\sqrt{k}}) k\,\Var_{\{u,v\}\sim\pi^{\sigma}_{1}}(f(\{u,v\})).
        \end{align}
\end{lemma}

To make the preceding remarks a bit more concrete, we first show how~\cref{thm:link-eigenvalue-switch-complex} is proved assuming the validity of~\cref{lem:covariance}. 

\begin{proof}[Proof of~\cref{thm:link-eigenvalue-switch-complex} assuming the validity of~\cref{lem:covariance}]
    To prove~\cref{thm:link-eigenvalue-switch-complex}, we use the variational characterization of $1-\lambda_{2}(\downup^{(\sigma)})$ in~\cref{eq:spectral-gap-variational-characterization-2}. Let $f\in\ell^{2}(\mathcal{R}_{\mathbf{d}}^{(\sigma)}(1),\pi_{1}^{\sigma})$ be a nonzero function satisfying $\E_{\{u,v\}\sim\pi_{1}^{\sigma}}(f(\{u,v\}))=0$. It follows that
        \begin{align*}
            \E_{\eta\sim\pi^{\sigma}_{k}}\Bigl(\sum_{\{u,v\}\in \eta}f(\{u,v\})\Bigr) &= k \E_{\{u,v\}\sim\pi_{1}^{\sigma}}(f(\{u,v\})) = 0.
        \end{align*}
    For $\{\{u,v\},\{w,z\}\}\in\mathcal{R}_{\mathbf{d}}^{(\sigma)}(2)$, the definition of $\pi_{2}^{\sigma}$ gives
        \begin{align*}
            \pi_{2}^{\sigma}(\{\{u,v\},\{w,z\}\}) &= \frac{1}{\binom{k}{2}} \sum_{\substack{\eta\in\mathcal{R}_{\mathbf{d}}^{(\sigma)}(k)\\ \{\{u,v\},\{w,z\}\}\subseteq \eta}}\pi_{k}^{\sigma}(\eta)\\
            &= \frac{1}{\binom{k}{2}} \Prob_{\eta\sim\pi_{k}^{\sigma}}(\{u,v\},\{w,z\}\in \eta).
        \end{align*}
    Thus, by~\cref{eq:link-walk-kernel},
        \begin{align*}
            \Prob_{\eta\sim\pi_{k}^{\sigma}}(\{u,v\},\{w,z\}\in \eta) &= \binom{k}{2}\pi_{2}^{\sigma}(\{\{u,v\},\{w,z\}\})\\
            &= k(k-1)\pi_{1}^{\sigma}(\{u,v\})\downup^{(\sigma)}(\{u,v\},\{w,z\}).
        \end{align*}
    The preceding equation still holds if $\{u,v\}\ne\{w,z\}$ and $\{\{u,v\},\{w,z\}\}\notin\mathcal{R}_{\mathbf{d}}^{(\sigma)}(2)$, since both sides are zero in this case. It follows from the preceding facts that
        \begin{align*}
            &\Var_{\eta\sim\pi_{k}^{\sigma}}\Bigl(\sum_{\{u,v\}\in\eta}f(\{u,v\})\Bigr)\\
            &\quad=\E_{\eta\sim\pi_{k}^{\sigma}}\Bigl(\Bigl(\sum_{\{u,v\}\in\eta}f(\{u,v\})\Bigr)^{2}\Bigr)\\
            &\quad=\sum_{\{u,v\}\in\mathcal{R}_{\mathbf{d}}^{(\sigma)}(1)}\Prob_{\eta\sim\pi_{k}^{\sigma}}(\{u,v\}\in\eta)f(\{u,v\})^{2}\\
            &\qquad+\sum_{\{u,v\}\in\mathcal{R}_{\mathbf{d}}^{(\sigma)}(1)} \sum_{\substack{\{w,z\}\in\mathcal{R}_{\mathbf{d}}^{(\sigma)}(1)\\ \{w,z\}\ne\{u,v\}}} \Prob_{\eta\sim\pi_{k}^{\sigma}}(\{u,v\},\{w,z\}\in\eta)\,f(\{u,v\})f(\{w,z\})\\
            &\quad= k\|f\|_{\pi_{1}^{\sigma}}^{2}+k(k-1)\sum_{\{u,v\}\in\mathcal{R}_{\mathbf{d}}^{(\sigma)}(1)}\pi_{1}^{\sigma}(\{u,v\})f(\{u,v\})(\downup^{(\sigma)}f)(\{u,v\}) \\
            &\quad=k\|f\|_{\pi_{1}^{\sigma}}^{2}+k(k-1)\langle f,\downup^{(\sigma)}f\rangle_{\pi_{1}^{\sigma}}.
        \end{align*}
    The second to last equality also uses $\downup^{(\sigma)}(\{u,v\},\{u,v\})=0$. By the assumption on $k$ and~\cref{lem:covariance}, it follows that
        \begin{align*}
            k\|f\|_{\pi_{1}^{\sigma}}^{2} +k(k-1)\langle f,\downup^{(\sigma)}f\rangle_{\pi_{1}^{\sigma}} &\le (1+\frac{200\,\Delta^{2}}{\sqrt{k}}) k\|f\|_{\pi_{1}^{\sigma}}^{2}.
        \end{align*}
    Subtracting $k\|f\|_{\pi_{1}^{\sigma}}^{2}$ and dividing by $k(k-1)\|f\|_{\pi_{1}^{\sigma}}^{2}$ gives
        \begin{align*}
            \frac{\langle f,\downup^{(\sigma)}f\rangle_{\pi_{1}^{\sigma}}} {\|f\|_{\pi_{1}^{\sigma}}^{2}} &\le \frac{200\,\Delta^{2}}{(k-1)\sqrt{k}}.
        \end{align*}
    Taking the supremum over all such $f$ therefore proves~\cref{eq:link-eigenvalue-switch-complex}. 
\end{proof}

The remainder of this section is therefore dedicated to the proof of~\cref{lem:covariance}, which requires several steps. We begin by representing completions of $\sigma$ in the form of ``valid pairings'' of a fixed set of ``stubs,'' a construction which can be considered a specialized or local form of the configuration model. Working with all pairings preserves the degree constraints while permitting violations of simplicity, allowing us to perform explicit calculations on generic pairings and then control the effect of the restrictions needed to obtain simple completions of $\sigma$. The configuration model, more generally, is well-known; see~\cite[Ch.~7]{vanderhofstad2017random} or~\cite{wormald1999models}. See also~\cite{mckay2010subgraphs,tikhomirov2023sharp} for instances of similar local constructions.

For the remainder of this section, fix an integer $2\le k\le m$ and $\sigma\in\mathcal{R}_{\mathbf{d}}(m-k)$. For each $v\in V$, define the \emph{residual degree}
    \begin{align*}
        \res_{\sigma}(v) &= d_{v}-\deg_{\sigma}(v),
    \end{align*}
i.e., $\res_{\sigma}(v)$ is the number of edges incident to $v$ that still need to be added to realize $\mathbf{d}$ from the partial realization $\sigma$. Since $\sigma$ has $m-k$ edges, we have that
    \begin{align*}
        \sum_{v\in V} \res_{\sigma}(v) &= 2m-2|\sigma|=2k,
    \end{align*}
and moreover that
    \begin{align*}
        \res_{\sigma}(v) + \deg_{\sigma}(v) &= d_{v}\le \Delta
    \end{align*}
for each $v\in V$. We equip each $v\in V$ with $\res_{\sigma}(v)$ labeled edge ends, called \emph{stubs}. Formally, the set $\mathcal{S}$ of stubs is given by
    \begin{align*}
        \mathcal{S} &= \{(v,a)\,:\, v\in V,\;\res_{\sigma}(v) >0,\; 1\le a\le \res_{\sigma}(v)\}.
    \end{align*}
The index $a$ associated to a stub $(v,a)$ is not particularly important in the calculations to follow, so we use the shorthand $\vec{v}$ to denote a stub of the form $(v,\cdot)$. A \emph{pairing} $X$ is a partition of $\mathcal{S}$ into unordered pairs. $X$ is called \emph{valid} if, for each pair of stubs $\{\vec{u},\vec{v}\}\in X$, the vertices $u$ and $v$ are distinct, the edge $\{u,v\}$ is not already present in $\sigma$, and if $\{\vec{w},\vec{z}\}\in X$ is distinct from $\{\vec{u},\vec{v}\}$, then $\{w,z\}\ne\{u,v\}$.

Let $\mathcal{P}$ denote the set of all pairings of $\mathcal{S}$, and let $\mathcal{Q}$ denote the set of valid pairings so that $\mathcal{Q}\subseteq\mathcal{P}$. Observe that for each $X\in\mathcal{Q}$, the corresponding set of edges
    \begin{align*}
        \Phi(X) &:= \{\{u,v\}\,:\,\{\vec{u},\vec{v}\}\in X\}
    \end{align*}
defines a facet of $\mathcal{R}_{\mathbf{d}}^{(\sigma)}$. If $\Phi(X)=\eta$, we say that $X$ \emph{projects} to $\eta$. We denote by $\pi_{\mathcal{Q}}$ the uniform law on $\mathcal{Q}$. Note that $\mathcal{Q},\mathcal{P}$ depend implicitly on $\sigma$, but we suppress this dependence in the notation. 

For each $f\in\ell^{2}(\mathcal{R}_{\mathbf{d}}^{(\sigma)}(1),\pi_{1}^{\sigma})$, we define a corresponding function $\widehat{f}\in\R^{\binom{\mathcal{S}}{2}}$ on pairs of stubs according to
    \begin{align*}
        \widehat{f}(\{ \vec{v}, \vec{u}\}) &=\begin{cases}
            f(\{v,u\}), &\text{if }\{v,u\}\text{ is a vertex of }\mathcal{R}_{\mathbf{d}}^{(\sigma)},\\
            0, &\text{otherwise}.
        \end{cases}
    \end{align*}
Define the linear map $\lift :\ell^{2}(\mathcal{R}_{\mathbf{d}}^{(\sigma)}(1),\pi_{1}^{\sigma}) \to \mathbb{R}^{\mathcal{P}}$ by writing
    \begin{align*}
        (\lift f)(X) &= \sum_{\{\vec{v},\vec{u}\}\in X} \widehat{f}(\{ \vec{v}, \vec{u}\})
    \end{align*}
for each $X\in \mathcal{P}$. The map $\lift$ is called the \emph{lift operator}.

Let $\alpha = \{\vec{v}, \vec{u}\}\in\binom{\mathcal{S}}{2}$ be a pair of stubs. Let $p_{\alpha}$ denote the marginal probability that $\alpha$ is included in a random pairing $X$ drawn from $\pi_{\mathcal{Q}} $, i.e.,
    \begin{align*}
        p_{\alpha} &= \Prob_{X\sim\pi_{\mathcal{Q}} }(\alpha\in X).
    \end{align*} 
The following lemma establishes a basic calculus for the measures $\pi_{\mathcal{Q}} $ and $\pi_{k}^{\sigma}$ and the lift operator.

\begin{lemma}\label{lem:uniform-projection-lift}
    The following facts hold.
        \begin{enumerate}[label=(\roman*)]
            \item Each $\eta'\in\mathcal{R}_{\mathbf{d}}^{(\sigma)}(k)$ has exactly $\prod_{v\in V}\res_{\sigma}(v)!$ preimages in $\mathcal{Q}$ satisfying $\Phi(X) = \eta'$. Consequently $X\sim\pi_{\mathcal{Q}} $ implies $\Phi(X)\sim\pi_{k}^{\sigma}$.
        \end{enumerate}
    Furthermore, let $f\in\ell^{2}(\mathcal{R}_{\mathbf{d}}^{(\sigma)}(1),\pi_{1}^{\sigma})$. Then
        \begin{enumerate}[label=(\roman*)]
            \setcounter{enumi}{1}
            \item $\E_{X\sim\pi_{\mathcal{Q}} }((\lift f)(X)) =k\E_{\{u,v\}\sim\pi_{1}^{\sigma}}(f(\{u,v\}))$,
            \item $\Var_{X\sim\pi_{\mathcal{Q}} }((\lift f)(X)) =\Var_{\eta\sim\pi_{k}^{\sigma}}(\sum_{\{u,v\}\in \eta}f(\{u,v\}))$,
            \item and $\sum_{\alpha\in\binom{\mathcal{S}}{2}}p_{\alpha}\widehat{f}(\alpha)^{2} =k\|f\|_{\pi_{1}^{\sigma}}^{2}$.
        \end{enumerate}
\end{lemma}

\begin{proof}
    First we prove claim \emph{(i)}. Let $\eta'\in\mathcal{R}_{\mathbf{d}}^{(\sigma)}(k)$ be fixed. Since $\eta'$ is a completion of $\sigma$, each $v\in V$ has exactly $\res_{\sigma}(v)$ incident edges in the graph $\eta'$. Assign these edges bijectively to the labeled stubs $(v,1),\dotsc,(v,\res_{\sigma}(v))$. There are $\res_{\sigma}(v)!$ such assignments at $v$, and the assignments at different vertices can be chosen independently. After choosing these assignments, each edge $\{u,v\}\in \eta'$ has exactly one stub at $u$ and one stub at $v$. By pairing these stubs we obtain a valid pairing which projects to $\eta'$.
 
    Conversely, suppose $X\in\mathcal{Q}$ is a valid pairing which projects to $\eta'$. Then each edge $\{u,v\}\in \eta'$ is the projection of exactly one pair $\{\vec{u}, \vec{v}\}$ in $X$. This identifies the stub assigned to each endpoint of this edge, and therefore recovers each assignment from the preceding step as $X$ varies over all valid pairings projecting to $\eta'$.
    
    Therefore this construction defines a bijection from the set of valid pairings projecting to $\eta'$ to the set of assignments of edges to stubs; in turn,
        \begin{align}\label{eq:preimages}
            |\{ X\in\mathcal{Q} \,:\, \Phi(X) = \eta'\}| &= \prod_{v\in V} \res_{\sigma}(v)!.
        \end{align}    
    Since the right-hand side of~\cref{eq:preimages} does not depend on $\eta'$, claim~\emph{(i)} follows.

    For claim~\emph{(ii)}, it follows from statement~\emph{(i)} that the random variable $(\lift f)(X)$ where $X\sim\pi_{\mathcal{Q}}$ has the same law as the random variable $\sum_{\{u,v\}\in\eta}f(\{u,v\})$ where $\eta\sim\pi_{k}^{\sigma}$. Additionally, it can be checked that
        \begin{align}\label{eq:expectation-sum-f}
            \E_{\eta\sim\pi_{k}^{\sigma}} \Bigl( \sum_{\{u,v\}\in \eta} f(\{u,v\}) \Bigr) &= \sum_{\{u,v\}\in \mathcal{R}_{\mathbf{d}}^{(\sigma)}(1)} \Prob_{\eta\sim \pi^{\sigma}_{k}}(\{u,v\}\in \eta) f(\{u,v\})\notag\\
            &= k \sum_{\{u,v\}\in \mathcal{R}_{\mathbf{d}}^{(\sigma)}(1)} \pi_{1}^{\sigma}(\{u,v\}) f(\{u,v\})\notag\\
            &= k \E_{\{u,v\}\sim\pi_{1}^{\sigma}}(f(\{u,v\})).
        \end{align}
    Thus claim~\emph{(ii)} follows, as does claim~\emph{(iii)} for similar reasons.

    Finally, we prove claim \emph{(iv)}. By definition, $p_{\alpha}=\Prob_{X\sim\pi_{\mathcal{Q}} }(\alpha\in X)$. Thus, after some rearranging, we have
        \begin{align*}
            \sum_{\alpha\in\binom{\mathcal{S}}{2}}p_{\alpha}\widehat{f}(\alpha)^{2} &= \sum_{\alpha\in\binom{\mathcal{S}}{2}} \Bigl(\sum_{\substack{X\in\mathcal{Q}\\\alpha\in X}}{\pi_{\mathcal{Q}}} (X)\Bigr) \widehat{f}(\alpha)^{2}\\
            &= \sum_{X\in\mathcal{Q}}{\pi_{\mathcal{Q}}} (X) \sum_{\{\vec{v},\vec{u}\}\in X}\widehat{f}(\{\vec{v},\vec{u}\})^{2}\\
            &= \E_{X\sim\pi_{\mathcal{Q}}} ((\lift(f^{2}))(X))\\
            &= \E_{\eta\sim\pi_{k}^{\sigma}} \Bigl(\sum_{\{u,v\}\in \eta} f(\{u,v\})^{2}\Bigr)\\
            &= k \|f\|_{\pi_{1}^{\sigma}}^{2},
        \end{align*}
    where the final two equalities use the projection law and statement~\emph{(ii)}, applied to $f^{2}$.
\end{proof}

Let $\{\vec{v}, \vec{u}\}\in\binom{\mathcal{S}}{2}$ be a pair of distinct stubs. For a pairing $X\in\mathcal{P}$, let $\theta_{\vec{v},\vec{u}}X\in\mathcal{P}$ denote the pairing obtained by interchanging these two stubs everywhere in $X$; for example, if $\{\vec{v}, \vec{w}\}, \{\vec{u}, \vec{z}\}\in X$ with all four stubs distinct, then $\theta_{\vec{v},\vec{u}}X$ will replace these two pairs with $\{\vec{u}, \vec{w}\}$ and $\{\vec{v}, \vec{z}\}$. Alternatively, if $\{\vec{v}, \vec{u}\}\in X$, then $\theta_{\vec{v},\vec{u}}X = X$. The map $X\mapsto \theta_{\vec{v},\vec{u}} X$ is called a \emph{transposition}. Note that all such transpositions are involutions.

The next lemma provides a useful bound on the marginal probability $p_{\alpha}$.

\begin{lemma}\label{lem:single-pair-marginals}
    Suppose $k\ge 6\Delta^{2}$. Let $\alpha=\{\vec{v},\vec{u}\}\in\binom{\mathcal{S}}{2}$ satisfy $\{u,v\}\in\mathcal{R}_{\mathbf{d}}^{(\sigma)}(1)$. Then
        \begin{align}\label{eq:pa-bound}
            \frac{1-5\Delta^{2}/k}{2k} &\le p_{\alpha} \le \frac{1}{2k-2-4\Delta^{2}}.
        \end{align}
\end{lemma}

We remark that the bounds in~\cref{lem:single-pair-marginals} are implied by those appearing in the preprint~\cite[Lem.~4.2]{larkin2025subgraphs}; in fact, the constants supplied by the latter are sharper than those in~\cref{eq:pa-bound}. However, both versions bound the relative error by an absolute constant times $\Delta^{2}/k$. In the interest of completeness, we give a proof below.

\begin{proof}[Proof of~\cref{lem:single-pair-marginals}]
    We prove both bounds in~\cref{eq:pa-bound} by constructing suitable bipartite graphs and using degree bounds to compare the sizes of their two vertex classes.
    
    For the upper bound, let $A=\{X\in\mathcal{Q}:\alpha\in X\}$ and $B=\mathcal{Q}$, and take disjoint copies of these sets as the two vertex classes. For $X\in A$ and $Y\in B$, draw an edge $\{X,Y\}$ if there exists a stub $\vec{z}\in\mathcal{S}\setminus \alpha$ such that $\theta_{\vec{u},\vec{z}} X=Y$.
    
    Let $X\in A$ and $\vec{z}\in\mathcal{S}\setminus\alpha$ be fixed, and let $\vec{w}$ denote the partner of $\vec{z}$ in $X$. Then $\theta_{\vec{u},\vec{z}} X$ is a valid pairing if \emph{(i)} the vertex $z$ does not belong to the set $N_{\sigma\cup\Phi(X)}[v]$ consisting of $v$ and neighbors of $v$ in the graph induced by $\sigma\cup\Phi(X)$, and \emph{(ii)} the vertex $w$ does not belong to the set $N_{\sigma\cup\Phi(X)}[u]$ consisting of $u$ and its neighbors in the graph induced by $\sigma\cup\Phi(X)$. There are at most $\Delta(\Delta+1)\le 2\Delta^{2}$ stubs $\vec{z}$ which fail to satisfy the first condition, and similarly at most $2\Delta^{2}$ stubs $\vec{w}$ and hence partners $\vec{z}$ which fail to satisfy the second condition. Distinct choices of $\vec{z}$ produce distinct pairings $Y$, since $\vec{z}$ is the partner of $\vec{v}$ in $Y$. Therefore, for each $X\in A$, there are at least $2k-2-4\Delta^{2}$ edges incident to $X$ in the bipartite graph.

    Now let $Y\in B$ be fixed. If $\alpha\in Y$, then the degree of $Y$ is zero. If $\alpha\notin Y$, the stub $\vec{v}$ has a partner stub $\vec{z}\in\mathcal{S}$, i.e., such that $\{\vec{v}, \vec{z}\}\in Y$. The only possible neighbor of $Y$ is $X=\theta_{\vec{u},\vec{z}}Y$, and this is a neighbor precisely when $X$ is valid. Combining the two cases, the degree of $Y$ is at most one. 

    By counting the edges spanning both sides and using the degree bounds obtained above, it follows that
        \begin{align*}
            |A|(2k-2-4\Delta^{2}) &= p_{\alpha}|\mathcal{Q}|(2k-2-4\Delta^{2}) \le |B|=|\mathcal{Q}|,
        \end{align*}
    which, up to rearranging, is the upper bound in~\cref{eq:pa-bound}.

    For the lower bound, let $A'$ denote the set of all pairings $X\in\mathcal{Q}$ such that $\{v,u\}\notin\Phi(X)$, and let $B'$ denote the set of all pairings $Y\in\mathcal{Q}$ such that $\alpha\in Y$. Let $X\in A'$, and let $\{\vec{z},\vec{v}\},\{\vec{w},\vec{u}\}\in X$ denote the (necessarily distinct) pairs containing the stubs $\vec{v},\vec{u}$ in $X$. For each choice of a distinct third pair $\beta=\{\vec{u'},\vec{v'}\}\in X$ together with an orientation $(\vec{u'},\vec{v'})$, draw an edge $\{X,Y\}$ from $X$ to the pairing
        \begin{align*}
            Y &= (X\setminus\{\{\vec{v},\vec{z}\},\{\vec{u},\vec{w}\},\beta\})\cup\{\alpha,\{\vec{z},\vec{u'}\},\{\vec{w},\vec{v'}\}\},
        \end{align*}
    provided $Y\in\mathcal{Q}$ as written.

    For each $X\in A'$, \emph{a priori} there are $2(k-2)$ oriented choices for the pair $\beta$. By the same neighborhood count as in the preceding step, at least $2(k-2)-4\Delta^{2}$ choices satisfy the condition that $u'$ avoids the closed neighborhood of $z$ in $\sigma\cup \Phi(X)$ and $v'$ avoids the closed neighborhood of $w$ in $\sigma\cup \Phi(X)$. Each of these choices produces a valid pairing $Y$ via the surgery above: indeed, the two edges $\{z,u'\}$ and $\{w,v'\}$ are non-loops and absent from $\sigma\cup \Phi(X)$. Moreover, neither equals $\{v,u\}$ since $z,w\notin\{v,u\}$ by validity of $X$ and the absence of $\{v,u\}$ from $\Phi(X)$. Furthermore they cannot equal each other either: equality would require either $z=w$ and $u'=v'$, making $\beta$ project to a loop, or $z=v'$ and $u'=w$, contradicting the neighborhood conditions because $\{u',v'\}\in \Phi(X)$. Thus each of these choices gives a valid pairing $Y$ containing $\alpha$. For fixed $X$, the partners of $\vec{z}$ and $\vec{w}$ in $Y$ recover the ordered pair $(\vec{u'},\vec{v'})$, so distinct oriented choices of $\beta$ produce distinct neighbors. Hence the degree of each vertex $X\in A'$ is at least $2(k-2)-4\Delta^{2}$.

    For the opposite side, let $Y\in B'$ be fixed. Specifying the ordered stubs $\vec{z},\vec{w}$ determines their partners $\vec{u'},\vec{v'}$ in $Y$, and hence recovers both the original pairing $X$ and the orientation of $\beta$. There are at most $(2k)^{2}$ choices of these ordered stubs. Some choices may not give a valid inverse, but including them in the upper bound is unproblematic.
    
    By counting edges once again, it follows that
        \begin{align*}
            |A'|(2(k-2)-4\Delta^{2}) &\le |B'|4k^{2}=|\mathcal{Q}|p_{\alpha}4k^{2}.
        \end{align*}
    There are at most $\res_{\sigma}(v)\res_{\sigma}(u)\le\Delta^{2}$ pairs of stubs projecting to $\{u,v\}$, so the union bound and the upper bound in~\cref{eq:pa-bound} give 
        \begin{align}\label{eq:prob-vu-notin-fq}
            \Prob_{X\sim \pi_{\mathcal{Q}} } (\{v,u\}\notin \Phi(X)) &\ge 1-\frac{\Delta^{2}}{2k-2-4\Delta^{2}} \ge 1-\frac{\Delta^{2}}{k},
        \end{align}
    and hence it follows that
        \begin{align*}
            |\mathcal{Q}| (2(k-2)-4\Delta^{2})(1-\frac{\Delta^{2}}{k}) &\le |\mathcal{Q}|p_{\alpha}4k^{2}.
        \end{align*}
    Rearranging and simplifying, we have
        \begin{align*}
            p_{\alpha} &\ge \frac{(2(k-2)-4\Delta^{2})}{4k^{2}} (1-\frac{\Delta^{2}}{k}) \ge \frac{1}{2k}(1-\frac{\Delta^{2}}{k})(1-\frac{4\Delta^{2}}{k}) \ge\frac{1-5\Delta^{2}/k}{2k}.
        \end{align*}
    Here, we used $4+4\Delta^{2}\le8\Delta^{2}$ and $(1-x)(1-y)\ge1-x-y$ for nonnegative $x,y$. This proves the lower bound in~\cref{eq:pa-bound} and completes the proof.
\end{proof}

The transpositions defined in the preceding discussion give rise to a simple graph $\mathcal{G}$ with vertices consisting of the set $\mathcal{P}$ and where, for distinct pairings $X,Y\in\mathcal{P}$, we say that $X\sim Y$ in $\mathcal{G}$ if $Y=\theta_{\vec{v},\vec{u}}X$ for some transposition $\theta_{\vec{v},\vec{u}}$.

Let $\gen$ denote the combinatorial Laplacian of $\mathcal{G}$, that is, the linear operator on $\R^{\mathcal{P}}$ defined by the equation
    \begin{align}\label{eq:gen-definition}
        (\gen g)(X) &=\sum_{\substack{Y\in\mathcal{P}\\Y\sim X}}(g(X)-g(Y)) =\frac{1}{2}\sum_{\{\vec{v},\vec{u}\}\in\binom{\mathcal{S}}{2}} (g(X)-g(\theta_{\vec{v},\vec{u}}X)),
    \end{align}
for each $g\in\R^{\mathcal{P}}$ and $X\in\mathcal{P}$. To obtain the second equality, observe that for distinct $X,Y$, the equation $Y=\theta_{\vec{v},\vec{u}}X$ has exactly two unordered solutions $\{\vec{v},\vec{u}\}$ when $X\sim Y$, and none otherwise.

The following lemma shows that $\gen$ acts on lifted functions by scalar multiplication, up to an additive constant. This will be useful in later calculations.

\begin{lemma}\label{lem:gen-lift}
    Let $f\in\ell^{2}(\mathcal{R}_{\mathbf{d}}^{(\sigma)}(1),\pi_{1}^{\sigma})$ and $X\in\mathcal{P}$. Then
        \begin{align*}
            (\gen\lift f)(X) &=(2k-1)(\lift f)(X)-\sum_{\alpha\in\binom{\mathcal{S}}{2}}\widehat{f}(\alpha).
        \end{align*}
\end{lemma}

\begin{proof}
    For $\alpha\in\binom{\mathcal{S}}{2}$, define $\mathbf{1}_{\alpha}\in\R^{\mathcal{P}}$ by writing $\mathbf{1}_{\alpha}(X)=\mathbf{1}_{\{\alpha\in X\}}$ for $X\in\mathcal{P}$. If $\alpha\notin X$, exactly two transpositions introduce $\alpha$, each exchanging one of its stubs with the partner of the other. If $\alpha\in X$, exactly $2(2k-2)$ transpositions remove $\alpha$, each exchanging one of its stubs with a stub outside $\alpha$. All other transpositions leave the indicator unchanged. Thus, by~\cref{eq:gen-definition},
        \begin{align*}
            (\gen \mathbf{1}_{\alpha})(X) &=(2k-1)\mathbf{1}_{\alpha}(X)-1.
        \end{align*}
    Since
        \begin{align*}
            (\lift f)(X) &=\sum_{\alpha\in\binom{\mathcal{S}}{2}}\widehat{f}(\alpha)\mathbf{1}_{\alpha}(X),
        \end{align*}
    the claimed identity follows by linearity.
\end{proof}

Along similar lines, we record the following immediate consequence of the definition of the lift operator for later use.

\begin{lemma}\label{lem:increment-lift-f}
    Let $f\in\ell^{2}(\mathcal{R}_{\mathbf{d}}^{(\sigma)}(1),\pi_{1}^{\sigma})$, $X\in\mathcal{P}$, and $\{\vec{v},\vec{u}\}\in\binom{\mathcal{S}}{2}$. If $\{\vec{v},\vec{u}\}\in X$, then $(\lift f)(\theta_{\vec{v},\vec{u}}X)-(\lift f)(X)=0$. Otherwise, let $\vec{z},\vec{w}$ be the respective partners of $\vec{v},\vec{u}$ in $X$. Then
        \begin{align*}
            (\lift f)(\theta_{\vec{v},\vec{u}}X)-(\lift f)(X) &= \widehat{f}(\{\vec{v},\vec{w}\})+\widehat{f}(\{\vec{u},\vec{z}\})\\
            &\qquad-\widehat{f}(\{\vec{v},\vec{z}\})-\widehat{f}(\{\vec{u},\vec{w}\}).
        \end{align*}
\end{lemma}

Next, denote by $\pi_{\mathcal{P}}$ the uniform law on $\mathcal{P}$. For each $g\in\ell^{2}(\mathcal{P},\pi_{\mathcal{P}} )$, define the \emph{all pairings energy}, denoted $\mathcal{E}_{\mathrm{all}}(g,g)$, according to
    \begin{align*}
        \mathcal{E}_{\mathrm{all}}(g,g) &= \langle g,\gen g\rangle_{\pi_{\mathcal{P}} }.
    \end{align*}
Note that by expanding the inner product and by~\cref{eq:gen-definition}, it follows that
    \begin{align*}
        \mathcal{E}_{\mathrm{all}}(g,g) &= \frac{1}{2}\E_{X\sim\pi_{\mathcal{P}} }\Bigl( \sum_{Y\sim X}(g(Y)-g(X))^{2}\Bigr)\\
        &=\frac{1}{4}\E_{X\sim\pi_{\mathcal{P}} }\Bigl( \sum_{\{\vec{v},\vec{u}\}\in\binom{\mathcal{S}}{2}}(g(\theta_{\vec{v},\vec{u}}X)-g(X))^{2}\Bigr).
    \end{align*}
Along similar lines, define the restricted \emph{valid pairings energy} according to
    \begin{align*}
        \mathcal{E}_{\mathrm{good}}(g,g) &= \frac{1}{2}\E_{X\sim\pi_{\mathcal{Q}} }\Bigl( \sum_{Y\sim X}(g(Y)-g(X))^{2}\Bigr)\\
        &= \frac{1}{4}\E_{X\sim\pi_{\mathcal{Q}} }\Bigl( \sum_{\{\vec{v},\vec{u}\}\in\binom{\mathcal{S}}{2}} (g(\theta_{\vec{v},\vec{u}}X)-g(X))^{2}\Bigr).
    \end{align*}
Here, only the starting pairing is required to be valid. The next lemma is the first of the two main estimates which facilitate the proof of~\cref{lem:covariance}.

\begin{lemma}\label{lem:energy-estimate}
    Assume $k\ge 10 \Delta^{2}$. Let $f\in \ell^{2}(\mathcal{R}_{\mathbf{d}}^{(\sigma)}(1),\pi_{1}^{\sigma})$. Then
        \begin{align}\label{eq:lemma-bound-1}
            \mathcal{E}_{\mathrm{good}}(\lift f,\lift f) &\le (2k+220\Delta^{2})k\|f\|_{\pi_{1}^{\sigma}}^{2}.
        \end{align}
\end{lemma}

\begin{proof}
    We first prove three separate estimates and then combine them at the end. First, we compare $\mathcal{E}_{\mathrm{all}}(\lift f, \lift f)$ with $\mathcal{E}_{\mathrm{good}}(\lift f, \lift f)$. We claim that
        \begin{align}\label{eq:good-vs-all-energy}
            \mathcal{E}_{\mathrm{good}}(\lift f,\lift f) &\le (1+\frac{20\Delta^{2}}{k})\mathcal{E}_{\mathrm{all}}(\lift f,\lift f).
        \end{align}
    To this end, fix $\{\vec{v},\vec{u}\}\in\binom{\mathcal{S}}{2}$. By~\cref{lem:increment-lift-f}, it follows that
        \begin{align*}
            &\E_{X\sim\pi_{\mathcal{Q}}}(((\lift f)(\theta_{\vec{v},\vec{u}}X)-(\lift f)(X))^{2})\\
            &\quad=\sum_{(\vec{z},\vec{w})}\sum_{\substack{X\in\mathcal{Q}\\\{\vec{v},\vec{z}\},\{\vec{u},\vec{w}\}\in X}}\pi_{\mathcal{Q}}(X)((\lift f)(\theta_{\vec{v},\vec{u}}X)-(\lift f)(X))^{2}\\
            &\quad=\sum_{(\vec{z},\vec{w})}\Prob_{X\sim\pi_{\mathcal{Q}}}(\{\vec{v},\vec{z}\},\{\vec{u},\vec{w}\}\in X)\times\Bigl(\widehat{f}(\{\vec{v},\vec{w}\})\\
            &\qquad+\widehat{f}(\{\vec{u},\vec{z}\})-\widehat{f}(\{\vec{v},\vec{z}\})-\widehat{f}(\{\vec{u},\vec{w}\})\Bigr)^{2}.
        \end{align*}
    In the last two expressions the outer sum is over ordered pairs of distinct stubs in $\mathcal{S}\setminus\{\vec{v},\vec{u}\}$. The first equality groups pairings according to the partners of $\vec{v}$ and $\vec{u}$; pairings containing $\{\vec{v},\vec{u}\}$ contribute zero. The same expression holds if we interchange $\pi_{\mathcal{Q}} $ with $\pi_{\mathcal{P}}$. Thus,
        \begin{align}\label{eq:increment-lift-f-bound-a}
            &\E_{X\sim\pi_{\mathcal{Q}}}(((\lift f)(\theta_{\vec{v},\vec{u}}X)-(\lift f)(X))^{2})\notag\\
            &\quad\le \Bigl(\max_{(\vec{z},\vec{w})}\frac{\Prob_{X\sim\pi_{\mathcal{Q}}}(\{\vec{v},\vec{z}\},\{\vec{u},\vec{w}\}\in X)}{\Prob_{X\sim\pi_{\mathcal{P}}}(\{\vec{v},\vec{z}\},\{\vec{u},\vec{w}\}\in X)}\Bigr)\notag\\
            &\qquad\times\E_{X\sim\pi_{\mathcal{P}}}(((\lift f)(\theta_{\vec{v},\vec{u}}X)-(\lift f)(X))^{2}).
        \end{align}
    The claimed comparison~\cref{eq:good-vs-all-energy} therefore reduces to a uniform bound on the ratios appearing above.

    For any two disjoint pairs $\{\vec{v},\vec{u}\},\{\vec{z},\vec{w}\}\in\binom{\mathcal{S}}{2}$ satisfying $\{u,v\},\{z,w\}\in\mathcal{R}_{\mathbf{d}}^{(\sigma)}(1)$, we obtain from~\cref{lem:single-pair-marginals} and the proof of the upper bound appearing therein (applied with the first pair held fixed) that
        \begin{align}\label{eq:mu-sigma-two-pairs}
            &\Prob_{X\sim\pi_{\mathcal{Q}} }(\{\vec{v},\vec{u}\}\in X,\ \{\vec{z},\vec{w}\}\in X)\notag\\
            &\quad= \Prob_{X\sim\pi_{\mathcal{Q}} }(\{\vec{v},\vec{u}\}\in X) \Prob_{X\sim\pi_{\mathcal{Q}} }(\{\vec{z},\vec{w}\}\in X\mid\{\vec{v},\vec{u}\}\in X)\notag\\
            &\quad\le \frac{1}{(2k-2-4\Delta^{2})(2k-4-4\Delta^{2})}.
        \end{align}
    This bound also holds if either pair does not appear in any valid pairing, since the joint inclusion probability is then zero.

    Separately, for the unrestricted law, fixing one pair leaves $2k-2$ stubs to pair, so
        \begin{align}\label{eq:nu-sigma-single-pair}
            \Prob_{X\sim\pi_{\mathcal{P}} }(\{\vec{v}, \vec{u}\}\in X)&=\frac{(2k-3)!!}{(2k-1)!!}=\frac{1}{2k-1},
        \end{align}
    since the number of ways to partition a set of $2k$ objects (in this case stubs) into pairs is $(2k-1)!!$. Similarly, for a second pair $\{\vec{z},\vec{w}\}\in\binom{\mathcal{S}}{2}$ disjoint from the first, one has
        \begin{align}\label{eq:nu-sigma-two-pairs}
            \Prob_{X\sim\pi_{\mathcal{P}} }( \{\vec{v},\vec{u}\}\in X,\ \{\vec{z},\vec{w}\}\in X) &= \frac{(2(k-2)-1)!!}{(2k-1)!!} = \frac{1}{(2k-1)(2k-3)}.
        \end{align}
    Computing the ratio of~\cref{eq:mu-sigma-two-pairs} and~\cref{eq:nu-sigma-two-pairs}, we have
        \begin{align}\label{eq:nu-sigma-mu-sigma-ratio}
            \frac{(2k-1)(2k-3)}{(2k-2-4\Delta^{2})(2k-4-4\Delta^{2})} &= \Bigl(1+\frac{1+4\Delta^{2}}{2k-2-4\Delta^{2}}\Bigr) \Bigl(1+\frac{1+4\Delta^{2}}{2k-4-4\Delta^{2}}\Bigr)\notag\\
            &\le (1+\frac{5\Delta^{2}}{k})^{2}\notag\\
            &= 1+\frac{10\Delta^{2}}{k}+\frac{25\Delta^{4}}{k^{2}} \le 1+\frac{20\Delta^{2}}{k},
        \end{align}
    since $\Delta\ge1$ and $k \ge 10\Delta^{2}$. Substituting this uniform ratio bound into~\cref{eq:increment-lift-f-bound-a} and summing over all $\{\vec{v},\vec{u}\}\in\binom{\mathcal{S}}{2}$ proves~\cref{eq:good-vs-all-energy}.

    The second estimate relates $\mathcal{E}_{\mathrm{all}}(\lift f,\lift f)$ to $\|\widehat{f}\|_{2}^{2}$. We claim that
        \begin{align}\label{eq:third-ingredient}
            \mathcal{E}_{\mathrm{all}}(\lift f,\lift f) &\le ( 1+\frac{1}{k} )\|\widehat{f}\|_{2}^{2}.
        \end{align}
    To prove the claim we require an estimate on the variance of $\lift f$ over $\pi_{\mathcal{P}}$. A slightly more generic estimate works here. Suppose $g\in\R^{\binom{\mathcal{S}}{2}}$. We claim that 
        \begin{align}\label{eq:variance-bound-g}
            \Var_{X\sim\pi_{\mathcal{P}} } \Bigl(\sum_{\alpha\in X} g(\alpha) \Bigr) &\le \frac{2k-2}{(2k-1)(2k-3)} \|g\|_{2}^{2}.
        \end{align}
    It suffices to prove~\cref{eq:variance-bound-g} in the case $\sum_{\alpha}g(\alpha)=0$. Observe first that
        \begin{align}\label{eq:disjoint-pairs-identity}
            \sum_{\substack{(\alpha,\beta)\in\binom{\mathcal{S}}{2}\times\binom{\mathcal{S}}{2}\\\alpha\cap\beta=\varnothing}} g(\alpha)g(\beta) &= \Bigl(\sum_{\alpha\in\binom{\mathcal{S}}{2}} g(\alpha)\Bigr)^{2}  + \|g\|_{2}^{2} - \sum_{\vec{v} \in\mathcal{S}}\Bigl( \sum_{\substack{\alpha\in\binom{\mathcal{S}}{2} \\ \vec{v}\in\alpha}} g(\alpha) \Bigr)^{2},
        \end{align}
    where the first sum is over all ordered pairs $(\alpha,\beta)$ of disjoint pairs of stubs. To see this, expand the rightmost term:
        \begin{align*}
            \sum_{\vec{v} \in\mathcal{S}}\Bigl(\sum_{\substack{\alpha\in\binom{\mathcal{S}}{2} \\ \vec{v}\in\alpha}} g(\alpha) \Bigr)^{2} &= \sum_{\vec{v} \in\mathcal{S}} \sum_{\substack{(\alpha,\beta)\in\binom{\mathcal{S}}{2}\times\binom{\mathcal{S}}{2}\\\vec{v}\in\alpha\cap\beta}} g(\alpha)g(\beta).
        \end{align*}  
    The coefficient of $g(\alpha)g(\beta)$ in this expansion is $|\alpha\cap\beta|$, which equals $1$ when the pairs share exactly one stub, $2$ when they coincide, and $0$ when they are disjoint. Therefore~\cref{eq:disjoint-pairs-identity} follows upon inspection. Since we assume that $\sum_{\alpha} g(\alpha) = 0$, the first term on the right-hand side vanishes; the sum of squares is nonnegative, so it follows that
        \begin{align}\label{eq:disjoint-pairs-inequality}
            \sum_{\substack{(\alpha,\beta)\in\binom{\mathcal{S}}{2}\times\binom{\mathcal{S}}{2}\\\alpha\cap\beta=\varnothing}} g(\alpha)g(\beta) &\le \|g\|_{2}^{2}.
        \end{align}
    By~\cref{eq:nu-sigma-single-pair}, we have
        \begin{align*}
            \E_{X\sim\pi_{\mathcal{P}} }\Bigl(\sum_{\alpha\in X}g(\alpha)\Bigr) &= \frac{1}{2k-1}\sum_{\alpha} g(\alpha)=0.
        \end{align*}
    Expanding the square, we obtain
        \begin{align*}
            \Var_{X\sim\pi_{\mathcal{P}}}\Bigl(\sum_{\alpha\in X}g(\alpha)\Bigr) &= \E_{X\sim\pi_{\mathcal{P}}}\Bigl(\Bigl(\sum_{\alpha\in X}g(\alpha)\Bigr)^{2}\Bigr)\\
            &= \sum_{\alpha,\beta\in\binom{\mathcal{S}}{2}}\Prob_{X\sim\pi_{\mathcal{P}}}(\alpha,\beta\in X)\,g(\alpha)g(\beta)\\
            &= \frac{\|g\|_{2}^{2}}{2k-1}+\frac{1}{(2k-1)(2k-3)}\sum_{\alpha\cap\beta=\varnothing}g(\alpha)g(\beta).
        \end{align*}
    The last equality separates the diagonal terms $\alpha=\beta$ and uses the fact that distinct pairs sharing a stub cannot both belong to $X$. Applying~\cref{eq:disjoint-pairs-inequality} now gives
        \begin{align*}
            \Var_{X\sim\pi_{\mathcal{P}}}\Bigl(\sum_{\alpha\in X}g(\alpha)\Bigr) &\le \Bigl(\frac{1}{2k-1}+\frac{1}{(2k-1)(2k-3)}\Bigr)\|g\|_{2}^{2}\\
            &= \frac{2k-2}{(2k-1)(2k-3)}\|g\|_{2}^{2}.
        \end{align*}
    Thus~\cref{eq:variance-bound-g} holds for all $g$, and hence if $g=\widehat{f}$, we have
        \begin{align}\label{eq:variance-bound-lift-f}
            \Var_{X\sim\pi_{\mathcal{P}} }((\lift f)(X)) &\le \frac{2k-2}{(2k-1)(2k-3)}\|\widehat{f}\|_{2}^{2}.
        \end{align}

    Next we transfer~\cref{eq:variance-bound-lift-f} to an energy bound by computing the unrestricted quadratic form for $\lift f$. By~\cref{eq:nu-sigma-single-pair} and the definition of $\lift f$, it follows that
        \begin{align*}
            \E_{Y\sim\pi_{\mathcal{P}} }((\lift f)(Y)) &= \sum_{\alpha\in\binom{\mathcal{S}}{2}} \widehat{f}(\alpha)\Prob_{Y\sim\pi_{\mathcal{P}} }(\alpha\in Y) = \frac{1}{2k-1}\sum_{\alpha\in\binom{\mathcal{S}}{2}}\widehat{f}(\alpha).
        \end{align*}
    Therefore, using the preceding and~\cref{lem:gen-lift}, we have
        \begin{align}\label{eq:gen-lift-f-variance}
            (\gen\lift f)(X) &=(2k-1)(\lift f)(X)-\sum_{\alpha\in\binom{\mathcal{S}}{2}}\widehat{f}(\alpha)\notag\\
            &=(2k-1)(\lift f)(X) -(2k-1)\E_{Y\sim\pi_{\mathcal{P}} }((\lift f)(Y))\notag\\
            &=(2k-1)((\lift f)(X)-\E_{Y\sim\pi_{\mathcal{P}} }((\lift f)(Y))),
        \end{align}
    where $X\in\mathcal{P}$ is fixed. Consequently, by applying~\cref{eq:variance-bound-lift-f} and~\cref{eq:gen-lift-f-variance}, we have that
        \begin{align*}
            \mathcal{E}_{\mathrm{all}}(\lift f,\lift f) &= \langle\lift f,\gen\lift f\rangle_{\pi_{\mathcal{P}} }\\
            &= (2k-1)( \E_{X\sim\pi_{\mathcal{P}} }((\lift f)(X)^{2}) -(\E_{X\sim\pi_{\mathcal{P}} }((\lift f)(X)))^{2})\\
            &= (2k-1)\Var_{X\sim\pi_{\mathcal{P}} }((\lift f)(X))\\
            &\le \frac{2k-2}{2k-3}\|\widehat{f}\|_{2}^{2}\\
            &= ( 1+\frac{1}{2k-3} )\|\widehat{f}\|_{2}^{2}\\
            &\le ( 1+\frac{1}{k} )\|\widehat{f}\|_{2}^{2}.
        \end{align*}
    The last inequality uses $k\ge3$, which follows from our assumption $k\ge 10\Delta^{2}$. This proves~\cref{eq:third-ingredient}.

    Finally, we bound $\|\widehat{f}\|_{2}^{2}$. By~\cref{lem:uniform-projection-lift}\emph{(iv)} and the lower bound in~\cref{lem:single-pair-marginals}, which applies to each pair $\alpha$ with $\widehat{f}(\alpha)\ne0$, it follows that
        \begin{align*}
            \|\widehat{f}\|_{2}^{2} &\le \frac{2k}{1-5\Delta^{2}/k} \sum_{\alpha\in\binom{\mathcal{S}}{2}} p_{\alpha}\widehat{f}(\alpha)^{2} = \frac{2k^{2}}{1-5\Delta^{2}/k}\|f\|_{\pi_{1}^{\sigma}}^{2}.
        \end{align*}
    Since $k\ge10\Delta^{2}$, we have $0\le5\Delta^{2}/k\le1/2$. Using $(1-x)^{-1}\le1+2x$ for $0\le x\le1/2$, we have
        \begin{align}\label{eq:hat-f-norm-bound}
            \|\widehat{f}\|_{2}^{2} &\le 2k^{2}(1+\frac{10\Delta^{2}}{k})\|f\|_{\pi_{1}^{\sigma}}^{2}\notag\\
            &= (2k+20\Delta^{2})k\|f\|_{\pi_{1}^{\sigma}}^{2}.
        \end{align}

    It remains to put all three ingredients together. Combining~\cref{eq:good-vs-all-energy,eq:third-ingredient}, we obtain
        \begin{align*}
            \mathcal{E}_{\mathrm{good}}(\lift f,\lift f) &\le (1+\frac{20\Delta^{2}}{k})\mathcal{E}_{\mathrm{all}}(\lift f,\lift f)\\
            &\le (1+\frac{20\Delta^{2}}{k})(1+\frac{1}{k})\|\widehat{f}\|_{2}^{2}\\
            &= (1+\frac{1}{k}+\frac{20\Delta^{2}}{k}+\frac{20\Delta^{2}}{k^{2}})\|\widehat{f}\|_{2}^{2}\\
            &\le (1+\frac{50\Delta^{2}}{k})\|\widehat{f}\|_{2}^{2}.
        \end{align*}
    Now we use~\cref{eq:hat-f-norm-bound} to get
        \begin{align*}
            \mathcal{E}_{\mathrm{good}}(\lift f,\lift f) &\le (1+\frac{50\Delta^{2}}{k}) (2k+20\Delta^{2})k\|f\|_{\pi_{1}^{\sigma}}^{2}\\
            &= (2k+120\Delta^{2}+\frac{1000\Delta^{4}}{k}) k\|f\|_{\pi_{1}^{\sigma}}^{2}\\
            &\le (2k+220\Delta^{2})k\|f\|_{\pi_{1}^{\sigma}}^{2}.
        \end{align*}
    The last inequality uses $k\ge10\Delta^{2}$. This proves~\cref{eq:lemma-bound-1}.
\end{proof}

For $X\in\mathcal{Q}$, a transposition $\theta_{\vec{v},\vec{u}}$ is called \emph{invalid} (relative to $X$) if $\theta_{\vec{v},\vec{u}}X\notin\mathcal{Q}$. For each $g\in\R^{\mathcal{P}}$ and $X\in\mathcal{Q}$, define
    \begin{align*}
        (\gen_{2}g)(X) &= \sum_{\substack{Y\sim X\\Y\notin\mathcal{Q}}} (g(X)-g(Y)) = \frac{1}{2}\sum_{\substack{\{\vec{v},\vec{u}\}\in\binom{\mathcal{S}}{2}\\ \theta_{\vec{v},\vec{u}}X\notin\mathcal{Q}}} (g(X)-g(\theta_{\vec{v},\vec{u}}X)).
    \end{align*}
The second equality follows from the same observation used in~\cref{eq:gen-definition}: each neighbor of $X$ is obtained by exactly two transpositions.

We now state and prove the second of the two main estimates which facilitate the proof of~\cref{lem:covariance}.

\begin{lemma}\label{lem:invalid-transpositions}
    Assume $k\ge 10 \Delta^{2}$ and let $f\in\ell^{2}(\mathcal{R}_{\mathbf{d}}^{(\sigma)}(1),\pi_{1}^{\sigma})$. Then
        \begin{align}\label{eq:bound-2-lemma}
            \|\gen_{2}\lift f\|_{\pi_{\mathcal{Q}} }^{2} &\le 192 k^{2}\Delta^{4}\|f\|_{\pi_{1}^{\sigma}}^{2}.
        \end{align}
\end{lemma}

\begin{proof}
    The operator $\gen_{2}:\R^{\mathcal{P}}\to\R^{\mathcal{Q}}$ is the part of the Laplacian sum over neighbors outside the valid set. In particular we have that
        \begin{align}\label{eq:gen-bad-norm-1}
            \|\gen_{2}\lift f\|_{\pi_{\mathcal{Q}} }^{2} &= \E_{X\sim\pi_{\mathcal{Q}} }( ((\gen_{2} \lift f)(X))^{2})\notag\\
            &= \frac{1}{4}\E_{X\sim\pi_{\mathcal{Q}} }\Bigl( \Bigl(\sum_{\substack{\{\vec{v},\vec{u}\}\in\binom{\mathcal{S}}{2}\\ \theta_{\vec{v},\vec{u}}X\notin\mathcal{Q}}} (\lift f (X)-\lift f (\theta_{\vec{v},\vec{u}}X)) \Bigr)^{2}\Bigr).
        \end{align}
    It follows from~\cref{lem:increment-lift-f} that
        \begin{align*}
            (\lift f)(X)-(\lift f)(\theta_{\vec{v},\vec{u}}X) &= \sum_{\alpha\in X\setminus \theta_{\vec{v},\vec{u}}X}\widehat{f}(\alpha)-\sum_{\alpha\in \theta_{\vec{v},\vec{u}}X\setminus X}\widehat{f}(\alpha).
        \end{align*}
    The Cauchy--Schwarz inequality then gives
        \begin{align}\label{eq:lift-f-difference-squared}
            ((\lift f)(\theta_{\vec{v},\vec{u}}X)-(\lift f)(X))^{2} &\le 4\sum_{\alpha\in X\setminus \theta_{\vec{v},\vec{u}}X}\widehat{f}(\alpha)^{2} + 4\sum_{\alpha\in \theta_{\vec{v},\vec{u}}X\setminus X}\widehat{f}(\alpha)^{2}.
        \end{align}
    To use this in~\cref{eq:gen-bad-norm-1}, we first claim that for each $X\in\mathcal{Q}$, at most $8k\Delta^{2}$ transpositions produce an invalid pairing. Fix $X\in\mathcal{Q}$ and a stub $\vec{v}$, and let $\vec{z}$ denote its partner in $X$. We count the choices of a second stub $\vec{u}\in\mathcal{S}\setminus\{\vec{v}\}$ for which $\theta_{\vec{v},\vec{u}}X\notin\mathcal{Q}$.

    If $\vec{u}=\vec{z}$, then the transposition interchanges two partners and leaves $X$ unchanged. Otherwise, let $\vec{w}$ denote the partner of $\vec{u}$ in $X$. These four stubs are distinct, although their underlying vertices need not be. The transposition $\theta_{\vec{v},\vec{u}}$ replaces $\{\vec{v},\vec{z}\}$ and $\{\vec{u},\vec{w}\}$ with $\{\vec{u},\vec{z}\}$ and $\{\vec{v},\vec{w}\}$, leaving all other pairs unchanged.

    This transposition can be invalid only if $u$ belongs to the closed neighborhood of $z$ or $w$ belongs to the closed neighborhood of $v$ in $\sigma\cup \Phi(X)$. Each closed neighborhood contains at most $\Delta+1$ vertices, and hence at most $\Delta(\Delta+1)\le2\Delta^{2}$ stubs. For the second condition, there are at most $2\Delta^{2}$ choices of $\vec{w}$, each of which uniquely determines its partner $\vec{u}$ in $X$. Consequently, for each $\vec{v}$, at most $4\Delta^{2}$ choices of $\vec{u}$ give an invalid transposition. Summing over the $2k$ choices of $\vec{v}$ gives at most $(2k)(4\Delta^{2})=8k\Delta^{2}$ invalid transpositions in total.

    Applying Cauchy--Schwarz to~\cref{eq:gen-bad-norm-1} and using this bound, we obtain
        \begin{align*}
            \|\gen_{2}\lift f\|_{\pi_{\mathcal{Q}} }^{2} &\le 2 k\Delta^{2}\E_{X\sim\pi_{\mathcal{Q}} }\Bigl( \sum_{\substack{\{\vec{v},\vec{u}\}\in\binom{\mathcal{S}}{2}\\ \theta_{\vec{v},\vec{u}}X\notin\mathcal{Q}}} ((\lift f)(\theta_{\vec{v},\vec{u}}X)-(\lift f)(X))^{2} \Bigr).
        \end{align*}
    Then we use~\cref{eq:lift-f-difference-squared} to obtain
        \begin{align}\label{eq:gen-bad-lift-norm-bound-intermediate}
            \|\gen_{2}\lift f\|_{\pi_{\mathcal{Q}} }^{2} &\le 8k\Delta^{2}\E_{X\sim\pi_{\mathcal{Q}} }\Bigl( \sum_{\substack{\{\vec{v},\vec{u}\}\in\binom{\mathcal{S}}{2}\\ \theta_{\vec{v},\vec{u}}X\notin\mathcal{Q}}} \Bigl(\sum_{\alpha\in X\setminus\theta_{\vec{v},\vec{u}}X}\widehat{f}(\alpha)^{2}+\sum_{\alpha\in\theta_{\vec{v},\vec{u}}X\setminus X}\widehat{f}(\alpha)^{2}\Bigr)\Bigr).
        \end{align}
    For each $\alpha\in\binom{\mathcal{S}}{2}$ and $X\in\mathcal{Q}$, let $N_{\alpha}^{-}(X)$ and $N_{\alpha}^{+}(X)$ count the unordered pairs of stubs whose transpositions produce an invalid pairing and, respectively, remove or introduce $\alpha$. Explicitly,
        \begin{align*}
            N_{\alpha}^{-}(X) &= \Bigl|\Bigl\{\{\vec{v},\vec{u}\}\in\binom{\mathcal{S}}{2}:\theta_{\vec{v},\vec{u}}X\notin\mathcal{Q},\ \alpha\in X\setminus\theta_{\vec{v},\vec{u}}X\Bigr\}\Bigr|,\\
            N_{\alpha}^{+}(X) &= \Bigl|\Bigl\{\{\vec{v},\vec{u}\}\in\binom{\mathcal{S}}{2}:\theta_{\vec{v},\vec{u}}X\notin\mathcal{Q},\ \alpha\in\theta_{\vec{v},\vec{u}}X\setminus X\Bigr\}\Bigr|.
        \end{align*}
    Putting this into~\cref{eq:gen-bad-lift-norm-bound-intermediate}, we have the upper bound
        \begin{align}\label{eq:gen-bad-lift-norm-bound}
            \|\gen_{2}\lift f\|_{\pi_{\mathcal{Q}} }^{2} &\le 8k\Delta^{2}\sum_{\alpha\in\binom{\mathcal{S}}{2}} \widehat{f}(\alpha)^{2}\E_{X\sim\pi_{\mathcal{Q}} }(N_{\alpha}^{-}+N_{\alpha}^{+}).
        \end{align}
    We now handle the removal and introduction terms separately.

    For the removal term, if $\alpha\notin X$, then $N_{\alpha}^{-}=0$. If $\alpha\in X$, the preceding count gives $N_{\alpha}^{-}\le8\Delta^{2}$. Recalling that $p_{\alpha}=\Prob_{X\sim\pi_{\mathcal{Q}} }(\alpha\in X)$, we obtain
        \begin{align*}
            \E_{X\sim\pi_{\mathcal{Q}} }(N_{\alpha}^{-}) &\le8\Delta^{2}\Prob_{X\sim\pi_{\mathcal{Q}} }(\alpha\in X) =8\Delta^{2}p_{\alpha}.
        \end{align*}
    For the introduction term in~\cref{eq:gen-bad-lift-norm-bound}, we will show that
        \begin{align}\label{eq:bad-introduction-count}
            \E_{X\sim\pi_{\mathcal{Q}} }(N_{\alpha}^{+})&\le16\Delta^{2}p_{\alpha}
        \end{align}
    whenever $\widehat{f}(\alpha)\ne0$. Begin by fixing $\alpha=\{\vec{v},\vec{u}\}$ with $\widehat{f}(\alpha)\ne0$, so $v\ne u$ and $\{v,u\}\notin\sigma$. If $\alpha\in X$, then $N_{\alpha}^{+}=0$. Otherwise, denote the partner of $\vec{v}$ in $X$ by $\vec{z}$ and the partner of $\vec{u}$ by $\vec{w}$. Exactly two transpositions introduce $\alpha$: $\theta_{\vec{v},\vec{w}}$ or $\theta_{\vec{u},\vec{z}}$. Both of these transpositions replace the original pairs by $\alpha$ and $\{\vec{z},\vec{w}\}$. Thus $N_{\alpha}^{+}$ equals two if this replacement is invalid, and zero otherwise. Splitting according to whether the graph edge $\{v,u\}$ is already present gives
        \begin{align}\label{eq:bad-introduction-split}
            \E_{X\sim\pi_{\mathcal{Q}} }(N_{\alpha}^{+}) &=2\Prob_{X\sim\pi_{\mathcal{Q}} }(N_{\alpha}^{+}=2,\ \{v,u\}\in \Phi(X))\notag\\
            &\qquad+2\Prob_{X\sim\pi_{\mathcal{Q}} }(N_{\alpha}^{+}=2,\ \{v,u\}\notin \Phi(X)).
        \end{align}
    We will bound each probability by $4\Delta^{2}p_{\alpha}$.

    For the first probability appearing on the right-hand side of~\cref{eq:bad-introduction-split}, there are at most $\res_{\sigma}(v)\res_{\sigma}(u)\le \Delta^{2}$ stub pairs projecting to $\{v,u\}$. It can be checked that as long as $k\ge 10\Delta^{2}$, by~\cref{lem:single-pair-marginals}, it holds
        \begin{align*}
            \frac{1}{4k}&\le p_{\alpha} \le \frac{1}{k}.
        \end{align*}
    We use these simplified bounds solely for convenience. By the union bound, it holds
        \begin{align*}
            \Prob_{X\sim\pi_{\mathcal{Q}} }(N_{\alpha}^{+}=2,\ \{v,u\}\in \Phi(X)) &\le\Prob_{X\sim\pi_{\mathcal{Q}} }(\{v,u\}\in \Phi(X)) \le\frac{\Delta^{2}}{k}\le4\Delta^{2}p_{\alpha},
        \end{align*}
    as desired.

    For the second probability appearing on the right-hand side of~\cref{eq:bad-introduction-split}, we compare the number of pairings with $N_{\alpha}^{+}=2$ and $\{v,u\}\notin \Phi(X)$ with the number of valid pairings containing $\alpha$, whose total count is $p_{\alpha}|\mathcal{Q}|$. As in the proof of~\cref{lem:single-pair-marginals}, we construct a bipartite graph and compare degrees on its two sides. Define the vertex sets
        \begin{align*}
            A &= \{X\in\mathcal{Q}:N_{\alpha}^{+}(X)=2,\ \{v,u\}\notin\Phi(X)\},\\
            B &= \{Y\in\mathcal{Q}:\alpha\in Y\}.
        \end{align*}
    Note that $|B|=p_{\alpha}|\mathcal{Q}|$.

    Fix $X\in A$, and as before denote the partner of $\vec{v}$ in $X$ by $\vec{z}$ and the partner of $\vec{u}$ by $\vec{w}$. For each third pair $\beta=\{\vec{u'},\vec{v'}\}\in X$, distinct from these two pairs, together with an orientation $(\vec{u'},\vec{v'})$, draw an edge from $X$ to the pairing
        \begin{align*}
            Y &= (X\setminus\{\{\vec{v},\vec{z}\},\{\vec{u},\vec{w}\},\beta\})\cup\{\alpha,\{\vec{z},\vec{u'}\},\{\vec{w},\vec{v'}\}\},
        \end{align*}
    provided $Y\in\mathcal{Q}$. As in the proof of~\cref{lem:single-pair-marginals}, $Y$ is valid if $u'$ avoids the closed neighborhood of $z$ in $\sigma\cup\Phi(X)$ and $v'$ avoids the closed neighborhood of $w$ in $\sigma\cup\Phi(X)$. Thus each $X$ under consideration admits at least
        \begin{align}\label{eq:q-choices}
            2(k-2)-4\Delta^{2}&\ge k
        \end{align}
    choices producing a valid pairing $Y$ containing $\alpha$, provided, e.g., that $k\ge 10 \Delta^{2}$. For fixed $X$, the partners of $\vec{z},\vec{w}$ in $Y$ recover the orientation $(\vec{u'},\vec{v'})$, so distinct choices produce distinct neighbors. Hence each vertex in $A$ has degree at least $k$.

    Now fix $Y\in B$. Specifying the ordered stubs $\vec{z},\vec{w}$ determines their partners $\vec{u'},\vec{v'}$ in $Y$. Reversing the three-pair exchange above then determines both $X$ and the orientation of $\beta$. We bound the number of choices for the ordered tuple $(\vec{z},\vec{w})$ which give rise to $X\in A$. Since $N_{\alpha}^{+}(X)=2$, both $\theta_{\vec{v},\vec{w}}$ and $\theta_{\vec{u},\vec{z}}$ are invalid relative to $X$. These transpositions have the same effect: they remove $\{\vec{v},\vec{z}\}$ and $\{\vec{u},\vec{w}\}$ and introduce $\alpha$ and $\{\vec{z},\vec{w}\}$. The new edge $\{v,u\}$ is permitted since $\widehat{f}(\alpha)\ne0$, and absent from $\Phi(X)$ by assumption. Thus invalidity must come from $\{\vec{z},\vec{w}\}$: either $z=w$ or $\{z,w\}\in\sigma\cup\Phi(X)$.

    We claim that either $z=w$ or else $\{z,w\}\in\sigma\cup\Phi(Y)$ also holds. If $z=w$ or $\{z,w\}\in\sigma$, there is nothing to prove. Otherwise, $\{z,w\}\in\Phi(X)$, and this edge can be removed in forming $Y$ only if
        \begin{align*}
            \{z,w\} &\in \{\{u,w\},\{z,v\},\{u',v'\}\}.
        \end{align*}
    We check all three cases and rule out each one by exhibiting a contradiction. If $\{z,w\}=\{u,w\}$, then $z=u$, so the original pair $\{\vec{v},\vec{z}\}\in X$ projects to $\{v,u\}$, contrary to the assumption $\{v,u\}\notin\Phi(X)$. If $\{z,w\}=\{z,v\}$, then $w=v$, so the original pair $\{\vec{u},\vec{w}\}\in X$ projects to $\{v,u\}$, again a contradiction. Finally, if $\{z,w\}=\{u',v'\}$, then either $u'=z$ and $v'=w$, in which case the new pairs $\{\vec{z},\vec{u'}\}$ and $\{\vec{w},\vec{v'}\}$ project to loops, or $u'=w$ and $v'=z$, in which case they project to the same edge. Both possibilities contradict the validity of $Y$. Hence none of the three removed edges is $\{z,w\}$, and consequently $z=w$ or $\{z,w\}\in\sigma\cup \Phi(Y)$.

    Therefore, for each valid $Y$ containing $\alpha$, there are at most $2k-2$ choices for $\vec{z}$, and, for each choice, the closed neighborhood of $z$ in $\sigma\cup\Phi(Y)$ contains at most $\Delta(\Delta+1)\le2\Delta^{2}$ possible stubs $\vec{w}$, so there are at most
        \begin{align*}
            (2k)(2\Delta^{2})&=4k\Delta^{2}
        \end{align*}
    ordered choices of $(\vec{z},\vec{w})$. Consequently, each vertex in $B$ has degree at most $4k\Delta^{2}$.

    We now combine the degree bounds from both sides. Counting edges gives
        \begin{align*}
            |A|k &\le |B|4k\Delta^{2}=p_{\alpha}|\mathcal{Q}|4k\Delta^{2}.
        \end{align*}
    Therefore,
        \begin{align*}
            k|\mathcal{Q}|\Prob_{X\sim\pi_{\mathcal{Q}} }(N_{\alpha}^{+}=2,\ \{v,u\}\notin \Phi(X)) &\le4k\Delta^{2}p_{\alpha}|\mathcal{Q}|.
        \end{align*}
    Dividing by $k|\mathcal{Q}|$ gives
        \begin{align*}
            \Prob_{X\sim\pi_{\mathcal{Q}} }(N_{\alpha}^{+}=2,\ \{v,u\}\notin \Phi(X)) &\le4\Delta^{2}p_{\alpha},
        \end{align*}
    as desired. Substituting this and the identical bound for the first term into~\cref{eq:bad-introduction-split}, we obtain
        \begin{align*}
            \E_{X\sim\pi_{\mathcal{Q}} }(N_{\alpha}^{+}) &\le2(4\Delta^{2}p_{\alpha})+2(4\Delta^{2}p_{\alpha})=16\Delta^{2}p_{\alpha},
        \end{align*}
    proving~\cref{eq:bad-introduction-count}. Combining the bounds on $\E_{X\sim\pi_{\mathcal{Q}}}(N_{\alpha}^{-})$ and $\E_{X\sim\pi_{\mathcal{Q}}}(N_{\alpha}^{+})$ for the pairs with $\widehat{f}(\alpha)\ne0$, we obtain from~\cref{eq:gen-bad-lift-norm-bound} that
        \begin{align*}
            \|\gen_{2}\lift f\|_{\pi_{\mathcal{Q}} }^{2} &\le 8 k \Delta^{2} (8\Delta^{2}+16\Delta^{2}) \sum_{\alpha\in\binom{\mathcal{S}}{2}}p_{\alpha}\widehat{f}(\alpha)^{2} = 192 k^{2} \Delta^{4}\|f\|_{\pi_{1}^{\sigma}}^{2},
        \end{align*}
    where the equality holds by~\cref{lem:uniform-projection-lift}\emph{(iv)}. This proves~\cref{eq:bound-2-lemma}.
\end{proof}

We are now ready to prove~\cref{lem:covariance}.

\begin{proof}[Proof of~\cref{lem:covariance}]
    The bound in~\cref{eq:covariance-bound} is unchanged by adding constants to $f$, since in particular for any $\eta\in \mathcal{R}_{\mathbf{d}}^{(\sigma)}(k)$ and any constant $c$, it holds
        \begin{align*}
            \sum_{\{u,v\}\in \eta} (f(\{u,v\}) + c) &= ck + \sum_{\{u,v\}\in \eta} f(\{u,v\}).
        \end{align*}
    Thus we may assume, without loss of generality, that $\E_{\{u,v\}\sim\pi_{1}^{\sigma}}(f(\{u,v\}))=0$. It follows from the proof of~\cref{lem:uniform-projection-lift}, specifically~\cref{eq:expectation-sum-f}, that additionally
        \begin{align*}
            \E_{\eta\sim\pi_{k}^{\sigma}}\Bigl(\sum_{\{u,v\}\in \eta} f(\{u,v\})\Bigr) &= 0.
        \end{align*}
    By~\cref{lem:uniform-projection-lift}\emph{(ii)} and~\emph{(iii)}, $\E_{X\sim\pi_{\mathcal{Q}} }((\lift f)(X))=0$ and
        \begin{align*}
            \|\lift f\|_{\pi_{\mathcal{Q}} }^{2} &= \Var_{\eta\sim\pi_{k}^{\sigma}}\Bigl(\sum_{\{u,v\}\in \eta}f(\{u,v\})\Bigr).
        \end{align*}
    Since $f$ and $\lift f$ are centered under their respective measures,  the desired bound~\cref{eq:covariance-bound} is therefore equivalent to
        \begin{align}\label{eq:tf-norm-bound}
            \|\lift f\|_{\pi_{\mathcal{Q}} }^{2} &\le (1+\frac{200 \Delta^{2}}{\sqrt{k}})k\|f\|_{\pi_{1}^{\sigma}}^{2}.
        \end{align}
    Let $\gen_{1}$ be the combinatorial Laplacian of the induced graph $\mathcal{G}[\mathcal{Q}]$, where we recall that $\mathcal{G}$ is the graph on $\mathcal{P}$ with edges corresponding to transpositions. Let $g\in\R^{\mathcal{Q}}$. Then for each $X\in\mathcal{Q}$,
        \begin{align*}
            (\gen_{1}g)(X) &=\sum_{\substack{Y\sim X\\Y\in\mathcal{Q}}}(g(X)-g(Y))=\frac{1}{2}\sum_{\substack{\{\vec{v},\vec{u}\}\in\binom{\mathcal{S}}{2}\\ \theta_{\vec{v},\vec{u}}X\in\mathcal{Q}}}(g(X)-g(\theta_{\vec{v},\vec{u}}X)).
        \end{align*}
    For $g\in\R^{\mathcal{P}}$ and $X\in\mathcal{Q}$, we therefore have $(\gen g)(X)=(\gen_{1}g)(X)+(\gen_{2}g)(X)$. We now express $\|\lift f\|_{\pi_{\mathcal{Q}} }^{2}$ in terms of $\gen_{1}$ and $\gen_{2}$. Taking an inner product with $\lift f$ and using~\cref{lem:gen-lift} gives
        \begin{align*}
            \langle\lift f,\gen\lift f\rangle_{\pi_{\mathcal{Q}} } &=(2k-1)\|\lift f\|_{\pi_{\mathcal{Q}} }^{2} - \Bigl(\sum_{\{\vec{z},\vec{w}\}\in\binom{\mathcal{S}}{2}} \widehat{f}(\{\vec{z},\vec{w}\})\Bigr)\E_{X\sim\pi_{\mathcal{Q}} }((\lift f)(X))\\
            &=(2k-1)\|\lift f\|_{\pi_{\mathcal{Q}} }^{2},
        \end{align*}
    where the last equality uses $\E_{X\sim\pi_{\mathcal{Q}} }((\lift f)(X))=0$. Separating $\gen\lift f$ into its valid and invalid parts therefore gives
        \begin{align}\label{eq:pairing-variance-identity}
            {\color{blue}}(2k-1)\|\lift f\|_{\pi_{\mathcal{Q}} }^{2} &= \langle\lift f,\gen_{1}\lift f\rangle_{\pi_{\mathcal{Q}} } +\langle\lift f,\gen_{2}\lift f\rangle_{\pi_{\mathcal{Q}} }.
        \end{align}
    To bound the first term, let $g\in\ell^{2}(\mathcal{Q},\pi_{\mathcal{Q}} )$. We compute 
        \begin{align*}
            {\color{blue}}\langle g, \gen_{1}g\rangle_{\pi_{\mathcal{Q}} } &= \frac{1}{2|\mathcal{Q}|} \sum_{X\in\mathcal{Q}}  g(X) \sum_{\substack{\{\vec{v},\vec{u}\}\in\binom{\mathcal{S}}{2}\\ \theta_{\vec{v},\vec{u}}X\in\mathcal{Q}}} (g(X)-g(\theta_{\vec{v},\vec{u}}X)).
        \end{align*}
    Since each transposition is an involution, for each valid transposition $\{\vec{v},\vec{u}\}$ and corresponding pairings $X$ and $Y=\theta_{\vec{v},\vec{u}}X$, the contributions to the quadratic form can be combined:
        \begin{align*}
            g(X)(g(X) - g(Y)) + g(Y)(g(Y) - g(X)) &= (g(X) - g(Y))^{2},
        \end{align*}
    so that
        \begin{align*}
            \langle g, \gen_{1}g\rangle_{\pi_{\mathcal{Q}} } &= \frac{1}{4}\E_{X\sim\pi_{\mathcal{Q}} } \Bigl(\sum_{\substack{\{\vec{v},\vec{u}\}\in\binom{\mathcal{S}}{2}\\ \theta_{\vec{v},\vec{u}}X\in\mathcal{Q}}} (g(\theta_{\vec{v},\vec{u}}X)-g(X))^{2} \Bigr).
        \end{align*}
    This leads to the following upper bound on the first term of the right-hand side of~\cref{eq:pairing-variance-identity}:
        \begin{align}\label{eq:first-term-upper-bound}
            \langle\lift f,\gen_{1}\lift f\rangle_{\pi_{\mathcal{Q}} } &= \frac{1}{4}\E_{X\sim\pi_{\mathcal{Q}} } \Bigl( \sum_{\substack{\{\vec{v},\vec{u}\}\in\binom{\mathcal{S}}{2}\\ \theta_{\vec{v},\vec{u}}X\in\mathcal{Q}}} ((\lift f)(\theta_{\vec{v},\vec{u}}X)-(\lift f)(X))^{2} \Bigr)\notag\\
            &\le \frac{1}{4}\E_{X\sim\pi_{\mathcal{Q}} } \Bigl( \sum_{\{\vec{v},\vec{u}\}\in\binom{\mathcal{S}}{2}} ((\lift f)(\theta_{\vec{v},\vec{u}}X)-(\lift f)(X))^{2} \Bigr)\notag\\
            &=\mathcal{E}_{\mathrm{good}}(\lift f,\lift f),
        \end{align}
    obtained by relaxing the constraint that the transposition must result in a valid pairing. The second term on the right-hand side of~\cref{eq:pairing-variance-identity} can be bounded from above by using the Cauchy--Schwarz inequality, namely,
        \begin{align}\label{eq:cauchy-schwarz-second-term}
            |\langle\lift f,\gen_{2}\lift f\rangle_{\pi_{\mathcal{Q}} }| &\le \|\lift f\|_{\pi_{\mathcal{Q}} } \|\gen_{2}\lift f\|_{\pi_{\mathcal{Q}} }.
        \end{align}
    Combining~\cref{eq:pairing-variance-identity} with the bounds~\cref{eq:first-term-upper-bound} and~\cref{eq:cauchy-schwarz-second-term} gives the inequality
        \begin{align}\label{eq:pairing-variance-bound}
            {\color{blue}}(2k-1)\|\lift f\|_{\pi_{\mathcal{Q}} }^{2} &\le \mathcal{E}_{\mathrm{good}}(\lift f,\lift f) + \|\lift f\|_{\pi_{\mathcal{Q}} } \|\gen_{2}\lift f\|_{\pi_{\mathcal{Q}} }.
        \end{align}
    If $f=0$, the desired bound is immediate. Otherwise, set
        \begin{align*}
            x &=\frac{\|\lift f\|_{\pi_{\mathcal{Q}} }}{\sqrt{k}\,\|f\|_{\pi_{1}^{\sigma}}}.
        \end{align*}
    First, dividing~\cref{eq:pairing-variance-bound} by $(2k-1)k\|f\|_{\pi_{1}^{\sigma}}^{2}$ gives
        \begin{align*}
            x^{2} &\le \frac{\mathcal{E}_{\mathrm{good}}(\lift f,\lift f)} {(2k-1)k\|f\|_{\pi_{1}^{\sigma}}^{2}} +\frac{\|\gen_{2}\lift f\|_{\pi_{\mathcal{Q}} }}{(2k-1)\sqrt{k}\,\|f\|_{\pi_{1}^{\sigma}}}x.
        \end{align*}
    Applying~\cref{lem:energy-estimate} to the first term, we obtain
        \begin{align*}
            x^{2} &\le \frac{2k+220\Delta^{2}}{2k-1} + \frac{\|\gen_{2}\lift f\|_{\pi_{\mathcal{Q}} }}{(2k-1)\sqrt{k}\,\|f\|_{\pi_{1}^{\sigma}}}x.
        \end{align*}
    Next, applying~\cref{lem:invalid-transpositions} to the second term gives
        \begin{align*}
            x^{2} &\le \frac{2k+220\Delta^{2}}{2k-1} + \frac{\sqrt{192}\,k \Delta^{2}\|f\|_{\pi_{1}^{\sigma}}}{(2k-1)\sqrt{k}\,\|f\|_{\pi_{1}^{\sigma}}}x = 1 + \frac{220\Delta^{2}+1}{2k-1} + \frac{\Delta^{2} \sqrt{192\, k}}{2k-1}x.
        \end{align*}
    Finally, since $k\ge10\Delta^{2}\ge10$, we have $2k-1\ge19k/10$. Hence
        \begin{align}\label{eq:x-bound}
            x^{2} &\le 1+\frac{10(220\Delta^{2}+1)}{19k} +\frac{10\sqrt{192}\,\Delta^{2}}{19\sqrt{k}}x \le 1+\frac{117\Delta^{2}}{k} + \frac{8\Delta^{2}}{\sqrt{k}}x.
        \end{align}

    Since $k\ge 10\Delta^{2}$ and $k\ge \Delta^{4}/2$, the preceding inequality gives
        \begin{align*}
            x^{2} &\le 1+\frac{117}{10}+8\sqrt{2}\,x \le 13+12x.
        \end{align*}
    Rearranging, we obtain $(x-13)(x+1)\le0$. Since $x\ge0$, it follows that $x\le13$.

    Substituting $x \le 13$ into~\cref{eq:x-bound}, we obtain
        \begin{align*}
            x^{2} &\le 1+\frac{117\Delta^{2}}{k} +\frac{8\Delta^{2}}{\sqrt{k}}x \le 1+(\frac{117}{\sqrt{k}}+104) \frac{\Delta^{2}}{\sqrt{k}} \le 1+\frac{200\Delta^{2}}{\sqrt{k}}.
        \end{align*}
    Here the last inequality follows from $k\ge10$, which gives $117/\sqrt{k}+104\le141<200$. Recalling the definition of $x$, we conclude that
        \begin{align*}
            \|\lift f\|_{\pi_{\mathcal{Q}} }^{2} &\le (1+\frac{200\Delta^{2}}{\sqrt{k}}) k\|f\|_{\pi_{1}^{\sigma}}^{2}.
        \end{align*}
    This proves~\cref{eq:tf-norm-bound}, which in turn proves~\cref{eq:covariance-bound}. The lemma is proved.
\end{proof}

\Cref{thm:intro-link-eigenvalue-realization-complex} now follows from~\cref{thm:link-eigenvalue-switch-complex}, after increasing the absolute constant.

\section{Spectral gaps of simplicial switch chains}\label{sec:spectral-gaps}

In the preceding section we stated and proved~\cref{thm:link-eigenvalue-switch-complex}, which asserts that after a certain cutoff, the links of the switch complex $\mathcal{R}_{\mathbf{d}}$ demonstrate strong spectral expansion properties. The purpose of this section is to extend this local property to a global property of the down-$2$-up-$2$ walk on the facets of $\mathcal{R}_{\mathbf{d}}$. We have the following result.

\begin{theorem}\label{thm:two-edge-eigenvalue}
    There exists an absolute constant $C>0$ with the following property. If $m\ge C\Delta^{8}$, then
        \begin{align*}
            1-\frac{2}{m} &\le\lambda_{2}(\downup_{2})\le1-\frac{1}{Cm}.
        \end{align*}
\end{theorem}

The proof of~\cref{thm:two-edge-eigenvalue} begins by viewing the energy of a block walk as the average variance remaining after most edges have been fixed. This reduces the comparison with $\downup_{2}$ to a collection of smaller completion problems. The useful observation is that a random restriction usually makes these problems equivalent to ones with bounded prescribed degrees, where a rough variance bound is enough. Deep link expansion ensures that the average conditional variance still controls the original variance, while the exceptional restrictions contribute little enough to absorb. Averaging the corresponding two-edge energies then gives the required comparison. An edge indicator supplies the matching upper bound on the spectral gap.

As a first step, we prove the following useful lemma, which allows us to decompose Dirichlet energies into averages of local contributions. This lemma holds for a generic weighted simplicial complex, and its proof largely amounts to basic manipulations.

\begin{lemma}\label{lem:dirichlet-energy-decomposition}
    Let $\mathcal{X}$ be a pure $m$-dimensional simplicial complex, and let $1\le j\le \ell\le k\le m$ be fixed integers. Let $\sigma\in\mathcal{X}(m-k)$. Then the following facts hold.
        \begin{enumerate}[label=(\roman*)]
            \item For each $\eta\in\mathcal{X}^{(\sigma)}(k)$ and each $\rho\subseteq\eta$ with $|\rho|=k-\ell$, we have
                \begin{align}\label{eq:fact-1}
                    \frac{\pi_{k}^{\sigma}(\eta)}{\binom{k}{k-\ell}} &= \pi_{k-\ell}^{\sigma}(\rho)\pi_{\ell}^{\sigma\cup\rho}(\eta\setminus\rho).
                \end{align}
        \end{enumerate}
    Furthermore, let $g\in\ell^{2}(\mathcal{X}^{(\sigma)}(k),\pi_{k}^{\sigma})$ be fixed. For each $\rho\in\mathcal{X}^{(\sigma)}(k-\ell)$, define $g^{(\rho)}(\eta')=g(\rho\cup\eta')$ for $\eta'\in\mathcal{X}^{(\sigma\cup\rho)}(\ell)$. Then
        \begin{enumerate}[label=(\roman*)]
            \setcounter{enumi}{1}
            \item $\mathcal{E}_{\downup_{\ell}^{(\sigma)}}(g,g)=\E_{\rho\sim\pi_{k-\ell}^{\sigma}}(\Var_{\eta'\sim\pi_{\ell}^{\sigma\cup\rho}}(g^{(\rho)}(\eta')))$.
            \item $\mathcal{E}_{\downup_{j}^{(\sigma)}}(g,g)=\E_{\rho\sim\pi_{k-\ell}^{\sigma}}(\mathcal{E}_{\downup_{j}^{(\sigma\cup\rho)}}(g^{(\rho)},g^{(\rho)}))$.
        \end{enumerate}
\end{lemma}

\begin{proof}
    To prove claim~\emph{(i)}, fix $\eta\in\mathcal{X}^{(\sigma)}(k)$ and $\rho\subseteq\eta$ with $|\rho|=k-\ell$. There are exactly $|\mathcal{X}^{(\sigma\cup\rho)}(\ell)|$ facets of $\mathcal{X}^{(\sigma)}$ containing $\rho$. Since $\pi_{k-\ell}^{\sigma}$ is obtained by choosing a facet uniformly and then choosing one of its $(k-\ell)$-faces uniformly, we have
        \begin{align*}
            \pi_{k-\ell}^{\sigma}(\rho) &= \frac{|\mathcal{X}^{(\sigma\cup\rho)}(\ell)|}{\binom{k}{k-\ell}|\mathcal{X}^{(\sigma)}(k)|}.
        \end{align*}
    The measures on the facets of both links are uniform. Consequently,
        \begin{align*}
            \pi_{k-\ell}^{\sigma}(\rho)\pi_{\ell}^{\sigma\cup\rho}(\eta\setminus\rho) &= \frac{|\mathcal{X}^{(\sigma\cup\rho)}(\ell)|}{\binom{k}{k-\ell}|\mathcal{X}^{(\sigma)}(k)|}\frac{1}{|\mathcal{X}^{(\sigma\cup\rho)}(\ell)|} = \frac{\pi_{k}^{\sigma}(\eta)}{\binom{k}{k-\ell}}.
        \end{align*}
    This proves claim~\emph{(i)}.

    We next prove claim~\emph{(ii)}. By definition, the Dirichlet energy is given by
        \begin{align*}
            \mathcal{E}_{\downup_{\ell}^{(\sigma)}}(g,g) &= \frac{1}{2}\sum_{\eta,\tau\in\mathcal{X}^{(\sigma)}(k)}\pi_{k}^{\sigma}(\eta)\downup_{\ell}^{(\sigma)}(\eta,\tau)(g(\eta)-g(\tau))^{2}.
        \end{align*}
    To expand $\downup_{\ell}^{(\sigma)}(\eta,\tau)$, note that $\downup_{\ell}^{(\sigma)}$ sums over retained faces $\rho\subseteq\eta\cap\tau$ of size $k-\ell$. Each such face is chosen with probability $1/\binom{k}{k-\ell}$, after which the completion $\tau\setminus\rho$ is chosen with probability $\pi_{\ell}^{\sigma\cup\rho}(\tau\setminus\rho)$. Expanding the kernel and grouping by $\rho$, we obtain
        \begin{align*}
            &\mathcal{E}_{\downup_{\ell}^{(\sigma)}}(g,g)\\
            &\quad= \frac{1}{2}\sum_{\rho\in\mathcal{X}^{(\sigma)}(k-\ell)}\sum_{\substack{\eta,\tau\in\mathcal{X}^{(\sigma)}(k)\\\rho\subseteq\eta\cap\tau}}\frac{\pi_{k}^{\sigma}(\eta)}{\binom{k}{k-\ell}}\pi_{\ell}^{\sigma\cup\rho}(\tau\setminus\rho)(g(\eta)-g(\tau))^{2}\\
            &\quad= \frac{1}{2}\sum_{\rho\in\mathcal{X}^{(\sigma)}(k-\ell)}\pi_{k-\ell}^{\sigma}(\rho)\sum_{\substack{\eta,\tau\in\mathcal{X}^{(\sigma)}(k)\\\rho\subseteq\eta\cap\tau}}\pi_{\ell}^{\sigma\cup\rho}(\eta\setminus\rho)\pi_{\ell}^{\sigma\cup\rho}(\tau\setminus\rho)(g^{(\rho)}(\eta\setminus\rho)-g^{(\rho)}(\tau\setminus\rho))^{2},
        \end{align*}
    where the second equality uses claim~\emph{(i)} and the definition of $g^{(\rho)}$. By the change of variables $\eta'=\eta\setminus\rho$ and $\eta''=\tau\setminus\rho$, we have
        \begin{align*}
            &\mathcal{E}_{\downup_{\ell}^{(\sigma)}}(g,g)\\
            &\quad= \frac{1}{2}\sum_{\rho\in\mathcal{X}^{(\sigma)}(k-\ell)}\Bigl(\pi_{k-\ell}^{\sigma}(\rho)\sum_{\eta',\eta''\in\mathcal{X}^{(\sigma\cup\rho)}(\ell)}\Bigl(\pi_{\ell}^{\sigma\cup\rho}(\eta')\pi_{\ell}^{\sigma\cup\rho}(\eta'')\\
            &\qquad\qquad\times(g^{(\rho)}(\eta')-g^{(\rho)}(\eta''))^{2}\Bigr)\Bigr)\\
            &\quad= \sum_{\rho\in\mathcal{X}^{(\sigma)}(k-\ell)}\Bigl(\pi_{k-\ell}^{\sigma}(\rho)\Bigl(\sum_{\eta'\in\mathcal{X}^{(\sigma\cup\rho)}(\ell)}\pi_{\ell}^{\sigma\cup\rho}(\eta')g^{(\rho)}(\eta')^{2}\\
            &\qquad\qquad-\Bigl(\sum_{\eta'\in\mathcal{X}^{(\sigma\cup\rho)}(\ell)}\pi_{\ell}^{\sigma\cup\rho}(\eta')g^{(\rho)}(\eta')\Bigr)^{2}\Bigr)\Bigr)\\
            &\quad= \E_{\rho\sim\pi_{k-\ell}^{\sigma}}(\Var_{\eta'\sim\pi_{\ell}^{\sigma\cup\rho}}(g^{(\rho)}(\eta'))).
        \end{align*}
    Finally, we prove claim~\emph{(iii)}. Fix $\tau\in\mathcal{X}^{(\sigma)}(k-j)$ and $\rho\subseteq\tau$ with $|\rho|=k-\ell$. By the definition of the induced measures and the same facet count used in claim~\emph{(i)}, we have
        \begin{align*}
            \pi_{k-\ell}^{\sigma}(\rho)\pi_{\ell-j}^{\sigma\cup\rho}(\tau\setminus\rho) &= \frac{\pi_{k-j}^{\sigma}(\tau)}{\binom{k-j}{k-\ell}}.
        \end{align*}
    For the energy average, we have
        \begin{align*}
            &\E_{\rho\sim\pi_{k-\ell}^{\sigma}}(\mathcal{E}_{\downup_{j}^{(\sigma\cup\rho)}}(g^{(\rho)},g^{(\rho)}))\\
            &\quad= \sum_{\rho\in\mathcal{X}^{(\sigma)}(k-\ell)}\pi_{k-\ell}^{\sigma}(\rho)\mathcal{E}_{\downup_{j}^{(\sigma\cup\rho)}}(g^{(\rho)},g^{(\rho)})\\
            &\quad= \sum_{\rho\in\mathcal{X}^{(\sigma)}(k-\ell)}\pi_{k-\ell}^{\sigma}(\rho)\sum_{\tau'\in\mathcal{X}^{(\sigma\cup\rho)}(\ell-j)}\pi_{\ell-j}^{\sigma\cup\rho}(\tau')\Var_{\eta'\sim\pi_{j}^{\sigma\cup\rho\cup\tau'}}(g^{(\rho)}(\tau'\cup\eta')).
        \end{align*}
    The second equality applies claim~\emph{(ii)} within each link $\mathcal{X}^{(\sigma\cup\rho)}$, with deletion size $j$. By the change of variables $\tau=\rho\cup\tau'$ and the definition of $g^{(\rho)}$, we obtain
        \begin{align*}
            &\E_{\rho\sim\pi_{k-\ell}^{\sigma}}(\mathcal{E}_{\downup_{j}^{(\sigma\cup\rho)}}(g^{(\rho)},g^{(\rho)}))\\
            &\quad= \sum_{\tau\in\mathcal{X}^{(\sigma)}(k-j)}\sum_{\substack{\rho\subseteq\tau\\|\rho|=k-\ell}}\pi_{k-\ell}^{\sigma}(\rho)\pi_{\ell-j}^{\sigma\cup\rho}(\tau\setminus\rho)\Var_{\eta'\sim\pi_{j}^{\sigma\cup\tau}}(g(\tau\cup\eta'))\\
            &\quad= \sum_{\tau\in\mathcal{X}^{(\sigma)}(k-j)}\frac{\pi_{k-j}^{\sigma}(\tau)}{\binom{k-j}{k-\ell}}\Bigl(\sum_{\substack{\rho\subseteq\tau\\|\rho|=k-\ell}}1\Bigr)\Var_{\eta'\sim\pi_{j}^{\sigma\cup\tau}}(g(\tau\cup\eta'))\\
            &\quad= \sum_{\tau\in\mathcal{X}^{(\sigma)}(k-j)}\pi_{k-j}^{\sigma}(\tau)\Var_{\eta'\sim\pi_{j}^{\sigma\cup\tau}}(g(\tau\cup\eta'))\\
            &\quad= \mathcal{E}_{\downup_{j}^{(\sigma)}}(g,g).
        \end{align*}
    The second equality uses the identity proved above, and the third evaluates the inner sum as $\binom{k-j}{k-\ell}$. The last equality is claim~\emph{(ii)} with $\ell$ replaced by $j$. This proves claim~\emph{(iii)} and completes the proof.
\end{proof}

The following lemma combines~\cref{thm:link-eigenvalue-switch-complex} with~\cref{thm:alev-rao} to bound the spectral gap of the down-$j$-up-$j$ walk on each link, provided $j$ is sufficiently large in terms of $\Delta$.

\begin{lemma}\label{lem:higher-order-poincare}
    Let $j,k$ be integers satisfying $10^{6}\Delta^{4}\le j\le k\le m$, and let $\sigma\in\mathcal{R}_{\mathbf{d}}(m-k)$. Then for each $g\in\ell^{2}(\mathcal{R}_{\mathbf{d}}^{(\sigma)}(k),\pi_{k}^{\sigma})$, it holds
        \begin{align}\label{eq:higher-order-poincare}
            \Var_{\eta\sim\pi_{k}^{\sigma}}(g(\eta)) &\le \frac{2k}{j} \mathcal{E}_{\downup_{j}^{(\sigma)}}(g,g).
        \end{align}
\end{lemma}

\begin{proof}
    If the link $\mathcal{R}_{\mathbf{d}}^{(\sigma)}$ has only a single facet, both sides vanish. If $j=k$, the walk resamples the entire completion and the resulting Markov kernel has spectral gap at least one, in which case~\cref{eq:higher-order-poincare} holds. We may therefore assume $j<k$. For each $j+1\le\ell\le k$ and $\rho\in\mathcal{R}_{\mathbf{d}}^{(\sigma)}(k-\ell)$, we have
        \begin{align*}
            (\mathcal{R}_{\mathbf{d}}^{(\sigma)})^{(\rho)} &= \mathcal{R}_{\mathbf{d}}^{(\sigma\cup\rho)},
        \end{align*}
    where $\sigma\cup\rho\in\mathcal{R}_{\mathbf{d}}(m-\ell)$. For the complex $\mathcal{R}_{\mathbf{d}}^{(\sigma)}$, as in~\cref{thm:alev-rao}, let $\gamma_{k-\ell}$ denote the maximum second eigenvalue of the link walks over all $(k-\ell)$-faces. By~\cref{thm:link-eigenvalue-switch-complex} and our assumption on $j$,
        \begin{align*}
            \gamma_{k-\ell} &= \max_{\rho\in\mathcal{R}_{\mathbf{d}}^{(\sigma)}(k-\ell)}\lambda_{2}(\downup^{(\sigma\cup\rho)}) \le \frac{200\Delta^{2}}{(\ell-1)\sqrt{\ell}}.
        \end{align*}
    We furthermore have
        \begin{align*}
            \sum_{\ell=j+1}^{k}\frac{200\Delta^{2}}{(\ell-1)\sqrt{\ell}} &\le 200\Delta^{2}\sum_{\ell=j}^{\infty}\ell^{-3/2}\le 200\Delta^{2}\int_{j-1}^{\infty}s^{-3/2}\,ds = \frac{400\Delta^{2}}{\sqrt{j-1}} \le \frac{1}{2},
        \end{align*}
    since $j\ge10^{6}\Delta^{4}$. Applying~\cref{thm:alev-rao} and reindexing the product gives
        \begin{align*}
            1-\lambda_{2}(\downup_{j}^{(\sigma)}) &\ge \frac{j}{k}\prod_{\ell=j+1}^{k}(1-\gamma_{k-\ell}) \ge \frac{j}{k}\prod_{\ell=j+1}^{k}\Bigl(1-\frac{200\Delta^{2}}{(\ell-1)\sqrt{\ell}}\Bigr).
        \end{align*}
    Using $\prod_{i}(1-a_{i})\ge1-\sum_{i}a_{i}$ for $0\le a_{i}\le1$, together with the preceding sum bound, we obtain
        \begin{align*}
            1-\lambda_{2}(\downup_{j}^{(\sigma)}) &\ge \frac{j}{k}\Bigl(1-\sum_{\ell=j+1}^{k}\frac{200\Delta^{2}}{(\ell-1)\sqrt{\ell}}\Bigr) \ge \frac{j}{k}\Bigl(1-\frac{400\Delta^{2}}{\sqrt{j-1}}\Bigr) \ge \frac{j}{2k}.
        \end{align*}
    By~\cref{prop:up-down-adjoint}, the walk $\downup_{j}^{(\sigma)}$ is reversible with respect to $\pi_{k}^{\sigma}$. The variational characterization of its spectral gap therefore gives
        \begin{align*}
            \mathcal{E}_{\downup_{j}^{(\sigma)}}(g,g) &\ge (1-\lambda_{2}(\downup_{j}^{(\sigma)}))\Var_{\eta\sim\pi_{k}^{\sigma}}(g(\eta)) \ge \frac{j}{2k}\Var_{\eta\sim\pi_{k}^{\sigma}}(g(\eta)).
        \end{align*}
    This proves~\cref{eq:higher-order-poincare}.
\end{proof}

\Cref{cor:intro-higher-order-eigenvalue} now follows from~\cref{lem:higher-order-poincare}, applied with $\sigma=\varnothing$ and $k=m$, by the variational characterization of the spectral gap.

The next step is to transfer the preceding Poincar\'{e} inequality to the down-$2$-up-$2$ walk by comparing their Dirichlet energies.

\begin{lemma}\label{lem:finite-comparison}
    Let $j,k$ be integers satisfying $15\Delta^{2}<j\le k\le m$, and let $\sigma\in\mathcal{R}_{\mathbf{d}}(m-k)$. Then, for every $g\in\ell^{2}(\mathcal{R}_{\mathbf{d}}^{(\sigma)}(k),\pi_{k}^{\sigma})$,
        \begin{align}\label{eq:finite-comparison}
            \mathcal{E}_{\downup_{j}^{(\sigma)}}(g,g) &\le C_{j}\mathcal{E}_{\downup_{2}^{(\sigma)}}(g,g).
        \end{align}
\end{lemma}

\begin{proof}
    We first claim that
        \begin{align}\label{eq:finite-local-variance}
            \Var_{\eta\sim\pi_{k}^{\sigma}}(g(\eta)) &\le C_{k}\mathcal{E}_{\downup_{2}^{(\sigma)}}(g,g),
        \end{align}
    for some $C_{k}\ge 1$ depending only on $k$. We will then use~\cref{lem:dirichlet-energy-decomposition}.

    We first give a crude bound on $|\mathcal{R}_{\mathbf{d}}^{(\sigma)}(k)|$. Each $\eta\in \mathcal{R}_{\mathbf{d}}^{(\sigma)}(k)$ consists of exactly $k$ edges such that $\sigma\cup\eta$ is a valid realization of $\mathbf{d}$. Any such completion can be obtained by pairing $2k$ labeled half-edges, with $\res_{\sigma}(v)$ half-edges assigned to each vertex $v\in V$. There are exactly $(2k-1)!!$ partitions of $2k$ objects, in this case half-edges, into pairs, so $|\mathcal{R}_{\mathbf{d}}^{(\sigma)}(k)|\le(2k-1)!!$.

    Next we claim that the graph of the walk $\downup_{2}^{(\sigma)}$ is connected. Suppose otherwise, and choose facets $\eta,\eta'\in\mathcal{R}_{\mathbf{d}}^{(\sigma)}(k)$ in different connected components maximizing $|\eta\cap\eta'|$. Choose an edge $\{v,w\}\in\eta'\setminus\eta$. Since the two completions have the same degrees, by removing their common edges we have that
        \begin{align*}
            \deg_{\eta\setminus\eta'}(z') &= \deg_{\eta'\setminus\eta}(z')
        \end{align*}
    for each $z'\in V$. The chosen edge $\{v,w\}\in\eta'\setminus\eta$ makes the right-hand side positive at both $v$ and $w$. The left-hand side is therefore positive at these vertices as well, so we can choose edges
        \begin{align*}
            \{u,v\},\{w,z\} &\in\eta\setminus\eta'.
        \end{align*}
    Now $\{u,v\}$ makes the left-hand side positive at $u$. Applying the same equality at $u$ gives an edge
        \begin{align*}
            \{u,u'\} &\in\eta'\setminus\eta.
        \end{align*}
    Simplicity gives $u\ne v$ and $z\ne w$. Moreover, $u=w$ or $z=v$ would make one of $\{u,v\},\{w,z\}$ equal to $\{v,w\}$, contradicting their membership in $\eta\setminus\eta'$. Thus $u,z\notin\{v,w\}$, which also implies $\{u,u'\}\ne\{v,w\}$. The equality $u=z$ is allowed. Thus we have a walk $(u',u,v,w,z)$ on at most five vertices. We illustrate the setup so far and the argument which follows in~\cref{fig:auxiliary-edge}. 

    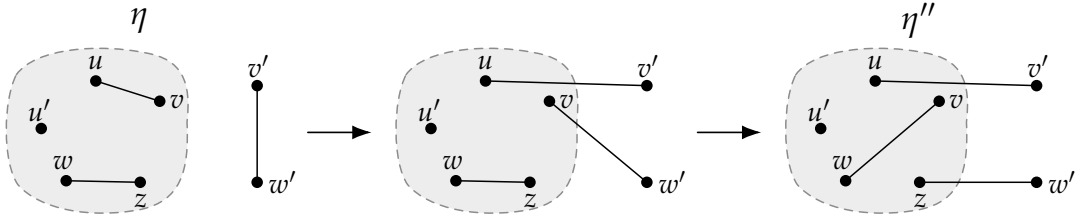
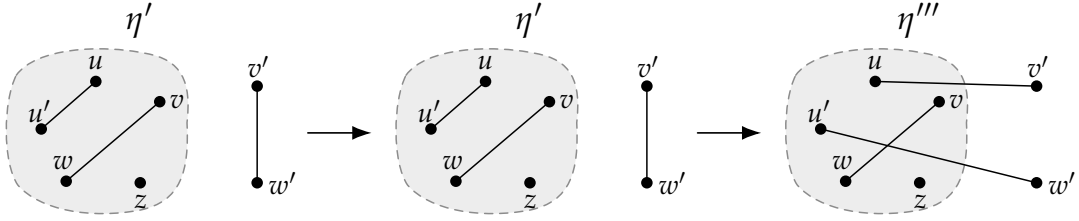
\begin{figure}[t!]
\centering

\def\FigureScale{1.0}

\tikzset{
  ordinary edge/.style={draw=black,line width=.55pt,solid},
  step arrow/.style={draw=black,line width=.6pt,-{Latex[length=2.5mm,width=1.8mm]}},
  forbidden set/.style={fill=black!7,draw=black!45,line width=.55pt,dash pattern=on 3pt off 2.5pt},
  vertex label/.style={font=\fontsize{10}{12}\selectfont,inner sep=0pt},
  state label/.style={font=\fontsize{12}{14}\selectfont,inner sep=0pt}
}

\newcommand{\FigFourVertexPositions}{%
  \coordinate (up) at (-.02,-.03);
  \coordinate (u) at (.68,.58);
  \coordinate (v) at (1.50,.32);
  \coordinate (w) at (.30,-.70);
  \coordinate (z) at (1.25,-.72);
  \coordinate (vp) at (2.75,.52);
  \coordinate (wp) at (2.75,-.72);
}

\newcommand{\FigFourForbiddenBlob}{%
  \path[forbidden set]
    (-.34,.66)
    .. controls (-.09,1.01) and (.74,1.06) .. (1.32,.94)
    .. controls (1.78,.86) and (1.91,.58) .. (1.85,.05)
    .. controls (1.88,-.58) and (1.77,-.96) .. (1.25,-1.07)
    .. controls (.61,-1.19) and (-.04,-1.07) .. (-.32,-.75)
    .. controls (-.52,-.46) and (-.51,.33) .. cycle;
}

\newcommand{\FigFourVerticesAndLabels}{%
  \foreach \q in {up,u,v,w,z,vp,wp} {
    \fill[black] (\q) circle[radius=2.1pt];
  }
  \node[vertex label] at (-.04,.20) {$u'$};
  \node[vertex label] at (.68,.82) {$u$};
  \node[vertex label] at (1.72,.32) {$v$};
  \node[vertex label] at (.26,-.46) {$w$};
  \node[vertex label] at (1.26,-.96) {$z$};
  \node[vertex label] at (2.77,.77) {$v'$};
  \node[vertex label,anchor=west] at (2.89,-.72) {$w'$};
}

\def\FigFourTopInitial{u/v,w/z,vp/wp}
\def\FigFourTopMiddle{u/vp,w/z,v/wp}
\def\FigFourTopFinal{u/vp,v/w,z/wp}
\def\FigFourBottomInitial{u/up,v/w,vp/wp}
\def\FigFourBottomFinal{u/vp,v/w,up/wp}

\newcommand{\FigFourDrawGraph}[3][]{%
  \begin{scope}[#1]
    \FigFourForbiddenBlob
    \FigFourVertexPositions
    \edef\FigFourEdgeList{#2}
    \foreach \a/\b in \FigFourEdgeList {
      \draw[ordinary edge] (\a)--(\b);
    }
    \FigFourVerticesAndLabels
    \node[state label] at (1.24,1.36) {$#3$};
  \end{scope}%
}

\begin{subfigure}{\linewidth}
  \centering
  \resizebox{\FigureScale\linewidth}{!}{%
    \begin{tikzpicture}[x=1cm,y=1cm,line cap=round,line join=round]
      \path[use as bounding box] (-.85,-1.25) rectangle (13.30,1.65);

      \FigFourDrawGraph{\FigFourTopInitial}{\eta}
      \FigFourDrawGraph[shift={(5,0)}]{\FigFourTopMiddle}{}
      \FigFourDrawGraph[shift={(10,0)}]{\FigFourTopFinal}{\eta''}

      \draw[step arrow] (3.40,-.08)--(4.22,-.08);
      \draw[step arrow] (8.40,-.08)--(9.22,-.08);

    \end{tikzpicture}%
  }
  \caption{Two steps of $\downup_2^{(\sigma)}$ from $\eta$ to $\eta''$.}
  \label{fig:auxiliary-edge-top}
\end{subfigure}

\par\medskip

\begin{subfigure}{\linewidth}
  \centering
  \resizebox{\FigureScale\linewidth}{!}{%
    \begin{tikzpicture}[x=1cm,y=1cm,line cap=round,line join=round]
      \path[use as bounding box] (-.85,-1.25) rectangle (13.30,1.65);

      \FigFourDrawGraph{\FigFourBottomInitial}{\eta'}
      \FigFourDrawGraph[shift={(5,0)}]{\FigFourBottomInitial}{\eta'}
      \FigFourDrawGraph[shift={(10,0)}]{\FigFourBottomFinal}{\eta'''}

      \draw[step arrow] (3.40,-.08)--(4.22,-.08);
      \draw[step arrow] (8.40,-.08)--(9.22,-.08);
    \end{tikzpicture}%
  }
  \caption{A step of $\downup_2^{(\sigma)}$ from $\eta'$ to $\eta'''$; the first two panels show the same state.}
  \label{fig:auxiliary-edge-bottom}
\end{subfigure}

\caption{Increasing overlap using an auxiliary edge $\{v',w'\}\in\eta\cap\eta'$. Only the relevant edges are drawn, in the case of seven distinct named vertices; all other edges remain unchanged. The gray regions depict the fixed forbidden set $B$, defined using the original $\sigma\cup\eta\cup\eta'$, with unnamed neighbors omitted. Both auxiliary endpoints lie outside $B$. The final completions lose the common edge $\{v',w'\}$ and gain the common edges $\{v,w\}$ and $\{u,v'\}$, so $|\eta''\cap\eta'''|\ge |\eta\cap\eta'|+1$.}
\label{fig:auxiliary-edge}
\end{figure}

    We use an auxiliary edge to increase the overlap, allowing both completions to move within their components. Let $B$ consist of $u,v,w,z,u'$ and all their neighbors in $\sigma\cup\eta\cup\eta'$. Both $\sigma\cup\eta$ and $\sigma\cup\eta'$ have maximum degree at most $\Delta$, so each vertex has at most $2\Delta$ neighbors in their union. Each of the five specified vertices therefore contributes at most $2\Delta+1$ vertices to $B$, giving
        \begin{align*}
            |B| &\le5(2\Delta+1).
        \end{align*}
    Every edge of $\eta$ meeting $B$ is counted at least once in $\sum_{z'\in B}\deg_{\eta}(z')$, and $\deg_{\eta}(z')\le\Delta$. Consequently,
        \begin{align*}
            |\{\{u'',v''\}\in\eta:\{u'',v''\}\cap B\ne\varnothing\}| &\le\sum_{z'\in B}\deg_{\eta}(z')\\
            &\le\Delta|B|\le5\Delta(2\Delta+1)\\
            &\le15\Delta^{2}<k.
        \end{align*}
    Choose $\{v',w'\}\in\eta$ with $\{v',w'\}\cap B=\varnothing$. In particular,
        \begin{align*}
            \{\{u,v'\},\{v,w'\},\{z,w'\},\{u',w'\}\}\cap(\sigma\cup\eta\cup\eta') &= \varnothing.
        \end{align*}
    In $\eta$, perform the following two steps of $\downup_2^{(\sigma)}$:
        \begin{align*}
            \{\{u,v\},\{w,z\},\{v',w'\}\} &\mapsto\{\{u,v'\},\{w,z\},\{v,w'\}\}\mapsto\{\{u,v'\},\{v,w\},\{z,w'\}\},
        \end{align*}
    obtaining $\eta''$ in the same component as $\eta$. The endpoint exclusions above and the choice of $\{v',w'\}$ ensure that both switches are valid. If $\{v',w'\}\notin\eta'$, these switches remove no common edges and add the common edge $\{v,w\}$, giving
        \begin{align*}
            |\eta''\cap\eta'| &= |\eta\cap\eta'|+1,
        \end{align*}
    contrary to maximality. Thus $\{v',w'\}\in\eta'$. Now perform one switch in $\eta'$:
        \begin{align*}
            \{\{u,u'\},\{v',w'\}\} &\mapsto\{\{u,v'\},\{u',w'\}\},
        \end{align*}
    obtaining $\eta'''$ in the same component as $\eta'$. This switch is valid by the choice of $\{v',w'\}$. Every original common edge except $\{v',w'\}$ remains common, while $\{v,w\}$ and $\{u,v'\}$ are now both common. Therefore
        \begin{align*}
            |\eta''\cap\eta'''| &\ge|\eta\cap\eta'|-1+2>|\eta\cap\eta'|.
        \end{align*}
    Since $\eta''$ and $\eta'''$ remain in different components, this contradicts maximality and proves the claim.

    The preceding argument proves something stronger and useful. Starting from any two completions, repeat the construction until they coincide. Each application uses two valid steps from the first completion and at most one from the second, increasing their overlap by at least one. Since both completions have $k$ edges, at most $k$ applications are needed. Following the switches from the first completion and then reversing those from the second gives a walk of length at most $3k$ between them.

    We now prove~\cref{eq:finite-local-variance} using the crude bound on $|\mathcal{R}_{\mathbf{d}}^{(\sigma)}(k)|$ and the preceding path bound. If $|\mathcal{R}_{\mathbf{d}}^{(\sigma)}(k)|=1$, the variance is zero. Otherwise, choose a simple path $\eta_{0},\dotsc,\eta_{\ell}$ from a minimum of $g$ to a maximum, with $\ell\le3k$. Telescoping and Cauchy--Schwarz give
        \begin{align*}
            \Var_{\eta\sim\pi_{k}^{\sigma}}(g(\eta)) &\le(g(\eta_{\ell})-g(\eta_{0}))^{2}\\
            &= \Bigl(\sum_{i=0}^{\ell-1}(g(\eta_{i+1})-g(\eta_{i}))\Bigr)^{2}\\
            &\le\ell\sum_{i=0}^{\ell-1}(g(\eta_{i+1})-g(\eta_{i}))^{2}.
        \end{align*}
    By the preceding and the fact that $\downup_{2}^{(\sigma)}$ is reversible with respect to the uniform measure on $\mathcal{R}_{\mathbf{d}}^{(\sigma)}(k)$, it follows that
        \begin{align*}
            \mathcal{E}_{\downup_{2}^{(\sigma)}}(g,g) &= \frac{1}{2|\mathcal{R}_{\mathbf{d}}^{(\sigma)}(k)|}\sum_{\eta,\eta'\in\mathcal{R}_{\mathbf{d}}^{(\sigma)}(k)}\downup_{2}^{(\sigma)}(\eta,\eta')(g(\eta')-g(\eta))^{2}\\
            &\ge\frac{1}{|\mathcal{R}_{\mathbf{d}}^{(\sigma)}(k)|}\sum_{i=0}^{\ell-1}\downup_{2}^{(\sigma)}(\eta_{i},\eta_{i+1})(g(\eta_{i+1})-g(\eta_{i}))^{2}\\
            &\ge\frac{1}{3\binom{k}{2}|\mathcal{R}_{\mathbf{d}}^{(\sigma)}(k)|}\sum_{i=0}^{\ell-1}(g(\eta_{i+1})-g(\eta_{i}))^{2}.
        \end{align*}
    In the last line, we use the basic bound
        \begin{align*}
            \downup_{2}^{(\sigma)}(\eta_{i},\eta_{i+1}) &\ge\frac{1}{3\binom{k}{2}},
        \end{align*}
    since the walk chooses one of $\binom{k}{2}$ edge pairs and then one of at most three completions. Combining these inequalities,
        \begin{align*}
            \Var_{\eta\sim\pi_{k}^{\sigma}}(g(\eta)) &\le3\binom{k}{2}|\mathcal{R}_{\mathbf{d}}^{(\sigma)}(k)|\ell\,\mathcal{E}_{\downup_{2}^{(\sigma)}}(g,g)\\
            &\le9k\binom{k}{2}(2k-1)!!\,\mathcal{E}_{\downup_{2}^{(\sigma)}}(g,g).
        \end{align*}
    This proves~\cref{eq:finite-local-variance}.

    Finally, for each $\rho\in\mathcal{R}_{\mathbf{d}}^{(\sigma)}(k-j)$ and $\eta\in\mathcal{R}_{\mathbf{d}}^{(\sigma\cup\rho)}(j)$, write $g^{(\rho)}(\eta)=g(\rho\cup\eta)$. Applying~\cref{eq:finite-local-variance} with $k$ replaced by $j$ and retained face $\sigma\cup\rho$, and using~\cref{lem:dirichlet-energy-decomposition}(ii), we obtain
        \begin{align*}
            &\mathcal{E}_{\downup_{j}^{(\sigma)}}(g,g)\\
            &\quad= \E_{\rho\sim\pi_{k-j}^{\sigma}}\Bigl(\Var_{\eta\sim\pi_{j}^{\sigma\cup\rho}}(g^{(\rho)}(\eta))\Bigr)\\
            &\quad\le C_{j}\E_{\rho\sim\pi_{k-j}^{\sigma}}\Bigl(\mathcal{E}_{\downup_{2}^{(\sigma\cup\rho)}}(g^{(\rho)},g^{(\rho)})\Bigr)\\
            &\quad= C_{j}\mathcal{E}_{\downup_{2}^{(\sigma)}}(g,g).
        \end{align*}
    This proves~\cref{eq:finite-comparison}.
\end{proof}

In~\cref{eq:finite-comparison}, one may take, e.g., $C_{j}=9j\binom{j}{2}(2j-1)!!$. In principle, combining~\cref{lem:higher-order-poincare,lem:finite-comparison} with $k=m$ already gives a bound on the spectral gap of $\downup_{2}$. However, the superpolynomial growth of $C_{j}$ limits the range of maximum degrees for which this bound is informative. We will improve this dependence through a good set--bad set argument that uses the finite comparison only for bounded degree sequences.

To this end, let $k_{0}=10^{10}$ and $C_{10}=2C_{k_{0}}/k_{0}$. Whenever $\Delta\le10$ and $k_{0}\le k\le m$,~\cref{lem:higher-order-poincare} and~\cref{lem:finite-comparison} give, for every $\sigma\in\mathcal{R}_{\mathbf{d}}(m-k)$ and $g\in\ell^{2}(\mathcal{R}_{\mathbf{d}}^{(\sigma)}(k),\pi_{k}^{\sigma})$,
    \begin{align}\label{eq:degree-ten-poincare}
        \Var_{\eta\sim\pi_{k}^{\sigma}}(g(\eta)) &\le\frac{2k}{k_{0}}\mathcal{E}_{\downup_{k_{0}}^{(\sigma)}}(g,g) \le C_{10}k\,\mathcal{E}_{\downup_{2}^{(\sigma)}}(g,g).
    \end{align}
The constant $C_{10}$ is independent of the degree sequence, retained face, and completion size.

For any face $\sigma\in\mathcal{R}_{\mathbf{d}}$, define the set of \emph{active vertices} by $V_{\sigma}=\{v\in V:\res_{\sigma}(v)>0\}$, and let $\widetilde{\sigma} = \sigma[V_{\sigma}]$ consist of all edges belonging to $\sigma$ with both endpoints in $V_{\sigma}$. Write $\widetilde{\mathbf{d}}=(\widetilde{d}_{v})_{v\in V_{\sigma}}$, where the \emph{modified residual degrees} $\widetilde{d}_{v}$ are given by
    \begin{align*}
        \widetilde{d}_{v} &= \res_{\sigma}(v)+\deg_{\widetilde{\sigma}}(v),
    \end{align*}
for each $v\in V_{\sigma}$. That is, $\widetilde{d}_{v}$ can be viewed as the target degree that $v$ must have in any completed realization restricted to the active vertex set $V_{\sigma}$. 

Since all residual degrees outside $V_{\sigma}$ vanish, the completions of $\sigma$ are exactly the edge sets $\eta$ on $V_{\sigma}$ that avoid $\widetilde{\sigma}$ and satisfy
    \begin{align*}
        \deg_{\eta}(v) &= \res_{\sigma}(v)=\widetilde{d}_{v}-\deg_{\widetilde{\sigma}}(v),
    \end{align*}
for each $v\in V_{\sigma}$. Thus $\widetilde{\sigma}\cup\eta$ realizes $\widetilde{\mathbf{d}}$, which implies that $\widetilde{\mathbf{d}}$ is graphical. Conversely, the same conditions characterize the completions of $\widetilde{\sigma}$ for $\widetilde{\mathbf{d}}$. Taking subsets gives the following equality of simplicial complexes:
    \begin{align}\label{eq:identical-links}
        \mathcal{R}_{\mathbf{d}}^{(\sigma)} &= \mathcal{R}_{\widetilde{\mathbf{d}}}^{(\widetilde{\sigma})}.
    \end{align}
The following useful fact is the key to our good set--bad set argument in proving an improved energy comparison.

\begin{lemma}\label{lem:good-restriction-probability}
    Let $k,\ell$ be integers satisfying $k\le m$ and $10^{6}\Delta^{4}\le\ell\le\sqrt{k}$, and let $\sigma\in\mathcal{R}_{\mathbf{d}}(m-k)$. Let $\mathcal{A}\subseteq\mathcal{R}_{\mathbf{d}}^{(\sigma)}(k-\ell)$ consist of those partial completions $\rho$ such that
        \begin{align*}
            \res_{\sigma\cup\rho}(v)+\deg_{\widetilde{\sigma\cup\rho}}(v) &\le10
        \end{align*}
    for each $v\in V_{\sigma\cup\rho}$. For each completion $\eta\in\mathcal{R}_{\mathbf{d}}^{(\sigma)}(k)$, let $\mu_{\eta}$ be the uniform probability measure on the $(k-\ell)$-edge subsets of $\eta$. Then
        \begin{align}\label{eq:bad-probability}
            \Prob_{\rho\sim\mu_{\eta}}(\rho\notin\mathcal{A}) &\le\frac{\ell}{4k}.
        \end{align}
\end{lemma}

\begin{proof}
    Fix $\eta\in\mathcal{R}_{\mathbf{d}}^{(\sigma)}(k)$, so that only $\rho\sim\mu_{\eta}$ is random. For every such $\rho$,
        \begin{align*}
            \res_{\sigma\cup\rho}(v) &= \res_{\sigma}(v)-\deg_{\rho}(v)=\deg_{\eta\setminus\rho}(v),\qquad v\in V.
        \end{align*}
    Thus $V_{\sigma\cup\rho}$ consists exactly of the endpoints of the $\ell$ edges in $\eta\setminus\rho$. For each $v\in V_{\sigma\cup\rho}$, the modified residual degree is therefore
        \begin{align*}
            \res_{\sigma\cup\rho}(v)+\deg_{\widetilde{\sigma\cup\rho}}(v) &= \deg_{\eta\setminus\rho}(v)+\deg_{\widetilde{\sigma\cup\rho}}(v)\\
            &= \deg_{(\sigma\cup\eta)[V_{\sigma\cup\rho}]}(v).
        \end{align*}
    Consequently, if $\rho\notin\mathcal{A}$, there exists some $v\in V_{\sigma\cup\rho}$ with at least $11$ active neighbors in the fixed realization $\sigma\cup\eta$.

    For each $v\in V_{\sigma}$, let $B_{v}$ be the closed neighborhood of $v$ in $\sigma\cup\eta$. Then $|B_{v}|\le\Delta+1$. If $\rho\notin\mathcal{A}$, there is some $v\in V_{\sigma\cup\rho}\subseteq V_{\sigma}$ for which
        \begin{align*}
            |B_{v}\cap V_{\sigma\cup\rho}| &\ge12.
        \end{align*}
    Each of these vertices is incident to an edge of $\eta\setminus\rho$, and one edge accounts for at most two vertices. Thus at least six distinct edges of $\eta\setminus\rho$ meet $B_{v}$. These edges need not be vertex-disjoint. The degree cutoff $10$ is what forces these six distinct edges.

    For a fixed center $v$, the number of edges of $\eta$ meeting $B_{v}$ is at most
        \begin{align*}
            |\{\{u,w\}\in\eta:\{u,w\}\cap B_{v}\ne\varnothing\}| &\le\sum_{u\in B_{v}}\deg_{\eta}(u) \le\Delta(\Delta+1) \le2\Delta^{2}.
        \end{align*}
    The set $\eta\setminus\rho$ is uniform among the $\ell$-edge subsets of $\eta$. Hence, for any fixed $\tau\subseteq\eta$ with $|\tau|=6$,
        \begin{align*}
            \Prob_{\rho\sim\mu_{\eta}}(\tau\subseteq\eta\setminus\rho) &= \frac{\binom{k-6}{\ell-6}}{\binom{k}{\ell}} = \frac{\ell}{k}\frac{\ell-1}{k-1}\cdots\frac{\ell-5}{k-5} \le\Bigl(\frac{\ell}{k}\Bigr)^{6}.
        \end{align*}
    The set $V_{\sigma}$ consists of the endpoints of the $k$ edges of $\eta$, so there are at most $2k$ possible centers. Taking a union bound over these centers and their six-edge subsets, and using $k\ge\ell^{2}$, gives
        \begin{align*}
            \Prob_{\rho\sim\mu_{\eta}}(\rho\notin\mathcal{A}) &\le2k\binom{2\Delta^{2}}{6}\Bigl(\frac{\ell}{k}\Bigr)^{6} \le\frac{128\Delta^{12}\ell^{6}}{k^{5}} \le\frac{128\Delta^{12}}{\ell^{3}}\frac{\ell}{k} \le\frac{128}{10^{18}}\frac{\ell}{k} \le\frac{\ell}{4k}.
        \end{align*}
    Here $\binom{a}{6}=0$ when $a<6$, and the penultimate inequality uses $\ell\ge10^{6}\Delta^{4}$. The completion $\eta$ was arbitrary, so~\cref{eq:bad-probability} holds for every starting completion.
\end{proof}

\begin{proposition}\label{prop:linear-comparison}
    There exist absolute constants $C,C'>0$ with the following property. Let $k$ be an integer satisfying $C\Delta^{8}\le k\le m$. Then, for every $f\in\ell^{2}(\mathcal{R}_{\mathbf{d}}(m),\pi_{m})$,
        \begin{align}\label{eq:linear-comparison}
            \mathcal{E}_{\downup_{k}}(f,f) &\le C'k\,\mathcal{E}_{\downup_{2}}(f,f).
        \end{align}
\end{proposition}

\begin{proof}
    Set $C=k_{0}^{2}$, and let $f\in\ell^{2}(\mathcal{R}_{\mathbf{d}}(m),\pi_{m})$. For each $\sigma\in\mathcal{R}_{\mathbf{d}}(m-k)$, write $f^{(\sigma)}(\eta)=f(\sigma\cup\eta)$ for $\eta\in\mathcal{R}_{\mathbf{d}}^{(\sigma)}(k)$. We first decompose the energy according to the retained $(m-k)$-face. By~\cref{lem:dirichlet-energy-decomposition}(ii), applied with the empty retained face, we have
        \begin{align*}
            \mathcal{E}_{\downup_{k}}(f,f) &= \E_{\sigma\sim\pi_{m-k}}\Bigl(\Var_{\eta\sim\pi_{k}^{\sigma}}(f^{(\sigma)}(\eta))\Bigr).
        \end{align*}
    We will bound each variance term appearing inside the mean above by the corresponding Dirichlet energy with respect to ${\downup_{2}^{(\sigma)}}$. These local energies average to the two-edge energy on the full complex: by~\cref{lem:dirichlet-energy-decomposition}(iii),
        \begin{align*}
            \E_{\sigma\sim\pi_{m-k}}\Bigl(\mathcal{E}_{\downup_{2}^{(\sigma)}}(f^{(\sigma)},f^{(\sigma)})\Bigr) &= \mathcal{E}_{\downup_{2}}(f,f).
        \end{align*}
    Thus it suffices to prove a uniform variance bound within each link. Fix $\sigma\in\mathcal{R}_{\mathbf{d}}(m-k)$ and $g\in\ell^{2}(\mathcal{R}_{\mathbf{d}}^{(\sigma)}(k),\pi_{k}^{\sigma})$. We will show that
        \begin{align}\label{eq:target-local-poincare}
            \Var_{\eta\sim\pi_{k}^{\sigma}}(g(\eta)) &\le4C_{10}k\,\mathcal{E}_{\downup_{2}^{(\sigma)}}(g,g).
        \end{align}
    Subtracting the mean of $g$ changes neither its variance nor its Dirichlet energies, so we may assume freely that $\E_{\eta\sim\pi_{k}^{\sigma}}(g(\eta))= 0$. Let $\ell=k_{0}\Delta^{4}$ be the number of edges left free after the additional restriction. Our assumptions give
        \begin{align}\label{eq:restriction-parameters}
            10^{6}\Delta^{4} &\le\ell\le\sqrt{k},
        \end{align}
    so~\cref{lem:higher-order-poincare} yields
        \begin{align}\label{eq:restriction-poincare}
            \Var_{\eta\sim\pi_{k}^{\sigma}}(g(\eta)) &\le\frac{2k}{\ell}\mathcal{E}_{\downup_{\ell}^{(\sigma)}}(g,g).
        \end{align}
    For each $\rho\in\mathcal{R}_{\mathbf{d}}^{(\sigma)}(k-\ell)$, write $g^{(\rho)}(\eta')=g(\rho\cup\eta')$ for $\eta'\in\mathcal{R}_{\mathbf{d}}^{(\sigma\cup\rho)}(\ell)$. Let $\mathcal{A}$ be the set in~\cref{lem:good-restriction-probability}, consisting of those faces $\rho$ for which the modified residual degrees satisfy
        \begin{align}\label{eq:good-restrictions}
            \res_{\sigma\cup\rho}(v)+\deg_{\widetilde{\sigma\cup\rho}}(v) &\le10,\qquad v\in V_{\sigma\cup\rho}.
        \end{align}
    This condition depends only on the retained face $\sigma\cup\rho$. By~\cref{lem:dirichlet-energy-decomposition}(iii), we can split the energy according to whether this condition holds:
        \begin{align}\label{eq:restriction-energy-split}
            \mathcal{E}_{\downup_{\ell}^{(\sigma)}}(g,g) &= \E_{\rho\sim\pi_{k-\ell}^{\sigma}}\Bigl(\mathbf{1}_{\rho\in\mathcal{A}}\Var_{\eta'\sim\pi_{\ell}^{\sigma\cup\rho}}(g^{(\rho)}(\eta'))\Bigr)\notag\\
            &\qquad+\E_{\rho\sim\pi_{k-\ell}^{\sigma}}\Bigl(\mathbf{1}_{\rho\notin\mathcal{A}}\Var_{\eta'\sim\pi_{\ell}^{\sigma\cup\rho}}(g^{(\rho)}(\eta'))\Bigr).
        \end{align}

    We first bound the contribution from the good set $\mathcal{A}$. For each $\rho\in\mathcal{A}$,~\cref{eq:identical-links} and the surrounding discussion identify $\mathcal{R}_{\mathbf{d}}^{(\sigma\cup\rho)}$ with a link of a switch complex whose prescribed degrees are positive and at most $10$. This identification preserves the uniform facet measure and the down-up walks, and the facets still have size $\ell$. Since $\ell\ge k_{0}$, it follows that
        \begin{align*}
            \Var_{\eta'\sim\pi_{\ell}^{\sigma\cup\rho}}(g^{(\rho)}(\eta')) &\le C_{10}\ell\,\mathcal{E}_{\downup_{2}^{(\sigma\cup\rho)}}(g^{(\rho)},g^{(\rho)}).
        \end{align*}
    Averaging over $\rho\in\mathcal{A}$, using nonnegativity of the energies, and applying~\cref{lem:dirichlet-energy-decomposition}(iii), we obtain
        \begin{align}\label{eq:averaged-good-bound}
            &\E_{\rho\sim\pi_{k-\ell}^{\sigma}}\Bigl(\mathbf{1}_{\rho\in\mathcal{A}}\Var_{\eta'\sim\pi_{\ell}^{\sigma\cup\rho}}(g^{(\rho)}(\eta'))\Bigr)\notag\\
            &\quad\le C_{10}\ell\,\E_{\rho\sim\pi_{k-\ell}^{\sigma}}\Bigl(\mathbf{1}_{\rho\in\mathcal{A}}\mathcal{E}_{\downup_{2}^{(\sigma\cup\rho)}}(g^{(\rho)},g^{(\rho)})\Bigr)\notag\\
            &\quad\le C_{10}\ell\,\mathcal{E}_{\downup_{2}^{(\sigma)}}(g,g).
        \end{align}

    For the second term appearing on the right-hand side of~\cref{eq:restriction-energy-split}, we bound the variance by the second moment. For each completion $\eta\in\mathcal{R}_{\mathbf{d}}^{(\sigma)}(k)$, let $\mu_{\eta}$ be uniform on its $(k-\ell)$-edge subsets. Expanding the second moment and using~\cref{lem:dirichlet-energy-decomposition}(i), we have
        \begin{align}\label{eq:bad-second-moment}
            &\E_{\rho\sim\pi_{k-\ell}^{\sigma}}\Bigl(\mathbf{1}_{\rho\notin\mathcal{A}}\Var_{\eta'\sim\pi_{\ell}^{\sigma\cup\rho}}(g^{(\rho)}(\eta'))\Bigr)\notag\\
            &\quad\le\sum_{\rho\in\mathcal{R}_{\mathbf{d}}^{(\sigma)}(k-\ell)}\Bigl(\pi_{k-\ell}^{\sigma}(\rho)\mathbf{1}_{\rho\notin\mathcal{A}}\sum_{\eta'\in\mathcal{R}_{\mathbf{d}}^{(\sigma\cup\rho)}(\ell)}\pi_{\ell}^{\sigma\cup\rho}(\eta')g^{(\rho)}(\eta')^{2}\Bigr)\notag\\
            &\quad=\sum_{\eta\in\mathcal{R}_{\mathbf{d}}^{(\sigma)}(k)}\Bigl(\pi_{k}^{\sigma}(\eta)g(\eta)^{2} \sum_{\substack{\rho\subseteq\eta\\|\rho|=k-\ell}}\frac{\mathbf{1}_{\rho\notin\mathcal{A}}}{\binom{k}{k-\ell}}\Bigr)\notag\\
            &\quad=\E_{\eta\sim\pi_{k}^{\sigma}}\Bigl(g(\eta)^{2}\Prob_{\rho\sim\mu_{\eta}}(\rho\notin\mathcal{A})\Bigr).
        \end{align}
    By~\cref{eq:restriction-parameters}, the hypotheses of~\cref{lem:good-restriction-probability} hold. Applying its probability bound inside the expectation in~\cref{eq:bad-second-moment} and using the fact that $g$ is centered, we obtain
        \begin{align}\label{eq:averaged-bad-bound}
            \E_{\rho\sim\pi_{k-\ell}^{\sigma}}\Bigl(\mathbf{1}_{\rho\notin\mathcal{A}}\Var_{\eta'\sim\pi_{\ell}^{\sigma\cup\rho}}(g^{(\rho)}(\eta'))\Bigr) & \le\frac{\ell}{4k}\E_{\eta\sim\pi_{k}^{\sigma}}(g(\eta)^{2}) = \frac{\ell}{4k}\Var_{\eta\sim\pi_{k}^{\sigma}}(g(\eta)).
        \end{align}

    Combining~\cref{eq:restriction-poincare,eq:restriction-energy-split,eq:averaged-good-bound,eq:averaged-bad-bound}, we obtain
        \begin{align*}
            \Var_{\eta\sim\pi_{k}^{\sigma}}(g(\eta)) &\le\frac{2k}{\ell}\mathcal{E}_{\downup_{\ell}^{(\sigma)}}(g,g) \le2C_{10}k\,\mathcal{E}_{\downup_{2}^{(\sigma)}}(g,g)+\frac{1}{2}\Var_{\eta\sim\pi_{k}^{\sigma}}(g(\eta)).
        \end{align*}
    Subtracting the last term and multiplying by $2$ proves~\cref{eq:target-local-poincare}. 

    Finally, applying~\cref{eq:target-local-poincare} with $g=f^{(\sigma)}$ for each $\sigma$ and using the two energy decompositions from the start of the proof, we obtain
        \begin{align*}
            \mathcal{E}_{\downup_{k}}(f,f) &= \E_{\sigma\sim\pi_{m-k}}\Bigl(\Var_{\eta\sim\pi_{k}^{\sigma}}(f^{(\sigma)}(\eta))\Bigr)\\
            &\le4C_{10}k\,\E_{\sigma\sim\pi_{m-k}}\Bigl(\mathcal{E}_{\downup_{2}^{(\sigma)}}(f^{(\sigma)},f^{(\sigma)})\Bigr)\\
            &=4C_{10}k\,\mathcal{E}_{\downup_{2}}(f,f).
        \end{align*}
    Thus $C=k_{0}^{2}$ and $C'=4C_{10}$ satisfy~\cref{eq:linear-comparison}. The proposition is proved.
\end{proof}

\Cref{thm:intro-linear-comparison} now follows from~\cref{prop:linear-comparison}.

We may now prove the main result~\cref{thm:two-edge-eigenvalue} stated at the beginning of this section.

\begin{proof}[Proof of~\cref{thm:two-edge-eigenvalue}]
    We first bound $\lambda_{2}(\downup_{2})$ from above. Taking $C \ge 6$ sufficiently large, we may apply~\cref{prop:linear-comparison} with $k=m$ to obtain
        \begin{align*}
            \mathcal{E}_{\downup_{2}}(f,f) &\ge\frac{1}{Cm}\mathcal{E}_{\downup_{m}}(f,f)=\frac{1}{Cm}\Var_{\eta\sim\pi_{m}}(f(\eta))
        \end{align*}
    for every $f\in\ell^{2}(\mathcal{R}_{\mathbf{d}}(m),\pi_{m})$. Here the equality follows because $\downup_{m}$ resamples the entire realization. By the variational characterization of the spectral gap, this gives the desired upper bound on $\lambda_{2}(\downup_{2})$.

    For the reverse bound, fix $\{u,v\}\in\mathcal{R}_{\mathbf{d}}(1)$ and write $p=\Prob_{\eta\sim\pi_{m}}(\{u,v\}\in\eta)$. By \cref{lem:uniform-projection-lift}\emph{(i)} and~\cref{lem:single-pair-marginals}, applied with $\sigma=\varnothing$ and $k=m$, summing over the $d_{u}d_{v}$ stub pairs projecting to $\{u,v\}$ gives
        \begin{align*}
            0<p &\le\frac{d_{u}d_{v}}{2m-2-4\Delta^{2}}\le\frac{\Delta^{2}}{2m-2-4\Delta^{2}}\le\frac{1}{3},
        \end{align*}
    where the last inequality uses $m\ge6\Delta^{2}$. Define $g(\eta)=\mathbf{1}_{\{\{u,v\}\in\eta\}}$, so that $\Var_{\eta\sim\pi_{m}}(g(\eta))=p(1-p)$.

    An edge present in the current realization can disappear only if it is among the two deleted edges, which occurs with probability $2/m$. After deletion there are at most three valid completions, including the original pair, so the conditional probability of losing $\{u,v\}$ is at most $2/3$. By reversibility,
        \begin{align*}
            \mathcal{E}_{\downup_{2}}(g,g) &=\sum_{\substack{\eta\in\mathcal{R}_{\mathbf{d}}(m)\\\{u,v\}\in\eta}}\pi_{m}(\eta)\sum_{\substack{\tau\in\mathcal{R}_{\mathbf{d}}(m)\\\{u,v\}\notin\tau}}\downup_{2}(\eta,\tau) \le p\frac{2}{m}\frac{2}{3}=\frac{4p}{3m}.
        \end{align*}
    Applying the variational characterization once more, we obtain
        \begin{align*}
            1-\lambda_{2}(\downup_{2}) &\le\frac{\mathcal{E}_{\downup_{2}}(g,g)}{\Var_{\eta\sim\pi_{m}}(g(\eta))}\le\frac{4}{3m(1-p)}\le\frac{2}{m}.
        \end{align*}
    This proves the remaining bound.
\end{proof}

\Cref{thm:intro-two-edge-eigenvalue} now follows from~\cref{thm:two-edge-eigenvalue}.

\section{Mixing times of simplicial switch chains}\label{sec:mixing-times-simplicial}

In this section we prove a simple bound on the total variation mixing time of the simplicial switch chain. Our main result is the following theorem.

\begin{theorem}\label{thm:two-edge-mixing}
    There exists an absolute constant $C>0$ with the following property. If $m\ge C\Delta^{8}$, then
        \begin{align*}
            t_{\mathrm{mix}}(\downup_{2},1/4) &\le C\Delta^{2}m\log(m).
        \end{align*}
\end{theorem}

We will prove~\cref{thm:two-edge-mixing} by applying the following generic mixing time bound for positive semidefinite Markov kernels due to Salez~\cite{salez2024cutoff}. We rework the statement of the original lemma slightly to focus on the setting of discrete time processes, and impose stronger reversibility assumptions for convenience. For completeness, we provide a proof.

\begin{lemma}[{\cite[Lem.~7]{salez2024cutoff}}]\label{lem:salez-entropy-gap}
    Let $\mathbf{Q}$ be a Markov kernel on a finite set $\mathcal{Y}$ with at least two elements and stationary measure $\pi$. Let $r\ge0$ be an integer and let $K\ge0$. Suppose that $\mathbf{Q}$ is irreducible, reversible, positive semidefinite, and that
        \begin{align*}
            D(\mathbf{Q}^{r}(x,\cdot)\,\Vert\,\pi) &\le K
        \end{align*}
    for every $x\in\mathcal{Y}$. Then
        \begin{align*}
            t_{\mathrm{mix}}(\mathbf{Q},1/4) &\le r+ \frac{4(K+1)}{1-\lambda_{2}(\mathbf{Q})}.
        \end{align*}
\end{lemma}

\begin{proof}
    If $K=0$, then $\mathbf{Q}^{r}(x,\cdot)=\pi$ for every $x$, and the conclusion follows, so we assume $K>0$. Fix $x\in\mathcal{Y}$ and write $\mu(y)=\mathbf{Q}^{r}(x,y)/\pi(y)$. Set $A=\{y\in\mathcal{Y}:\mu(y)< e^{1+8K}\}$. The function $u\mapsto u\log u - u + 1$ is nonnegative on $u\ge 0$ and satisfies $u\log u-u+1\ge 8Ku$ for each $u\ge e^{1+8K}$. Thus
        \begin{align*}
            8K\sum_{y\notin A}\pi(y)\mu(y) &\le\sum_{y\notin A}\pi(y)(\mu(y)\log \mu(y)-\mu(y)+1)\\
            &\le\sum_{y\in\mathcal{Y}}\pi(y)(\mu(y)\log \mu(y)-\mu(y)+1)\\
            &=D(\mathbf{Q}^{r}(x,\cdot)\,\Vert\,\pi)\le K.
        \end{align*}
    Consequently $a=\sum_{y\in A}\pi(y)\mu(y)\ge7/8$. Define $\nu(y)=\mu(y)\mathbf{1}_{\{y\in A\}}/a$. Then $\E_{y\sim\pi}(\nu(y))=1$, and
        \begin{align*}
            \frac{1}{2}\sum_{y\in\mathcal{Y}}\pi(y)|\mu(y)-\nu(y)| &=\frac{1}{2}(a(\frac{1}{a}-1)+1-a)=1-a\le\frac{1}{8}.
        \end{align*}
    Additionally, we have
        \begin{align*}
            \Var_{y\sim\pi}(\nu(y)) &\le\sum_{y\in\mathcal{Y}}\pi(y)\nu(y)^{2}\le \frac{e^{1+8K}}{a} \sum_{y\in\mathcal{Y}}\pi(y)\nu(y)\le\frac{8}{7}e^{1+8K}\le e^{2+8K}.
        \end{align*}
    For each integer $s\ge1$, the triangle inequality and stationarity of $\pi$ give
        \begin{align}\label{eq:triangle-inequality-h-h0}
            \sum_{y\in\mathcal{Y}}\pi(y)|\mathbf{Q}^{s}(\mu-\nu)(y)| &\le\sum_{y,z\in\mathcal{Y}}\pi(y)\mathbf{Q}^{s}(y,z)|\mu(z)-\nu(z)|\notag\\
            &=\sum_{z\in\mathcal{Y}}\pi(z)|\mu(z)-\nu(z)|\le\frac{1}{4}.
        \end{align}
    By reversibility,
        \begin{align*}
            \pi(y)(\mathbf{Q}^{s}\mu)(y) &=\sum_{z\in\mathcal{Y}}\pi(z)\mathbf{Q}^{s}(z,y)\mu(z)=\mathbf{Q}^{r+s}(x,y),
        \end{align*}
    and therefore
        \begin{align*}
            \|\mathbf{Q}^{r+s}(x,\cdot)-\pi\|_{\mathrm{TV}} &=\frac{1}{2}\sum_{y\in\mathcal{Y}} \pi(y) |(\mathbf{Q}^{s} \mu)(y) - 1|\\
            &\le \frac{1}{2}\E_{y\sim\pi}(|(\mathbf{Q}^{s}\mu)(y) - (\mathbf{Q}^{s}\nu)(y)|) + \frac{1}{2}\E_{y\sim\pi}(|(\mathbf{Q}^{s}\nu)(y) - 1|)\\
            &\le \frac{1}{8} + \frac{1}{2}\Bigl(\sum_{y\in\mathcal{Y}}\pi(y) |(\mathbf{Q}^{s}\nu)(y) - 1|^{2} \Bigr)^{1/2},
        \end{align*}
    where the last inequality follows from~\cref{eq:triangle-inequality-h-h0} and the Cauchy--Schwarz inequality. By reversibility and positive semidefiniteness, $\mathbf{Q}$ has an $\ell^{2}(\mathcal{Y},\pi)$-orthonormal basis of eigenfunctions with eigenvalues in $[0,1]$. Since $\nu-\mathbf{1}$ is orthogonal to the first such eigenfunction with eigenvalue $\lambda_{1}(\mathbf{Q}) = 1$, it follows that
        \begin{align*}
            \sum_{y\in\mathcal{Y}}\pi(y) |(\mathbf{Q}^{s}\nu)(y) - 1|^{2} & \le \lambda_{2}(\mathbf{Q})^{2s}\|\nu-\mathbf{1}\|_{\pi}^{2} = \lambda_{2}(\mathbf{Q})^{2s}\Var_{y\sim\pi}(\nu(y)).
        \end{align*}
    Combining the preceding bounds, we obtain
        \begin{align*}
            \|\mathbf{Q}^{r+s}(x,\cdot)-\pi\|_{\mathrm{TV}} &\le \frac{1}{8} + \frac{1}{2}\lambda_{2}(\mathbf{Q})^{s}\sqrt{\Var_{y\sim\pi}(\nu(y))}\\
            &\le\frac{1}{8}+\frac{1}{2}\exp(1+4K-s(1-\lambda_{2}(\mathbf{Q}))).
        \end{align*}
    Taking $s=\lfloor4(K+1)/(1-\lambda_{2}(\mathbf{Q}))\rfloor$, we have $s(1-\lambda_{2}(\mathbf{Q}))\ge4(K+1)-(1-\lambda_{2}(\mathbf{Q}))\ge4K+3$, since $0<1-\lambda_{2}(\mathbf{Q})\le1$. The last expression is therefore at most $1/8+e^{-2}/2<1/4$, uniformly in $x$. Since $s\le4(K+1)/(1-\lambda_{2}(\mathbf{Q}))$, this proves the lemma.
\end{proof}

To prove~\cref{thm:two-edge-mixing} via~\cref{thm:two-edge-eigenvalue} and~\cref{lem:salez-entropy-gap}, it remains to certify that after order $m\log{m}$ steps, the entropy of $\downup_2^{r}(\sigma,\cdot)$ relative to the uniform measure is at most order $\Delta^{2}\log{m}$ uniformly in $\sigma$. This input is supplied by the following proposition.

\begin{proposition}\label{prop:two-edge-entropy-bound}
    Suppose that $m\ge2$, and let $r$ be an integer satisfying $r\ge3m\log(2m)$. Then
        \begin{align*}
            \max_{\sigma\in\mathcal{R}_{\mathbf{d}}(m)}D(\downup_{2}^{r}(\sigma,\cdot)\,\Vert\,\pi_{m}) &\le193\Delta^{2}\log(2m).
        \end{align*}
\end{proposition}

The basic idea behind the proof of~\cref{prop:two-edge-entropy-bound} is to lift the down-$2$-up-$2$ walk to a corresponding walk on valid pairings and compare it to an unrestricted transposition walk on all pairings. Using the weak triangle inequality, we can then split the relative entropy into separate pieces corresponding to the cost of the original comparison to the unrestricted walk and the gap between the unrestricted walk and its equilibrium measure. These can then be handled separately; the former case boils down to controlling the number of invalid transpositions, for which we have a number of tools at this point, and the latter case can be handled by existing results concerning walks on the symmetric group. Namely, we use the following theorem of Diaconis and Saloff-Coste.

\begin{theorem}[{\cite[Sec.~4.3, Eq.~(4.2)]{diaconis1996logarithmic}}]\label{thm:random-transposition-l2}
    Let $N\ge2$, and let $\mathbf{Q}$ be the Markov kernel on the symmetric group $\mathfrak{S}_{N}$ on $N$ elements that at each step, from a permutation $\varphi$, moves to $\varphi\theta$, where $\theta$ is chosen uniformly among the $\binom{N}{2}$ transpositions. Let $\pi_{\mathfrak{S}}$ denote the uniform probability measure on $\mathfrak{S}_{N}$, and set $\mu_{t}=e^{t(\mathbf{Q}-\mathbf{I})}(\iota,\cdot)$ for each $t\ge0$, where $\iota$ is the identity permutation. Then, if $t\ge N\log N$, it holds
        \begin{align*}
            \Bigl\|\frac{\mu_{t}}{\pi_{\mathfrak{S}}}-\mathbf{1}\Bigr\|_{\pi_{\mathfrak{S}}}^{2} &\le1.
        \end{align*}
\end{theorem}

For further background, we refer the reader to, e.g.,~\cite[Prop.~4.6]{saloffcoste2008refined}.

\begin{proof}[Proof of~\cref{prop:two-edge-entropy-bound}]
    We use the pairing construction from~\cref{sec:deep-link-expansion} with $\sigma=\varnothing$ and $k=m$. As in that setting, write $\pi_{\mathcal{P}}$ and $\pi_{\mathcal{Q}}$ for the uniform probability measures on $\mathcal{P}$ and $\mathcal{Q}$, respectively. Fix $X\in\mathcal{Q}$.
    
    Define the kernel $\widehat{\downup}$ on $\mathcal{Q}$ as follows. From $X\in\mathcal{Q}$, choose $W\in\binom{X}{2}$ uniformly and replace its two pairs by a uniformly chosen pairing of their four stubs for which the resulting pairing belongs to $\mathcal{Q}$. For each such $W$, let $\mathcal{M}(W)$ be the set of all three pairings of these four stubs, and write
        \begin{align*}
            g_{X}(W) &=|\{W'\in\mathcal{M}(W):(X\setminus W)\cup W'\in\mathcal{Q}\}|.
        \end{align*}
    Since $X$ was valid to begin with, $1\le g_{X}(W)\le3$. The two selected stub pairs correspond bijectively to two selected graph edges. If those edges are vertex-disjoint, valid stub completions correspond bijectively to graph completions, each with probability $1/g_{X}(W)$. If the selected pairs are $\{\vec{v},\vec{z}\},\{\vec{u},\vec{w}\}$ with $v=u$, the matching $\{\{\vec{u},\vec{z}\},\{\vec{v},\vec{w}\}\}$ gives the same graph as the original matching, while the third matching contains the pair $\{\vec{v},\vec{u}\}$, which projects to a loop. Thus $g_{X}(W)=2$, and the probability of restoring the graph is $1/2+1/2=1$. Averaging over the selected pairs gives, for every $\tau'\in\mathcal{R}_{\mathbf{d}}(m)$,
        \begin{align*}
            &\sum_{\substack{Y\in\mathcal{Q}\\\Phi(Y)=\tau'}}\widehat{\downup}(X,Y)\\
            &\quad=\frac{1}{\binom{m}{2}}\sum_{W\in\binom{X}{2}}\frac{|\{W'\in\mathcal{M}(W):(X\setminus W)\cup W'\in\mathcal{Q},\ \Phi((X\setminus W)\cup W')=\tau'\}|}{g_{X}(W)}\\
            &\quad=\downup_{2}(\Phi(X),\tau').
        \end{align*}
    The same projection identity holds for powers of the kernels, by induction. Let $s\ge 1$ be fixed. Then 
        \begin{align}\label{eq:pairing-projection-recursion}
            \sum_{\substack{Y\in\mathcal{Q}\\\Phi(Y)=\tau'}}\widehat{\downup}^{\,s+1}(X,Y) &=\sum_{Z\in\mathcal{Q}}\widehat{\downup}^{\,s}(X,Z)\downup_{2}(\Phi(Z),\tau')\notag\\
            &=\sum_{\tau''\in\mathcal{R}_{\mathbf{d}}(m)}\downup_{2}(\tau'',\tau') \sum_{\substack{Z\in\mathcal{Q} \\ \Phi(Z)=\tau''}}  \widehat{\downup}^{\,s}(X,Z)\notag\\
            &=\sum_{\tau''\in\mathcal{R}_{\mathbf{d}}(m)}\downup_{2}^{s}(\Phi(X),\tau'')\downup_{2}(\tau'',\tau')\notag\\
            &=\downup_{2}^{s+1}(\Phi(X),\tau'),
        \end{align}
    where in the second equality we group the terms according to $\tau''=\Phi(Z)$.
    
    Now let $r\ge 1$ be an integer. Using~\cref{lem:uniform-projection-lift}\emph{(i)} and the identity~\cref{eq:pairing-projection-recursion}, the log-sum inequality gives
        \begin{align*}
            &D(\downup_{2}^{r}(\Phi(X),\cdot)\,\Vert\,\pi_{m})\\
            &\quad=\sum_{\tau'\in\mathcal{R}_{\mathbf{d}}(m)}\downup_{2}^{r}(\Phi(X),\tau')\log\frac{\downup_{2}^{r}(\Phi(X),\tau')}{\pi_{m}(\tau')}\\
            &\quad=\sum_{\tau'\in\mathcal{R}_{\mathbf{d}}(m)}\Bigl(\sum_{\substack{Y\in\mathcal{Q}\\\Phi(Y)=\tau'}}\widehat{\downup}^{\,r}(X,Y)\Bigr)\log\frac{\displaystyle\sum_{\substack{Y\in\mathcal{Q}\\\Phi(Y)=\tau'}}\widehat{\downup}^{\,r}(X,Y)}{\displaystyle\sum_{\substack{Y\in\mathcal{Q}\\\Phi(Y)=\tau'}}{\pi_{\mathcal{Q}}}(Y)}\\
            &\quad\le\sum_{\tau'\in\mathcal{R}_{\mathbf{d}}(m)}\sum_{\substack{Y\in\mathcal{Q}\\\Phi(Y)=\tau'}}\widehat{\downup}^{\,r}(X,Y)\log\frac{\widehat{\downup}^{\,r}(X,Y)}{\pi_{\mathcal{Q}}(Y)}\\
            &\quad=\sum_{Y\in\mathcal{Q}}\widehat{\downup}^{\,r}(X,Y)\log\frac{\widehat{\downup}^{\,r}(X,Y)}{\pi_{\mathcal{Q}}(Y)}\\
            &\quad=\sum_{Y\in\mathcal{Q}}\widehat{\downup}^{\,r}(X,Y)\log\frac{\widehat{\downup}^{\,r}(X,Y)}{\pi_{\mathcal{P}}(Y)}+\log\frac{|\mathcal{Q}|}{|\mathcal{P}|}\\
            &\quad\le D(\widehat{\downup}^{\,r}(X,\cdot)\,\Vert\,\pi_{\mathcal{P}}).
        \end{align*}
    Next, define the kernel $\mathbf{K}$ on $\mathcal{P}$ by uniformly choosing two distinct pairs and resampling the pairing of their four stubs. Each of the three possible pairings is chosen with probability $1/3$, and all other pairs remain unchanged. Extend $\widehat{\downup}$ to a Markov kernel on $\mathcal{P}$ by setting $\widehat{\downup}(Z,\cdot)=\mathbf{K}(Z,\cdot)$ for $Z\notin\mathcal{Q}$. This leaves its laws started in $\mathcal{Q}$ unchanged. Every transition with positive probability under $\widehat{\downup}$ also has positive probability under $\mathbf{K}$, and the same holds for their powers. Hence~\cref{lem:weak-triangle-inequality} gives
        \begin{align}\label{eq:relative-entropy-weak-triangle}
            D(\widehat{\downup}^{\,r}(X,\cdot)\,\Vert\,\pi_{\mathcal{P}}) &\le2D(\widehat{\downup}^{\,r}(X,\cdot)\,\Vert\,\mathbf{K}^{r}(X,\cdot))\notag\\
            &\qquad+\log\Bigl(1+\Bigl\|\frac{\mathbf{K}^{r}(X,\cdot)}{\pi_{\mathcal{P}}}-\mathbf{1}\Bigr\|_{\pi_{\mathcal{P}}}^{2}\Bigr).
        \end{align}
    We bound the two terms on the right-hand side separately.

    For the first term in~\cref{eq:relative-entropy-weak-triangle}, recall from the proof of~\cref{lem:invalid-transpositions}, with $\sigma=\varnothing$ and $k=m$, that the number of invalid transpositions from $X$ is at most $8m\Delta^{2}$. Each nonidentity replacement of two pairs is obtained by exactly two transpositions: for example, replacing $\{\vec{v},\vec{z}\},\{\vec{u},\vec{w}\}$ by $\{\vec{u},\vec{z}\},\{\vec{v},\vec{w}\}$ is obtained by $\theta_{\vec{v},\vec{u}}$ or $\theta_{\vec{z},\vec{w}}$. Each such replacement has probability $1/(3\binom{m}{2})$ under $\mathbf{K}$, so
        \begin{align}\label{eq:invalid-transitions-bound}
            \sum_{Y\in\mathcal{P}\setminus\mathcal{Q}}\mathbf{K}(X,Y) &\le\frac{8m\Delta^{2}}{2}\frac{1}{3\binom{m}{2}}=\frac{8\Delta^{2}}{3(m-1)}\le\frac{8\Delta^{2}}{m}.
        \end{align}
    Conditional on selecting $W\in\binom{X}{2}$ for replacement, each of the three pairings of its four stubs is chosen with probability $1/3$ under $\mathbf{K}$, and each valid completion is selected with probability $1/g_{X}(W)$ under $\widehat{\downup}$. Applying the log-sum inequality when grouping choices $(W,W')$ that give the same pairing yields
        \begin{align*}
            &D(\widehat{\downup}(X,\cdot)\,\Vert\,\mathbf{K}(X,\cdot))\\
            &\quad=\sum_{Y\in\mathcal{Q}}\widehat{\downup}(X,Y)\log\frac{\widehat{\downup}(X,Y)}{\mathbf{K}(X,Y)}\\
            &\quad=\sum_{Y\in\mathcal{Q}} \Bigl(\sum_{\substack{(W,W')\\(X\setminus W)\cup W'=Y}}\frac{1}{\binom{m}{2}\,g_{X}(W)}\Bigr) \log\frac{\displaystyle\sum_{\substack{(W,W')\\(X\setminus W)\cup W'=Y}}\frac{1}{\binom{m}{2}\,g_{X}(W)}}{\displaystyle\sum_{\substack{(W,W')\\(X\setminus W)\cup W'=Y}}\frac{1}{3\binom{m}{2}}}\\
            &\quad\le \frac{1}{\binom{m}{2}}\sum_{W\in\binom{X}{2}}\sum_{\substack{W'\in\mathcal{M}(W)\\(X\setminus W)\cup W'\in\mathcal{Q}}} \frac{1}{g_{X}(W)}\log\frac{3}{g_{X}(W)}\\
            &\quad=\frac{1}{\binom{m}{2}}\sum_{W\in\binom{X}{2}}\log\frac{3}{g_{X}(W)}.
        \end{align*}
    Using $\log(3/u)\le3-u$ for $1\le u\le3$, followed by the bound in~\cref{eq:invalid-transitions-bound}, we obtain
        \begin{align}\label{eq:heatbath-local-kl}
            D(\widehat{\downup}(X,\cdot)\,\Vert\,\mathbf{K}(X,\cdot)) &\le\frac{1}{\binom{m}{2}}\sum_{W\in\binom{X}{2}}(3-g_{X}(W)) =3\sum_{Y\in\mathcal{P}\setminus\mathcal{Q}}\mathbf{K}(X,Y)\le\frac{24\Delta^{2}}{m}.
        \end{align}
    By~\cref{lem:entropy-chain-rule}, for every integer $s\ge1$,
        \begin{align*}
            D(\widehat{\downup}^{\,s}(X,\cdot)\,\Vert\,\mathbf{K}^{s}(X,\cdot)) &\le D(\widehat{\downup}^{\,s-1}(X,\cdot)\,\Vert\,\mathbf{K}^{s-1}(X,\cdot))+\frac{24\Delta^{2}}{m},
        \end{align*}
    where we use~\cref{eq:heatbath-local-kl}. Applying this inequality repeatedly gives
        \begin{align}\label{eq:pairing-path-kl}
            D(\widehat{\downup}^{\,r}(X,\cdot)\,\Vert\,\mathbf{K}^{r}(X,\cdot)) &\le r\frac{24\Delta^{2}}{m},
        \end{align}
    since both initial distributions assign mass one to $X$.
    
    For the second term in~\cref{eq:relative-entropy-weak-triangle}, fix a labelling of the set $\mathcal{S}$ of all $2m$ stubs, and let $\widetilde{\mathcal{P}}$ denote the set of orderings of $\mathcal{S}$. Define the map $\Psi:\widetilde{\mathcal{P}}\to\mathcal{P}$ according to
        \begin{align*}
            (\vec{v}_{1},\dotsc,\vec{v}_{2m}) &\mapsto\{\{\vec{v}_{1},\vec{v}_{2}\},\dotsc,\{\vec{v}_{2m-1},\vec{v}_{2m}\}\}
        \end{align*}
    for each $(\vec{v}_{1},\dotsc,\vec{v}_{2m})\in\widetilde{\mathcal{P}}$. $\Psi$ identifies each ordering of the stubs with a pairing of the stubs by forgetting the order of the pairs and the order of the stubs within each pair. Each fiber $\Psi^{-1}(Z)$ of this map contains $(m!)2^{m}$ ordered presentations of $Z$, independent of $Z\in\mathcal{P}$. Set $p=\frac{2m-1}{3(m-1)}\in[2/3,1]$. Then $\frac{2p}{\binom{2m}{2}}=\frac{1}{3\binom{m}{2}}$, which is the probability assigned by $\mathbf{K}$ to each nontrivial transition. Define the kernel $\widetilde{\mathbf{K}}$ on $\widetilde{\mathcal{P}}$ by staying at the current ordering with probability $1-p$, and otherwise exchanging two uniformly chosen distinct entries.

    Fix an ordering $\widetilde{X}=(\vec{v}_{1},\dotsc,\vec{v}_{2m})$ with $\Psi(\widetilde{X})=X$, and let $t\ge3m\log(2m)$. For each ordering $\widetilde{Y}\in\widetilde{\mathcal{P}}$, let $\varphi_{\widetilde{Y}}\in\mathfrak{S}_{2m}$ be the unique permutation satisfying
        \begin{align*}
            \widetilde{Y} &=(\vec{v}_{\varphi_{\widetilde{Y}}(1)},\dotsc,\vec{v}_{\varphi_{\widetilde{Y}}(2m)}).
        \end{align*}
    This gives a bijection between $\widetilde{\mathcal{P}}$ and $\mathfrak{S}_{2m}$ under which the initial ordering $\widetilde{X}$ corresponds to the identity permutation $\iota$.

    Let $\mathbf{Q}$ denote the random transposition kernel from~\cref{thm:random-transposition-l2}, with $N=2m$. Under the identification $\widetilde{Y}\mapsto\varphi_{\widetilde{Y}}$, exchanging entries $i$ and $j$ of an ordering corresponds to replacing $\varphi_{\widetilde{Y}}$ by $\varphi_{\widetilde{Y}}\circ(i\,j)$. Choosing the two entries uniformly therefore gives one step of $\mathbf{Q}$. Thus, after identifying the rows and columns by this bijection, $\widetilde{\mathbf{K}}=(1-p)\mathbf{I}+p\mathbf{Q}$. Since $\varphi_{\widetilde{X}}=\iota$, it follows that for each $\widetilde{Y}\in\widetilde{\mathcal{P}}$,
        \begin{align*}
            e^{t(\widetilde{\mathbf{K}}-\mathbf{I})}(\widetilde{X},\widetilde{Y}) &=e^{t((1-p)\mathbf{I}+p\mathbf{Q}-\mathbf{I})}(\iota,\varphi_{\widetilde{Y}}) =e^{pt(\mathbf{Q}-\mathbf{I})}(\iota,\varphi_{\widetilde{Y}}) =\mu_{pt}(\varphi_{\widetilde{Y}}),
        \end{align*}
    where in the last equality we recall the definition of $\mu_{t}$ in~\cref{thm:random-transposition-l2}. Since $p\ge2/3$ and $t\ge3m\log(2m)$, we have $pt\ge2m\log(2m)$. Thus, by~\cref{thm:random-transposition-l2},
        \begin{align*}
            \Bigl\|\frac{e^{t(\widetilde{\mathbf{K}}-\mathbf{I})}(\widetilde{X},\cdot)}{\pi_{\mathfrak{S}}}-\mathbf{1}\Bigr\|_{\pi_{\mathfrak{S}}}^{2} &\le1.
        \end{align*}
    
    We now apply this estimate to the walk $\mathbf{K}$ on $\mathcal{P}$. For every $\widetilde{Z}\in\widetilde{\mathcal{P}}$ and $Y\in\mathcal{P}$, we claim that
        \begin{align}\label{eq:pairing-projection}
            \sum_{\substack{\widetilde{Y}\in\widetilde{\mathcal{P}}\\\Psi(\widetilde{Y})=Y}}\widetilde{\mathbf{K}}(\widetilde{Z},\widetilde{Y}) &=\mathbf{K}(\Psi(\widetilde{Z}),Y).
        \end{align}
    Fix $\widetilde{Z}$ and write $Z=\Psi(\widetilde{Z})$. Suppose first that $Y\ne Z$ is obtained by replacing two pairs of $Z$. The two exchanges producing this replacement give two distinct orderings representing $Y$, each reached with probability $p/\binom{2m}{2}$. Thus, using $p=(2m-1)/(3(m-1))$, we have
        \begin{align*}
            \sum_{\substack{\widetilde{Y}\in\widetilde{\mathcal{P}}\\\Psi(\widetilde{Y})=Y}}\widetilde{\mathbf{K}}(\widetilde{Z},\widetilde{Y}) &=\frac{p}{\binom{2m}{2}}+\frac{p}{\binom{2m}{2}} =\frac{1}{3\binom{m}{2}}=\mathbf{K}(Z,Y).
        \end{align*}
    For $Y=Z$, the sum includes the original ordering, reached with probability $1-p$, and the $m$ orderings obtained by exchanging the two stubs within one pair, each reached with probability $p/\binom{2m}{2}$. Therefore
        \begin{align*}
            \sum_{\substack{\widetilde{Y}\in\widetilde{\mathcal{P}}\\\Psi(\widetilde{Y})=Z}}\widetilde{\mathbf{K}}(\widetilde{Z},\widetilde{Y}) &=1-p+\frac{pm}{\binom{2m}{2}} =\frac{1}{3}=\mathbf{K}(Z,Z).
        \end{align*}
    If $Y\ne Z$ cannot be obtained by replacing two pairs, both sides of the claimed identity are zero. For $X$ and its ordering $\widetilde{X}$ fixed earlier, the same induction as in~\cref{eq:pairing-projection-recursion} gives, for every integer $s\ge0$,
        \begin{align*}
            \mathbf{K}^{s}(X,Y) &=\sum_{\substack{\widetilde{Y}\in\widetilde{\mathcal{P}}\\\Psi(\widetilde{Y})=Y}}\widetilde{\mathbf{K}}^{\,s}(\widetilde{X},\widetilde{Y}).
        \end{align*}
    Therefore it follows from~\cref{eq:pairing-projection} that
        \begin{align*}
            e^{t(\mathbf{K}-\mathbf{I})}(X,Y) &=e^{-t}\sum_{s=0}^{\infty}\frac{t^{s}}{s!}\mathbf{K}^{s}(X,Y) =e^{-t}\sum_{s=0}^{\infty}\frac{t^{s}}{s!}\sum_{\substack{\widetilde{Y}\in\widetilde{\mathcal{P}}\\\Psi(\widetilde{Y})=Y}}\widetilde{\mathbf{K}}^{\,s}(\widetilde{X},\widetilde{Y}) =\sum_{\substack{\widetilde{Y}\in\widetilde{\mathcal{P}}\\\Psi(\widetilde{Y})=Y}}e^{t(\widetilde{\mathbf{K}}-\mathbf{I})}(\widetilde{X},\widetilde{Y}).
        \end{align*}
    Since each pairing has $2^{m} m!$ ordered presentations, we have
        \begin{align*}
            |\widetilde{\mathcal{P}}| &=(2m)!=2^{m} m!\,|\mathcal{P}|,
        \end{align*}
    and, for each $\widetilde{Y}\in\widetilde{\mathcal{P}}$ with $\Psi(\widetilde{Y})=Y$, we have
        \begin{align*}
            \frac{\pi_{\mathcal{P}}(Y)}{2^{m} m!}&=\frac{1}{(2m)!}=\pi_{\mathfrak{S}}(\widetilde{Y}).
        \end{align*}
    Dividing the preceding transition identity by $\pi_{\mathcal{P}}(Y)$ therefore gives
        \begin{align*}
            \frac{e^{t(\mathbf{K}-\mathbf{I})}(X,Y)}{\pi_{\mathcal{P}}(Y)}-1 &=\frac{1}{2^{m} m!}\sum_{\substack{\widetilde{Y}\in\widetilde{\mathcal{P}}\\\Psi(\widetilde{Y})=Y}}\Bigl(\frac{e^{t(\widetilde{\mathbf{K}}-\mathbf{I})}(\widetilde{X},\widetilde{Y})}{\pi_{\mathfrak{S}}(\widetilde{Y})}-1\Bigr).
        \end{align*}
    Applying Jensen's inequality to the preceding average, we obtain
        \begin{align*}
            \Bigl\|\frac{e^{t(\mathbf{K}-\mathbf{I})}(X,\cdot)}{\pi_{\mathcal{P}}}-\mathbf{1}\Bigr\|_{\pi_{\mathcal{P}}}^{2} &\le\sum_{Y\in\mathcal{P}}\frac{\pi_{\mathcal{P}}(Y)}{2^{m} m!}\sum_{\substack{\widetilde{Y}\in\widetilde{\mathcal{P}}\\\Psi(\widetilde{Y})=Y}}\Bigl(\frac{e^{t(\widetilde{\mathbf{K}}-\mathbf{I})}(\widetilde{X},\widetilde{Y})}{\pi_{\mathfrak{S}}(\widetilde{Y})}-1\Bigr)^{2}\\
            &=\sum_{\widetilde{Y}\in\widetilde{\mathcal{P}}}{\pi_{\mathfrak{S}}}(\widetilde{Y})\Bigl(\frac{e^{t(\widetilde{\mathbf{K}}-\mathbf{I})}(\widetilde{X},\widetilde{Y})}{\pi_{\mathfrak{S}}(\widetilde{Y})}-1\Bigr)^{2}\\
            &=\Bigl\|\frac{e^{t(\widetilde{\mathbf{K}}-\mathbf{I})}(\widetilde{X},\cdot)}{\pi_{\mathfrak{S}}}-\mathbf{1}\Bigr\|_{\pi_{\mathfrak{S}}}^{2}\le1,
        \end{align*}
    using the choice $t\ge3m\log(2m)$. Here, the first equality uses $\pi_{\mathcal{P}}(Y)/(2^{m} m!)=\pi_{\mathfrak{S}}(\widetilde{Y})$ and the fact that each ordering projects to exactly one pairing.

    It remains to pass from continuous time back to discrete time. The kernel $\mathbf{K}$ is symmetric, since each nonidentity transition has probability $1/(3\binom{m}{2})$. For $f\in\ell^{2}(\mathcal{P},\pi_{\mathcal{P}})$, let $W$ range over partial pairings with $m-2$ pairs. Each $W$ has three completions, and each pairing contains $\binom{m}{2}$ such $W$. Therefore
        \begin{align*}
            \langle f,\mathbf{K}f\rangle_{\pi_{\mathcal{P}}} &=\frac{1}{3|\mathcal{P}|\binom{m}{2}}\sum_{W}\sum_{\substack{Y,Z\in\mathcal{P}\\W\subseteq Y\cap Z}}f(Y)f(Z) = \frac{1}{3|\mathcal{P}|\binom{m}{2}}\sum_{W}\Bigl(\sum_{\substack{Y\in\mathcal{P}\\W\subseteq Y}}f(Y)\Bigr)^{2}\ge0.
        \end{align*}
    Thus $\mathbf{K}$ has eigenvalues in $[0,1]$. Let $\phi_{1},\dotsc,\phi_{|\mathcal{P}|}$ be an orthonormal basis of right eigenfunctions in $\ell^{2}(\mathcal{P},\pi_{\mathcal{P}})$, with $\phi_{1}=\mathbf{1}$ and corresponding eigenvalues $\lambda_{j}(\mathbf{K})$. For every integer $r\ge3m\log(2m)$, using $x^{2r}\le e^{-2r(1-x)}$ for $0\le x\le 1$ and evaluating the preceding heat-kernel bound at $t=r$, we have
        \begin{align*}
            \Bigl\|\frac{\mathbf{K}^{r}(X,\cdot)}{\pi_{\mathcal{P}}}-\mathbf{1}\Bigr\|_{\pi_{\mathcal{P}}}^{2} &=\sum_{j=2}^{|\mathcal{P}|}\lambda_{j}(\mathbf{K})^{2r}\phi_{j}(X)^{2}\\
            &\le\sum_{j=2}^{|\mathcal{P}|}e^{-2r(1-\lambda_{j}(\mathbf{K}))}\phi_{j}(X)^{2}\\
            &=\Bigl\|\frac{e^{r(\mathbf{K}-\mathbf{I})}(X,\cdot)}{\pi_{\mathcal{P}}}-\mathbf{1}\Bigr\|_{\pi_{\mathcal{P}}}^{2}\le1.
        \end{align*}
    Thus the second term in~\cref{eq:relative-entropy-weak-triangle} satisfies
        \begin{align*}
            \log\Bigl(1+\Bigl\|\frac{\mathbf{K}^{r}(X,\cdot)}{\pi_{\mathcal{P}}}-\mathbf{1}\Bigr\|_{\pi_{\mathcal{P}}}^{2}\Bigr) &\le\log2\le1.
        \end{align*}
    Combining this with~\cref{eq:relative-entropy-weak-triangle,eq:pairing-path-kl} and the entropy projection estimate at the start of the proof gives
        \begin{align*}
            D(\downup_{2}^{r}(\Phi(X),\cdot)\,\Vert\,\pi_{m}) &\le2r\frac{24\Delta^{2}}{m}+1.
        \end{align*}
    The preceding upper bound increases with $r$, so we first bound it on a finite interval and then use entropy monotonicity to cover all later times. For every integer $3m\log(2m)\le r\le4m\log(2m)$, we have
        \begin{align*}
            2r\frac{24\Delta^{2}}{m}+1 &\le192\Delta^{2}\log(2m)+1\le193\Delta^{2}\log(2m).
        \end{align*}
    This interval contains $r_{0}=\lceil3m\log(2m)\rceil$, since $m\ge2$. Finally, stationarity of $\pi_{m}$ and~\cref{lem:entropy-chain-rule}, applied with both kernels equal to $\downup_{2}$, give
        \begin{align*}
            D(\downup_{2}^{s+1}(\tau,\cdot)\,\Vert\,\pi_{m}) &\le D(\downup_{2}^{s}(\tau,\cdot)\,\Vert\,\pi_{m})
        \end{align*}
    for every realization $\tau$ and integer $s\ge0$. Thus the bound at $r_{0}$ holds at every later integer time. Since $X\in\mathcal{Q}$ was arbitrary and every realization has a preimage under $\Phi$, this proves~\cref{prop:two-edge-entropy-bound}.
\end{proof}

We may now prove the mixing time bound stated at the beginning of the section.

\begin{proof}[Proof of~\cref{thm:two-edge-mixing}]
    The walk $\downup_{2}$ is reversible with respect to $\pi_{m}$ and has nonnegative eigenvalues by~\cref{prop:up-down-adjoint}. For sufficiently large $C$, our hypothesis and~\cref{thm:two-edge-eigenvalue} give a positive spectral gap, so the walk is irreducible.

    Choose an integer $r$ satisfying $3m\log(2m)\le r\le4m\log(2m)$. Applying~\cref{lem:salez-entropy-gap} with the entropy bound from~\cref{prop:two-edge-entropy-bound}, we obtain
        \begin{align*}
            t_{\mathrm{mix}}(\downup_{2},1/4) &\le r+\frac{4(193\Delta^{2}\log(2m)+1)}{1-\lambda_{2}(\downup_{2})}\\
            &\le\Bigl(4m+\frac{780\Delta^{2}}{1-\lambda_{2}(\downup_{2})}\Bigr)\log(2m),
        \end{align*}
    where we use $r\le4m\log(2m)$, $\Delta\ge1$ and $\log(2m)\ge\log2$.

    If $m\ge C\Delta^{8}$ for a sufficiently large absolute constant $C$, then~\cref{thm:two-edge-eigenvalue} gives $1-\lambda_{2}(\downup_{2})\ge1/(C'm)$ for an absolute constant $C'>0$. Substituting into the preceding bound yields
        \begin{align*}
            t_{\mathrm{mix}}(\downup_{2},1/4) &\le(4+780C'\Delta^{2})m\log(2m)\le C\Delta^{2}m\log(2m),
        \end{align*}
    after increasing $C$ if necessary. The theorem is proved.
\end{proof}

\Cref{cor:intro-salez-sharp-switch-mixing} now follows from~\cref{thm:two-edge-mixing}.

\section{Comparing simplicial and classical switch chains}\label{sec:classical-switch-chain}

The purpose of this section is to transfer the spectral gap and mixing time bounds for $\downup_{2}$ to the classical switch chain of Cooper, Dyer, and Greenhill~\cite[Fig.~1]{cooper2007sampling}. This process is defined on the same state space $\mathcal{R}_{\mathbf{d}}(m)$ as $\downup_{2}$, but exhibits a different transition rule. Let $\mathbf{S}$ denote the Markov kernel associated to the switch chain introduced in~\cref{sec:introduction}.

Write $b=\binom{m}{2}$ and $a=\binom{m}{2}-\sum_{v\in V}\binom{d_{v}}{2}$. For every realization, $a$ is the number of unordered pairs of vertex-disjoint edges. For distinct realizations $\tau,\tau'$ related by a single switch, it follows from the definition of the switch chain that
    \begin{align*}
        \mathbf{S}(\tau,\tau') &= \frac{1}{6a}.
    \end{align*}
The transition probabilities are symmetric, so $\mathbf{S}$ is reversible with respect to $\pi_{m}$.

\begin{proposition}\label{prop:classical-simplicial-comparison}
    Let $f\in\ell^{2}(\mathcal{R}_{\mathbf{d}}(m),\pi_{m})$ be arbitrary. Then
        \begin{align}\label{eq:classical-simplicial-energy}
            \frac{b}{3a}\mathcal{E}_{\downup_{2}}(f,f) &\le \mathcal{E}_{\mathbf{S}}(f,f) \le \frac{b}{2a}\mathcal{E}_{\downup_{2}}(f,f).
        \end{align}
    Consequently,
        \begin{align}\label{eq:classical-simplicial-eigenvalue}
            \frac{b}{3a}(1-\lambda_{2}(\downup_{2})) &\le 1-\lambda_{2}(\mathbf{S}) \le \frac{b}{2a}(1-\lambda_{2}(\downup_{2})).
        \end{align}
\end{proposition}

\begin{proof}
    Fix distinct realizations $\tau,\tau'$ related by a switch, and let $\sigma=\tau\cap \tau'$. The two deleted edges are determined by $\tau,\tau'$, and $\sigma$ has $g(\sigma)\in\{2,3\}$ valid completions. Thus
        \begin{align*}
            \downup_{2}(\tau,\tau')&=\frac{1}{b\,g(\sigma)}\;, 
        \end{align*}
    and similarly,
        \begin{align*}
            \frac{\mathbf{S}(\tau,\tau')}{\downup_{2}(\tau,\tau')}&=\frac{b\,g(\sigma)}{6a}.
        \end{align*}
    If the two edges selected by $\downup_{2}$ share a vertex, their residual degrees determine their completion uniquely, resulting in a holding step. Using the preceding two identities to bound the transition probabilities in the Dirichlet forms of $\downup_{2}$ and $\mathbf{S}$ proves~\cref{eq:classical-simplicial-energy}. Taking the infimum over functions of positive variance proves~\cref{eq:classical-simplicial-eigenvalue}.
\end{proof}

For sufficiently large $C$, the hypothesis $m\ge C\Delta^{8}$ ensures that $b\ge a\ge b-m(\Delta-1)\ge b/2$. Combining this with~\cref{thm:two-edge-eigenvalue,prop:classical-simplicial-comparison} gives the following corollary.

\begin{corollary}\label{cor:classical-switch-eigenvalue}
    There exists an absolute constant $C>0$ with the following property. If $m\ge C\Delta^{8}$, then
        \begin{align*}
            1-\frac{2}{m} &\le\lambda_{2}(\mathbf{S})\le1-\frac{1}{Cm}.
        \end{align*}
\end{corollary}

The proofs of~\cref{thm:two-edge-mixing} and~\cref{prop:two-edge-entropy-bound}, together with~\cref{cor:classical-switch-eigenvalue}, yield the following bound on the mixing time of the switch chain.

\begin{corollary}\label{cor:classical-switch-mixing}
    There exists an absolute constant $C>0$ with the following property. If $m\ge C\Delta^{8}$, then
        \begin{align}\label{eq:classical-switch-mixing-quantitative}
            t_{\mathrm{mix}}(\mathbf{S},1/4) &\le C\Delta^{2}m\log(m).
        \end{align}
\end{corollary}

\Cref{cor:classical-switch-mixing} can be shown using an argument very similar to that used in the proof of~\cref{prop:two-edge-entropy-bound}, with some replacements. Borrowing the notation appearing in that proof, the basic idea is to replace the kernel $\mathbf{K}$ with the lazy kernel 
    \begin{align*}
        \mathbf{K}_{\alpha} &= (1-\alpha)\mathbf{I} + \alpha\mathbf{K},
    \end{align*}
where $\alpha=b/(2a)\in[1/2,1]$. Similarly, replace $\widehat{\downup}$ with the kernel which rejects proposals of $\mathbf{K}_{\alpha}$ outside the set of valid pairings $\mathcal{Q}$. This kernel projects to $\mathbf{S}$ under $\Phi$, since each switch between distinct realizations has probability $\alpha/(3b)=1/(6a)$. The remaining entropy comparisons and the application of~\cref{lem:salez-entropy-gap} proceed much as in the proof of~\cref{thm:two-edge-mixing}, with only absolute constants changed;~\cref{cor:classical-switch-eigenvalue} then supplies the stated bound on the mixing time of the switch chain.

\Cref{thm:intro-mixing} now follows from~\cref{cor:classical-switch-mixing}.

\section*{Acknowledgements and AI Methodology}

This work was completed while the author was appointed as a research fellow at the Simons Institute for the Theory of Computing, as part of their program \emph{Spectral Theory Beyond Graphs}. The author wishes to thank Lap Chi Lau for helpful conversations at the Simons Institute.

ChatGPT-6 was used throughout the preparation of this manuscript. The author developed the overall proof strategy independently of AI, including the conjectures that led to the deep link expansion theorem (\cref{thm:link-eigenvalue-switch-complex}) and the energy comparison result for higher order switch chains (\cref{prop:linear-comparison}). ChatGPT-6 was used to assist in the development of the proofs of the main steps and their technical lemmas, and later in refining those proofs, as well as for routine copyediting. The author has deliberately made no effort to query ChatGPT-6 or other AI tools for greatly compacted ``black box'' proofs of the main result(s). While simpler proofs of these facts would presumably be of interest to the community, the author places at least as much importance on preserving and disseminating the human-driven ideas that ultimately led to the resolution of~\cref{conj:cdg}.


\end{document}